\documentclass{article}
\usepackage{amsmath, amsfonts, amssymb, mathrsfs, amsthm, accents, sectsty,hyphenat, palatino, calc, url, stmaryrd, mathdots,array}
\usepackage[stable]{footmisc}
\usepackage{url, amsrefs}
\usepackage[all]{xy}
\usepackage{rotating}

\usepackage[T3,T1]{fontenc}
\DeclareSymbolFont{tipa}{T3}{cmr}{m}{n}
\DeclareMathAccent{\invbreve}{\mathalpha}{tipa}{16}

\begin{document}

\newcommand{\zele}[1]{\operatorname{Z}( #1 )}
\newcommand{\supp}{\operatorname{supp}}
\newcommand{\Z}{\mathbb{Z}}
\newcommand{\R}{\mathbb{R}}
\newcommand{\C}{\mathbb{C}}
\newcommand{\Ker}{\operatorname{Ker}}
\newcommand{\Hom}{\operatorname{Hom}}
\newcommand{\Irr}{\operatorname{Irr}}
\newcommand{\GL}{\operatorname{GL}}
\newcommand{\std}{\zeta}
\newcommand{\lstd}{\lambda}
\newcommand{\lrstd}{\gamma}
\newcommand{\rest}{\big|}
\newcommand{\EI}{SI}
\newcommand{\lshft}[1]{\overset{\leftarrow}{#1}}
\newcommand{\rshft}[1]{\overset{\rightarrow}{#1}}
\newcommand{\abs}[1]{\left|{#1}\right|}
\newcommand{\Rep}{\mathfrak{R}}
\newcommand{\Nrd}{\operatorname{Nrd}}
\newcommand{\RI}{LM}
\newcommand{\TU}{totally unlinked}
\newcommand{\ulb}{unlabeled}
\newcommand{\rltn}{\leadsto}
\newcommand{\trns}{traversable}
\newcommand{\linjct}{LS}
\newcommand{\lrs}{LRS}
\newcommand{\speh}{\mathfrak{sp}}
\newcommand{\opp}[1]{#1^\circ}
\newcommand{\zi}{\mathbf{t}}
\newcommand{\combsoclang}{\mathfrak{socl}}
\newcommand{\combsoczele}{\mathfrak{socz}}
\newcommand{\newm}{\mathfrak{N}}
\newcommand{\rprd}{\mathfrak{R}}
\newcommand{\sngls}{\mathcal{S}}
\newcommand{\socc}{\operatorname{cos}}
\newcommand{\rigid}{\operatorname{rigid}}


\newcommand{\lang}[1]{\operatorname{L}( #1 )}
\newcommand{\seg}{\operatorname{Seg}}
\newcommand{\soc}{\operatorname{soc}}
\newcommand{\MS}{\operatorname{Mult}}
\newcommand{\rp}{\mathbf{r}}
\newcommand{\m}{\mathfrak{m}}
\newcommand{\n}{\mathfrak{n}}
\newcommand{\Mat}{\operatorname{Mat}}

\newcommand{\id}{\operatorname{id}}
\newcommand{\maxdel}{\Delta_{\operatorname{max}}}
\newcommand{\mindel}{\Delta_{\operatorname{min}}}

\newcommand{\ordered}{aligned}
\newcommand{\rangee}{ordered}

\newcommand{\II}{\mathcal{I}}
\newcommand{\JJ}{\mathcal{J}}
\newcommand{\tII}{\tilde\II}
\newcommand{\tJJ}{\tilde\JJ}

\newcommand{\G}{G}
\newcommand{\B}{B}
\newcommand{\F}{F}
\newcommand{\D}{D}
\newcommand{\M}{M}
\newcommand{\X}{X}
\newcommand{\NN}{\mathbb{N}}

\newcommand{\Cusp}{\mathscr{C}}
\newcommand{\Rr}{\mathscr{R}}

\newcommand{\lcond}{LI}
\newcommand{\rcond}{RI}
\newcommand{\llcond}{LC}
\newcommand{\rrcond}{RC}
\newcommand{\nlcond}{NC}


\newtheorem{thm}{Theorem}[section]
\newtheorem{lem}[thm]{Lemma}
\newtheorem{prop}[thm]{Proposition}
\newtheorem{hps}[thm]{Hypothesis}
\newtheorem{ass}[thm]{Assumptions}
\newtheorem{cnd}[thm]{Conditions}
\newtheorem{coro}[thm]{Corollary}
\newtheorem{defi}[thm]{Definition}
\newtheorem{fct}[thm]{Fact}
\newtheorem{clm}[thm]{Claim}
\newtheorem{conj}[thm]{Conjecture}
\newtheorem{rem}[thm]{Remark}

    \textwidth=125mm
    \textheight=185mm
    \parindent=8mm
    \evensidemargin=0pt
    \oddsidemargin=0pt
    \frenchspacing

\sectionfont{\center\sc\normalsize}
\subsectionfont{\bf\normalsize}

\numberwithin{equation}{section}

\title{On parabolic induction on inner forms of the general linear group over a non-archimedean local field}

\author{Erez Lapid\footnote{Partially supported by grant \# 711733 from the Minerva Foundation} \and Alberto M\'inguez\footnote{partially supported by the
ANR ArShiFo ANR-BLAN-0114, MTM2010-19298 (FEDER) and P12-FQM-2696}}


\date{\today}

\maketitle

\begin{center}
To Joseph Bernstein, in admiration
\end{center}


\begin{abstract}
We give new criteria for the irreducibility of parabolic induction on the general linear group and its inner forms over
a local non-archimedean field. In particular, we give a necessary and sufficient condition
when the inducing data is of the form $\pi\otimes\sigma$ where $\pi$ is a ladder representation
and $\sigma$ is an arbitrary irreducible representation.
As an application we simplify the proof of the classification of the unitary dual.
\end{abstract}

\setcounter{tocdepth}{1}
\tableofcontents

\section*{Introduction}

One of the cornerstones of the representation theory of reductive groups over non-archimedean local fields is the seminal work of Bernstein and Zelevinsky
in the 1970s \cites{MR0425030, MR0425031, MR0579172}.
Much of their work is concentrated on the general linear group. It culminated in Zelevinsky's classification of the (complex, smooth) irreducible
representations of $\GL_n(F)$ (where $F$ is a non-archimedean local field) in terms of multisegments --
an almost purely combinatorial object \cite{MR584084}.
Roughly speaking, the basic building blocks are supercuspidal representations.
From them one constructs (using parabolic induction) the so-called segment representations.
A general irreducible representation is given as a subrepresentation of the parabolic induction of the tensor product, in a certain order, of segment representations.

Later on this classification was extended to inner forms of the general linear groups, namely to the groups $\GL_n(D)$
where $D$ is a local non-archimedean division algebra. (See Appendix \ref{sec: appendAD} for the history of this problem.)

While the work of Bernstein and Zelevinsky answered some of the most fundamental questions in the representation theory of $\GL_n(F)$,
there are still some outstanding problems. For instance, it is not clear how to combinatorially characterize (in terms of multisegments)
the irreducibility of parabolic induction in general or even whether such a characterization is realistic.
(In principle, it is possible to determine irreducibility of parabolic induction
either by computing a vast number of Kazhdan--Lustzig polynomials \cites{MR617466, MR783619, MR2838836, MR2320806}
or using algorithms related to canonical bases (cf. \cite{MR1959765} and the references therein).
However, neither method is very practical for large $n$.)

In this paper we will study a special case of this problem.
First, we give a complete answer in the case where one of the representations (call it $\pi$) is a segment representation.
(The case where $\pi$ is supercuspidal was considered in \cite{MR2527415} and \cite{MR2306606} independently.)
We will actually give two combinatorial criteria -- one in terms of the Zelevinsky involution (for which a combinatorial recipe was given
by M\oe glin and Waldspurger in \cite{MR863522}) and another which looks superficially simpler.
See Proposition \ref{prop: delsig} and Corollary \ref{cor:irred2} in the body of the paper for the statement of the
first and second combinatorial criteria respectively.
It is straightforward to extend these criteria from segment representations to unramified or generic representations
(see Corollary \ref{coro: gentimesany}).
Remarkably, the first criterion works equally well when $\pi$ is a Speh representation (Corollary \ref{coro:subspeh}).
Moreover, building on the square-integrable case we will also provide (Corollary \ref{cor: laddercombsoc}) a complete (and computationally feasible,
albeit complicated) answer in the case where $\pi$ is a ladder representation
(or more generally, an irreducible representation parabolically induced from ladder representations).
This is a class of representations introduced in \cite{MR3163355} -- see also \cite{1409.1367}.
We also give an alternative, more aesthetic, necessary and sufficient combinatorial condition for irreducibility in this case
(Proposition \ref{prop: laddercond}).

Our method is largely combinatorial using the Bernstein--Zelevinsky geometric lemma as the main tool.
We will also use intertwining operators and the Zelevinsky involution which interplays between the Zelevinsky classification and the (refined)
Langlands classification in the case at hand. The proofs apply equally well for $\GL_n(D)$.
As an application we will proffer some (easily verifiable) sufficient conditions for the irreducibility of parabolic induction.
As a further application we present a simple, essentially combinatorial,
proof of the necessary ingredients in the Tadi\'c program for the classification of the unitary dual of $\GL_n(D)$ \cite{MR1040995}.
This problem emerged after Tadi\'c had completed the classification in the case $D=F$ \cite{MR870688}.
In the general case, the program was completed in \cite{MR2732351} but the proof uses rather heavy machinery.
(See \S\ref{sec: U0} for a more complete history of the problem.)

The contents of the paper are as follows.
After setting up the notation and terminology in \S\ref{sec: notation} we recall Zelevinsky's classification in \S\ref{sec: classification1},
deferring a more elaborate discussion to the appendix.
In \S\ref{sec: leftmultiplier} we discuss the central notion of this paper, namely a `left multiplier' representation (\RI).
This is an irreducible representation $\pi$ of $\GL_n(D)$ such that for any $m\ge1$ and any irreducible representation $\sigma$ of $\GL_m(D)$,
the parabolic induction $\pi\times\sigma$ admits a unique irreducible subrepresentation, which occurs with multiplicity one in the Jordan--H\"older sequence of $\pi\times\sigma$.
We provide examples for such $\pi$'s. We note that an example of Leclerc shows that not every irreducible representation is \RI.

In \S\ref{sec: square-integrable} we study the parabolic induction $\pi\times\sigma$ where $\pi$ is a segment representation
and $\sigma$ is an arbitrary representation.
We provide a combinatorial description of the Zelevinsky parameters of the socle of $\pi\times\sigma$ as well as
a characterization of the irreducibility of $\pi\times\sigma$.
In \S\ref{sec: applications} we give some applications of this characterization. Namely, we provide a sufficient combinatorial criterion for the irreducibility
of $\pi\times\sigma$ (for general irreducible representations $\pi$ and $\sigma$). While this condition is not necessary in general,
we do get a precise criterion for irreducibility in the case where $\pi$ is a Speh representation, or more generally a ladder representation.
In \S\ref{sec: U0} we revisit the classification of the unitary dual of $\GL_n(D)$.
In the case where $D$ is commutative, this is a well known result of Tadi\'c \cite{MR870688}.
Tadi\'c also formulated the necessary ingredients to deal with general $D$ and proved some of them \cite{MR1040995}.
His program was eventually completed by the combined efforts of several mathematicians (culminating in \cite{MR2732351}).
However, so far all proofs ultimately reduced the problem to the case $D=F$, and require ponderous prerequisites.
As a consequence of our results in previous sections, we provide in \S\ref{sec: U0} a simple proof of the classification of the unitary dual
which uses minimum prerequisites and which is uniform in $D$ and the characteristic of $F$.
In the appendix, joint with Marko Tadi\'c we revisit the classification of the admissible dual of $GL_n(F)$ and its inner forms and outline
a simple proof of the classification with minimal sine qua nons.
More precisely, we only use the combinatorics of the Jacquet modules and certain unitarity tricks in the spirit of \cite{CassNotes} to deal with some minimal cases.
(Thus, the proof is not completely algebraic.)
Notably, at some point we use the group $\GL_{n^2}$ to derive results about $\GL_n$.
The proof is very much along the lines of \cite{MR3049700} except that we replace the ingredients
which rely on type theory by a self-contained argument.

We hope that better understanding of the representation theory of the general linear group will shed some light on
the representation theory of classical groups, and in particular on their unitary dual.

\vspace{.5cm}

\noindent\textbf{Acknowledgment}

We are happy to humbly dedicate this paper in honor of Joseph Bernstein on the occasion of his 70th birthday.
Apart from his global mathematical impact, Bernstein has been an inspirational figure and a role model to us, as he has been to a great many mathematicians.
The first-named author is fortunate to have been taught by Bernstein several courses in his early mathematical education,
with a lasting impact.
Beyond that, Bernstein has always been extremely generous with his time and ideas.

Part of this work was done while the authors were visiting NYU Shanghai in spring 2014.
We are both grateful for the generous hospitality of NYU Shanghai.
During the last part of this work, the first-named author was visiting MSRI.
He would like to thank MSRI and the organizers of the program
``New Geometric Methods in Number Theory and Automorphic Forms'' for the invitation.
We would like to thank Joseph Bernstein, Bernard Leclerc and Marko Tadi\'c for helpful discussions and communication.
Especially we would like to thank Marko Tadi\'c for joining us in writing the appendix.
Thanks are also due to Stephen DeBacker, Guy Henniart and Colette M\oe glin for useful correspondence
and to Arnab Mitra, Omer Offen and Eitan Sayag for their comments on the paper.
Finally, we are grateful to the referee for carefully reading the manuscript and for providing an argument
which simplifies some aspects of the original version of the paper.

\section{Preliminaries} \label{sec: notation}

Throughout the paper let $\F$ be a non-archimedean local field with normalized absolute value $\abs{\cdot}$
and let $\D$ be a finite-dimensional central division $\F$-algebra.

For any $m\in\Z_{\ge1}$ let $\G_m = \GL_m(\D)$ and let $\Rep(\G_m)$ be the category of smooth complex representations of $\G_m$ of finite length.
If $\pi$ is an object in $\Rep(\G_m)$ we write $\deg(\pi)=m$.
Denote by $\Irr(G_m)$ the set of equivalence classes of irreducible objects of $\Rep(\G_m)$
and by $\Cusp(G_m)$ the subset consisting of supercuspidal representations.
(By convention we define $\G_0$ as the trivial group and $\Irr(G_0)$ consists of the one-dimensional representation of $\G_0$.)
Let $\Irr$ be the disjoint union of $\Irr(G_m)$, $m \geq 0$. Similarly, let $\Cusp=\cup_{m\geq 1}\Cusp(G_m)$.

Let $\Rr_m$ be the Grothendieck group of $\Rep(\G_m)$ and $\Rr= \underset{m \geq 0}{\oplus} \Rr_m$.
The submonoid of $\Rr$ consisting of finite sums of elements in $\Irr$ (including $0$) is denoted by $\Rr^+$.
The ordering in $\Rr$ is defined by $\pi_1 \leq \pi_2$ if $\pi_2 - \pi_1 \in \Rr^+$.
The canonical map from the objects of $\Rep(\G_m)$ to $\Rr_m$ will be denoted by $\pi\mapsto[\pi]$.

For any $\pi \in \Rep(\G_m)$, we denote by $\pi^\vee$ the contragredient of $\pi$.
(The sets $\Irr$ and $\Cusp$ are invariant under $^\vee$.)
If moreover $\chi$ is a complex-valued character of $F^*$, we write $\pi\chi$ for the twist of $\pi$
by $\chi\circ\Nrd$ where $\Nrd:\G_m\rightarrow F^*$ is the reduced norm.

\subsection{Induction and Jacquet functor}
Let $\alpha=\left( m_1,\dots,m_r\right)$ be a composition of $m$.
(Sometimes, it will be convenient to allow $m_i=0$.)
We denote by $\M_{\alpha}$ the subgroup of $\G_m$ isomorphic to $\G_{m_1}\times\dots\times\G_{m_r}$ consisting of invertible matrices
which are diagonal by blocks of size $m_1,  \dots,m_r$ and by $P_{\alpha}$ the
subgroup of $\G_m$ generated by $\M_\alpha$ and the upper unitriangular matrices.
A standard parabolic subgroup of $\G_m$ is a subgroup of the form $P_{\alpha}$ and its (standard) Levi factor is $\M_{\alpha}$.

Write $\rp_\alpha$ for the normalized Jacquet functor associated to $P_\alpha$.
Often, $\alpha$ will be clear from the context and we will simply write $\rp$.
If $\pi_i\in\Rep(\G_{m_i})$, $i=1,\dots,r$ we write $\pi_1\times\dots\times\pi_r\in\Rep(\G_m)$ for the representation
parabolically induced from $\pi_1\otimes\dots\otimes\pi_r$ via the parabolic subgroup $P_\alpha$. If $\pi \in \Rep(\G_m)$,
we also write $\pi^{\times n}=\pi\times\dots\times\pi$ ($n$ times).

For any $\pi \in \Irr$ there exist $\rho_1, \dots, \rho_r \in \Cusp$, uniquely determined up to permutation,
such that $\pi$ is a subrepresentation of $\rho_1\times\dots\times\rho_r$.
In the sequel we denote by $\supp(\pi)\subset\Cusp$ the set (not the multiset) $\{\rho_1, \dots, \rho_r\}$ and call it the \emph{supercuspidal support} of $\pi$.
More generally, if $\pi\in\Rep(\G_m)$ then $\supp(\pi)$ is by definition the union of the supercuspidal supports of
the composition factors of $\pi$. If $\pi$ is of the form $\pi_1\times\dots\times\pi_r$ with $\pi_1,\dots,\pi_k\in\Irr$
then all composition factors of $\pi$ have the same supercuspidal support.

The map:
\[
([\pi_1],\dots,[\pi_r])\mapsto[\pi_1\times\cdots\times\pi_r]
\]
induces by linearity a $\Z$-multilinear map from $\Rr_{m_1}\times\dots\times\Rr_{m_r}$ to $\Rr_{m}$.
This endows $\Rr$ with the structure of an associative, commutative, $\Z_{\ge0}$-graded ring.
(For the commutativity see \cite{MR874050}*{\S 5.4} which is valid in any characteristic.)

\subsection{Geometric Lemma}
\label{lemmegeo}
Next we recall the combinatorial version of the Geometric Lemma of Bernstein--Zelevinsky \cite{MR0579172}
(or more precisely, its consequence after semisimplification).
Let $\alpha=(m_1,\dots,m_r)$ and $\beta=(n_1,\dots,n_s)$ be two compositions of an integer $m\ge1$.
For every $i\in\{1,\ldots,r\}$, let $\pi_i\in\Rep(\G_{m_i})$.
Denote by $\Mat^{\alpha,\beta}$ the (finite) set of $r\times s$ matrices $\B=(b_{i,j})$ with non-negative integer entries such that:
\begin{equation*}
\label{2}
\sum_{j=1}^s b_{i,j}=m_i,
\quad i\in\{1,\ldots,r\}, \quad \sum_{i=1}^rb_{i,j}=n_j, \quad j\in\{1,\ldots,s\}.
\end{equation*}
Fix $\B\in\Mat^{\alpha,\beta}$.
For any $i\in\{1,\ldots,r\}$, $\alpha_i=(b_{i,1},\dots,b_{i,s})$ is a composition of $m_i$ and we write
the composition factors of $\rp_{\alpha_i}(\pi_i)$ as
\begin{equation*}
\sigma^{(k)}_{i}=\sigma^{(k)}_{i,1}\otimes\cdots\otimes\sigma^{(k)}_{i,s},
\quad
\sigma^{(k)}_{i,j}\in\Irr(G_{b_{i,j}}),
\quad k\in\{1,\ldots,l_i\},
\end{equation*}
where $l_i$ is the length of $\rp_{\alpha_i}(\pi_i)$.
For any $j\in\{1,\ldots,s\}$ and a sequence $\underline{k}=(k_1,\ldots,k_r)$ of integers such that $1\le k_i\le l_i$, define
\begin{equation*}
\Sigma_j^{B,\underline{k}}=\sigma^{(k_1)}_{1,j}\times\cdots\times\sigma^{(k_r)}_{r,j}\in\Rep(G_{n_j}).
\end{equation*}
Then we have:
\begin{equation*}
\label{Druon}
[\rp_{\beta}(\pi_1\times\dots\times\pi_r)]=\underset{\B\in \Mat^{\alpha,\beta},\, \underline{k}}{\sum}
[\Sigma_1^{B,\underline{k}}\otimes\cdots\otimes\Sigma_s^{B,\underline{k}}].
\end{equation*}

Let $\rho\in\Cusp$ and $\pi\in\Rep(G_n)$. We say that $\pi$ is \emph{left $\rho$-reduced} if there does not exist $\pi'\in\Irr$
such that $\rho\otimes\pi'\le [\rp(\pi)]$. If $\pi \in \Irr$, this is equivalent to not having $\pi'\in\Irr$ such that $\pi\hookrightarrow\rho\times\pi'$.

If $I\subset\Cusp$ then we say that $\pi$ is left $I$-reduced if $\pi$ is left $\rho$-reduced for all $\rho\in I$.

\begin{lem}\label{lem: decomI}
For every $\pi\in\Irr$ and $I\subset\Cusp$ there exist unique $\pi_I,\pi_{\lnot I}\in\Irr$ such that
\begin{enumerate}
\item $\pi\hookrightarrow\pi_I\times\pi_{\lnot I}$.
\item $\supp(\pi_I)\subset I$.
\item $\pi_{\lnot I}$ is left $I$-reduced.\footnote{We caution
that $\pi_{\lnot I}$ is not $\pi_J$ where $J$ is the complement of $I$.}
\end{enumerate}
Moreover, $\pi$ is the unique irreducible subrepresentation of $\pi_I\times\pi_{\lnot I}$.
\end{lem}

\begin{proof}
This is proved in \cite{MR2306606}*{Lemma 2.1.2}. Although, the Lemma is stated in \cite{MR2306606} only for $D=F$ and representations
in the line of a cuspidal representation $\rho$ (see \S\ref{par: cuspidal_lines}), the proof therein is valid in our setting.
\end{proof}

Note that in the case where $I=\{\rho\}$ is a singleton, $\pi_I$ is of the form $\rho^{\times a}$ for some $a\in\Z_{\ge0}$.

We single out the following consequence of the geometric lemma:
\begin{lem} \label{lem: mult1 jf}
\begin{enumerate}
\item Let $I\subset\Cusp$ and suppose that $\pi_i\in\Rep(G_{m_i})$, $i=1,\dots,r$ are left $I$-reduced.
Then $\pi_1\times\dots\times\pi_r$ is left $I$-reduced.
\item Let $\pi_i\in\Rep(\G_{m_i})$ and $\sigma_i\in\Irr\G_{m_i}$, $i=1,2$.
Assume that $\sigma_i$ occurs with multiplicity one in $[\pi_i]$, $i=1,2$ and that $\pi_2$ is left $\supp(\pi_1)$-reduced.
Then $\sigma_1\otimes\sigma_2$ occurs with multiplicity one in $\rp(\pi_1\times\pi_2)$.
\end{enumerate}
\end{lem}

\subsection{Unique irreducible subrepresentations in induced representations}

Let $m \geq 1$ and let $\pi \in \Rep(\G_m)$.
Recall that the \emph{socle} (resp, \emph{cosocle}) of $\pi$, denoted by $\soc(\pi)$ (resp., $\socc(\pi)$), is the largest semi-simple subrepresentation (resp., quotient) of $\pi$.
Thus, $\socc(\pi)^\vee=\soc(\pi^\vee)$.

\begin{defi}\label{def:defsnew}
We say that $\pi\in\Rep(\G_n)$ is \emph{socle irreducible} (\EI) if $\soc(\pi)$ is irreducible and occurs with multiplicity one in $[\pi]$.
\end{defi}

\EI\ representations play an important role in the representation theory of $p$-adic groups and in particular
in the classification of irreducible representations of $G_n$, $n \geq 1$.
We will recall that below.

The following straightforward remark will be used repeatedly.
\begin{rem}\label{rem:main}
Let $\pi$ be a subrepresentation of $\pi'$. If $\pi'$ is \EI\ then so is $\pi$ and $\soc(\pi')=\soc(\pi)$.
\end{rem}

We could of course define ``cosocle irreducible'' in an analogous way. Note that $\pi$ is cosocle irreducible if and only if
$\pi^\vee$ is \EI.

The following standard lemma will be useful.
\begin{lem}\label{lem:main}
Assume that $\pi_i\in\Rep(\G_{n_i})$, $i=1,2$ are \EI.
Let $\sigma_i=\soc(\pi_i)\in\Irr\G_{n_i}$, $i=1,2$ and suppose that $\sigma_1 \otimes \sigma_2$ appears with multiplicity $1$ in
$$[\rp(\pi_1 \times \pi_2)].$$
(For instance, this holds if $\pi_2$ is left $\supp(\pi_1)$-reduced, by Lemma \ref{lem: mult1 jf}.)
Then $\pi_1 \times \pi_2$ (and hence also $\sigma_1\times\sigma_2$) is \EI \ and $\soc(\pi_1\times\pi_2)=\soc(\sigma_1\times\sigma_2)$.
\end{lem}

\begin{proof}
Note that if $\tau$ is a subrepresentation of $\pi_1\times\pi_2$ then by Frobenius reciprocity there is a non-trivial map
from $\rp(\tau)$ to $\pi_1\otimes\pi_2$. Thus, $\sigma_1\otimes\sigma_2$ occurs as a subquotient of $\rp(\tau)$.
Thus, if $\tau, \tau'$ are distinct (not necessarily inequivalent) irreducible subrepresentations of $\pi_1 \times \pi_2$ then
\[
2[\sigma_1 \otimes \sigma_2 ]\le [\rp(\tau+ \tau') ]\le [\rp(\pi_1 \times \pi_2)],
\]
a contradiction. It follows that $\soc(\pi_1\times\pi_2)$ is irreducible.
Similarly, we cannot have $2[\soc(\pi_1\times\pi_2)]\le[\pi_1\times\pi_2]$ for otherwise, applying $\rp_{(n_1,n_2)}$ we would get
$2[\sigma_1\otimes\sigma_2]\le[\rp(\pi_1 \times \pi_2)]$ which once again contradicts the assumption of the lemma.
The last part follows from Remark \ref{rem:main}.
\end{proof}

In particular, by induction on $k$ (using Lemmas \ref{lem: mult1 jf} and \ref{lem:main}) if $\rho_1,\dots,\rho_k\in\Cusp$ are distinct then
$\rho_1\times\dots\times\rho_k$ is \EI. (Of course, this case is particularly simple since the Jacquet modules are multiplicity free.)

\section{Classification} \label{sec: classification1}

In this section we review the classification of $\Irr$ in terms of multisegments and set some more notation.
All the results in this section are well known.
In Appendix \ref{sec: appendAD} we will recall the history of the classification and outline a proof along the lines of \cite{MR3049700}
which uses minimum prerequisites.

\subsection{Segments} \label{sec: segment}
Let $m\in\Z_{\ge1}$ and $\rho\in\Cusp(G_m)$.
To $\rho$ one can associate (see Theorem \ref{thm:Olshanski}) an unramified character $\nu_\rho$ of $F^*$
of the form $\nu_\rho=\abs{\cdot}^{s_\rho}$ where $s_\rho\in\R_{>0}$
with the property that for any $\rho'\in\Cusp(G_{m'})$, $m'\in\Z_{\ge1}$ the induced representation $\rho \times \rho'$ is
reducible if and only if $m'=m$ and either $\rho'\simeq\rho\nu_\rho$ or $\rho\simeq\rho'\nu_\rho$.
We write $\rshft{\rho}=\rho\nu_\rho$ and $\lshft{\rho}=\rho\nu_\rho^{-1}$.
Note that $\nu_{\rshft\rho}=\nu_{\lshft\rho}=\nu_\rho$.

Let $A$ be a subset of $\Cusp$.
By definition, a \emph{connected component} of $A$ is an equivalence class of the equivalence
relation generated by $\rho\sim\rshft{\rho}$ whenever $\rho$ and $\rshft{\rho}$ are in $A$.

\begin{defi}
A {\it segment} $\Delta$ is a finite nonempty connected subset of $\Cusp$ (i.e., a single connected component).
Thus, $\Delta$ is of the form
\begin{equation}
\Delta=\{\rho\nu_{\rho}^c,\rho\nu_{\rho}^{c+1},\dots,\rho\nu_{\rho}^d\}
\end{equation}
where $\rho\in \Cusp$ and $c,d\in\Z$ are two integers such that $c\le d$.
We denote $\Delta$ as above by $\left[c,d \right]_\rho$.
Thus, $\left[c,d \right]_\rho=\left[c',d' \right]_{\rho'}$ precisely when $\rho\nu_{\rho}^c=\rho'\nu_{\rho'}^{c'}$
and $\rho\nu_{\rho}^d=\rho'\nu_{\rho'}^{d'}$.

We write $\seg$ for the set of all segments (as we vary $\rho\in\Cusp$).
\end{defi}

We write $l(\Delta)$ for the size of $\Delta$ (also called the length of $\Delta$).
By convention we also set $[c,d]_\rho=\emptyset$ if $d<c$.

We denote the extremities of $\Delta=\left[c,d \right]_\rho$ by $b(\Delta)=\rho\nu_\rho^c\in\Cusp$ and $e(\Delta)=\rho\nu_\rho^d\in\Cusp$ respectively.
We also write:
\begin{align*}
\Delta^\vee&=[-d,-c]_{\rho^\vee}\\
\lshft\Delta&= \left[c-1,d-1\right]_\rho\\
\Delta^-&=\left[c,d-1\right]_\rho\\
^-\Delta&=\left[c+1,d \right]_\rho\\
\Delta^+&=\left[c,d+1\right]_\rho\\
^+\Delta&=\left[c-1,d \right]_\rho.
\end{align*}

To any segment $\Delta=\left[c,d\right]_\rho$ define
\[
 \zele{\Delta}=\soc(\rho\nu_{\rho}^c\times\rho\nu_{\rho}^{c+1}\times\dots\times\rho\nu_{\rho}^d) \simeq
\cos(\rho\nu_{\rho}^d\times\rho\nu_{\rho}^{d-1}\times\dots\times\rho\nu_{\rho}^c).
\]
It is an irreducible representation which can be characterized as the unique $\pi \in \Irr$ such that
\begin{equation} \label{eq: Jaczele}
\rp(\pi)=\rho\nu_{\rho}^c\otimes\rho\nu_{\rho}^{c+1}\otimes\dots\otimes\rho\nu_{\rho}^d.
\end{equation}
Note in particular that $\zele{\Delta}$ is left $\Cusp\setminus\{b(\Delta)\}$-reduced.
It will be useful to set $\zele{\emptyset}$ equal to the one-dimensional representation of the trivial group $\G_0$.

We have $\zele{\Delta^\vee}=\zele{\Delta}^\vee$.

\begin{defi}\label{def:segments}
\label{link}
Let $\Delta=\left[c,d\right]_\rho$ and $\Delta'=\left[c',d'\right]_{\rho'}$ be two segments.
We say that $\Delta$ and $\Delta'$ are \emph{linked} if $\Delta\cup\Delta'$ forms a segment but neither
$\Delta\subset\Delta'$ nor $\Delta'\subset\Delta$.
If $\Delta$ and $\Delta'$ are linked and $\rho'\nu_{\rho'}^{c'}=\rho\nu_\rho^{c+j}$ with $j>0$ then
we say that $\Delta$ \emph{precedes} $\Delta'$ and write $\Delta \prec \Delta'$.
\end{defi}

Thus, if $\Delta$ and $\Delta'$ are linked then either $\Delta\prec\Delta'$ or $\Delta'\prec\Delta$, but not both.

\subsection{Multisegments and classification} \label{sec: multi and classification}
Given a set $\X$, write $\NN(\X)$ for the commutative semigroup of maps from
$\X$ to $\NN$ with finite support.

\begin{defi}
A \textit{multisegment} is a multiset of segments, that is an element in $\MS:= \NN(\seg)$.
\end{defi}

We write a multisegment $\m$ as a finite sum $\m=\Delta_1+ \dots + \Delta_N$, with $\Delta_i \in \seg$.
For such $\m$ we define the support of $\m$, denoted $\supp\m$, to be $\cup\Delta_i\subset\Cusp$.
We also write $\m^\vee=\Delta_1^\vee+ \dots + \Delta_N^\vee$.

A sequence of segments $(\Delta_1, \dots, \Delta_N)$ is said to be \textit{\rangee} if for every $1 \leq i < j \leq N$, $\Delta_i$ does not precede $\Delta_j$.
If $\m \in \MS$ and $(\Delta_1, \dots, \Delta_N)$ is an \rangee\ sequence of segments such that $\m=\Delta_1+ \dots+ \Delta_N$,
we say that $(\Delta_1, \dots, \Delta_N)$ is an \emph{\rangee\ form} of $\m$.

For any \rangee\ form $(\Delta_1,\dots,\Delta_N)$ of a multisegment $\m \in \MS$ let
$$\std(\m):= \zele{\Delta_1} \times\dots\times \zele{\Delta_N}$$
be the standard module. Also, set
$$\tilde{\std}(\m):= \zele{\Delta_N} \times\dots\times \zele{\Delta_1}.$$
Up to isomorphism, $\std(\m)$ and $\tilde{\std}(\m)$ do not depend on the choice of the \rangee\ form.
Note that $[\std(\m)]=[\tilde{\std}(\m)]$.

The set $\Irr$ is in natural bijection with $\MS$ (the Zelevinsky classifications).
More precisely

\begin{thm}  \label{thm: classification}
For any $\m\in\MS$, $\std(\m)$ is \EI\ and the map
\[
\m \mapsto \zele{\m}:=\soc(\std(\m))=\cos(\tilde{\std}(\m))
\]
is a bijection from $\MS$ to $\Irr$.
\end{thm}

Given $\m,\n\in\MS$ we write $\m\not\prec\n$ for the condition that $\Delta\not\prec\Delta'$ for any $\Delta$ in $\m$ and $\Delta'\in\n$.

We will say that $\m=\Delta_1+\dots+\Delta_N$ is \emph{\TU} if the $\Delta_i$'s are pairwise unlinked. \label{sec: TU}

In the following proposition we record some properties of the map $\m\mapsto\zele{\m}$ which we will use often.
(See \S\ref{par: preuve}.) We thank the referee for suggesting to us to add part \ref{part: anyorder}.

\begin{prop}\label{useful_properties}
Let $\m,\n\in\MS$. Then
\begin{enumerate}
\item \label{part: support is preserved} $\supp\zele{\m}=\supp\m$.

\item \label{part: pairwise unlinked}
The induced representation $\std(\m)$ is irreducible if and only if $\m$ is \TU.

\item \label{part: anyorder}
If $\m=\Delta_1+\dots+\Delta_N$ (in any order) then there exists a non-zero intertwining map
\[
\zele{\Delta_1}\times\dots\times\zele{\Delta_N}\xrightarrow{f}\std(\m).
\]
Moreover, $[\zele{\m}]\not\le[\Ker f]$.

\item \label{part: vdash}
$[\zele{\m} ] \leq [\std(\n)]$ if and only if $\m \vdash \n$, where $\vdash$ is the
\emph{``union-intersection''} order of Zelevinsky (see \cite{MR584084}*{\S7}).

\item \label{part: lmm'}
$\zele{\m+\n}$ occurs with multiplicity one in $[\zele{\m} \times\zele{\n} ]$.

\item \label{part: pwnp}
If $\m\not\prec\n$ then $\zele{\m+\n}=\soc(\zele{\m} \times \zele{\n})=\cos(\zele{\n}\times\zele{\m})$.
In particular, if both $\m\not\prec\n$ and $\n\not\prec\m$, i.e. if no $\Delta \in \m$ and $\Delta' \in \n$ are linked,
then $\zele{\m} \times\zele{\n}$ is irreducible.

\item \label{part: contragredient}
$\zele{\m^\vee}=\zele{\m}^\vee$.
\end{enumerate}
\end{prop}

\subsection{Langlands classification}
The Langlands classification for the groups $G_n$ runs parallel to the Zelevinsky classification.
To any segment $\Delta=\left[c,d\right]_\rho$ we denote by $\lang{\Delta} $ the cosocle of
$\rho\nu_{\rho}^c\times\rho\nu_{\rho}^{c+1}\times\dots\times\rho\nu_{\rho}^d$.
It is the unique $\pi \in \Irr$ such that
\begin{equation} \label{eq: JacLang}
\rp(\pi)=\rho\nu_{\rho}^d\otimes\rho\nu_{\rho}^{d-1}\otimes\dots\otimes\rho\nu_{\rho}^c.
\end{equation}

The representation $\lang{\Delta}$ is essentially square-integrable.
Conversely, any essentially square-integrable $\pi\in\Irr$ is of the form $\lang{\Delta}$ for some $\Delta\in\seg$.
However, we will not use this fact.

For any \rangee\ form $(\Delta_1,\dots,\Delta_N)$ of a multisegment $\m \in \MS$ define
$\lstd(\m):= \lang{\Delta_1} \times\dots\times \lang{\Delta_N}$ and
$\tilde{\lstd}(\m):= \lang{\Delta_N} \times\dots\times \lang{\Delta_1}$.
As before, $\lstd(\m)$ and $\tilde{\lstd}(\m)$ do not depend (up to an isomorphism) on the choice of the \rangee\ form.

\begin{thm}  \label{thm: classificationL}
For any $\m\in\MS$, $\lstd(\m)^\vee$ is \EI\ and the map
\[
\m \mapsto \lang{\m}:=\cos(\lstd(\m))=\soc(\tilde{\lstd}(\m))
\]
is a bijection from $\MS$ to $\Irr$ which satisfies the analogue of Proposition \ref{useful_properties} with

\begin{itemize}
\item $\zele{\cdot}$ replaced by $\lang{\cdot}$,
\item $\std(\cdot)$ replaced by $\lstd(\cdot)$,
\item the map $f$ in part \ref{part: anyorder} goes in the opposite direction, and,
\item the condition $\m\not\prec\n$ replaced by $\n\not\prec\m$ in part \ref{part: pwnp}.
\end{itemize}
\end{thm}

By Theorem \ref{thm: classification}, for any $\m\in\MS$ there exists a unique multisegment $\m^t\in \MS$ such that $\lang{\m}=\zele{\m^t}$.
On a combinatorial level, the map $\m\mapsto\m^t$ is described by the M\oe glin--Waldspurger algorithm (MW-algorithm) of \cite{MR863522}.
In particular, it is an involution, i.e., $\zele{\m}=\lang{\m^t}$.
See \S\ref{ap:involution} and also below.

\subsection{Cuspidal lines}\label{par: cuspidal_lines}
Let $\rho \in \Cusp$. Define
\begin{align*}
\Z_\rho &=\{\rho \nu_\rho^a:a \in \Z \}, \\
\Z_{\leq \rho} &=\{\rho \nu_\rho^a:a \in \Z_{\leq 0} \}, \\
\seg_\rho&=\{\Delta\in\seg:\Delta\subset\Z_\rho\},\\
\Irr_\rho&=\{ \pi \in \Irr:\supp(\pi) \subset \Z_\rho\},\\
\MS_\rho&=\{ \m \in \MS:\supp(\m) \subset \Z_\rho\}=\NN(\seg_\rho).
\end{align*}
We call $\Z_\rho$ the \emph{line of $\rho$}.
Irreducible representations (resp., multisegments) with support in a line are called \textit{rigid} representations (resp., multisegments).
It follows immediately from Proposition \ref{useful_properties} part \ref{part: pwnp}, that if the lines of $\rho_1, \dots, \rho_r$
are distinct (hence, pairwise disjoint) and $\pi_i \in \Irr_{\rho_i}$, $i=1,\dots,r$ then
$$\pi_1 \times \dots \times \pi_r$$
is irreducible. In practice, this allows to confine ourselves to rigid representations and rigid multisegments.

We write $\le$ for the partial order on $\Z_\rho$ inherited by $\Z$. (Of course this does not depend on the choice of $\rho$.)
Given $\pi\in\Irr_\rho$ we write $e(\pi)\in\Cusp$ (resp., $b(\pi)$) for the largest (resp., smallest) element in $\supp\pi$.
Similarly, we use the notation $e(\m)$ and $b(\m)$ for $\m\in\MS_\rho$.

We define two lexicographic orders on $\seg_\rho$.
Namely, given $\Delta,\Delta'\in\seg_\rho$ we write $\Delta \ge_e\Delta'$ if $e(\Delta)\ge e(\Delta')$ and either $e(\Delta)>e(\Delta')$ or
$b(\Delta)\ge b(\Delta')$.
Similarly, $\Delta\ge_b\Delta'$ if $b(\Delta)\ge b(\Delta')$ and either $b(\Delta)>b(\Delta')$ or $e(\Delta)\ge e(\Delta')$.
Given $\m\in\MS$ we will also write $\Delta<_b\m$ if $\Delta<_b\Delta'$ for all $\Delta'\in\m$.

Given $\m\in\MS_\rho$ let us write $\maxdel(\m)$ (resp., $\mindel(\m)$) for the maximal (resp., minimal) segment in $\m$ with respect to $\le_e$.

A sequence $(\Delta_1,\dots, \Delta_N)$ of segments in the same line is \emph{right \ordered} (resp., left \ordered) if $\Delta_1 \ge_e \dots \ge_e \Delta_N$
(resp., $\Delta_1 \ge_b \dots \ge_b \Delta_N$). (Either condition implies \rangee.)
To each rigid multisegment $\m$ we can associate a unique right {\ordered} sequence of segments $(\Delta_1,\dots \Delta_N)$ such that
$\m= \Delta_1+\dots +\Delta_N$. We call it the \emph{right {\ordered} form} of $\m$. Similarly, there is a unique left {\ordered} form of $\m$.

We mention the following consequence of Lemma \ref{lem: mult1 jf}.
\begin{rem}\label{rem:foreff}
Let $\m=\Delta_1+ \dots  + \Delta_N\in\MS_\rho$ be a multisegment such that, for every $1 \leq i \leq N$, $\Delta_i$ is not contained in $\Z_{\leq \rho}$.
Then $\lstd(\m)$, $\tilde{\lstd}(\m)$ and therefore also $\lang{\m}$, are left $\Z_{\le\rho}$-reduced.
\end{rem}

Given $\rho\in\Cusp$ and $\pi\in\Irr_\rho$ we will write for brevity $\pi_{\le\rho}$ and $\pi_{\not\le\rho}$ for $\pi_I$ and $\pi_{\lnot I}$
respectively (see Lemma \ref{lem:  decomI}) where $I=\Z_{\le\rho}$.

We also use the following convention: for any property $P$ of segments and $\m=\Delta_1+\dots+\Delta_N\in\MS$ we set
\begin{equation} \label{eq: m_P}
\m_P=\sum_{i:\Delta_i\text{ satisfies P}}\Delta_i.
\end{equation}
For instance, for any $\Delta\in\seg_\rho$ and $\m\in\MS_\rho$ we write
\[
\m_{\le_b\Delta}=\sum_{i:\Delta_i\le_b\Delta}\Delta_i.
\]

Using Theorem \ref{thm: classificationL} and Remark \ref{rem:foreff} we can describe $\pi_{\le\rho}$ and $\pi_{\not\le\rho}$ combinatorially
in terms of the Langlands classification as follows: if $\pi=\lang{\n}\in\Irr_\rho$ then
\[
\pi_{\le\rho}=\lang{\n_{\le_e\rho}}
\]
and
\[
\pi_{\not\le\rho}=\lang{\n_{>_e\rho}}.
\]

Suppose that $\pi=\zele{\m}\in\Irr_\rho$ with $\m=\Delta_1+\dots+\Delta_N$ and
consider the M\oe glin--Waldspurger algorithm for $\m$ \cite{MR863522}.
The segments $\Delta'_1,\dots,\Delta'_{N'}$ of $\m^t$ are obtained by a recursive procedure one at a time.
More precisely, one forms a sequence of multisegments $\m_i$, $i=0,\dots,N'$ with $\m_0=\m$ and $\m_{N'}=0$
such that $\m_i$ is obtained from $\m_{i-1}$ by removing the end points of certain segments of it, and
these end points form the segment $\Delta'_i$. Then $\pi_{\le\rho}=\zele{\m_k}$ where $k$ is the smallest index such that
$\supp(\m_k)\subset\Z_{\le\rho}$.

For any rigid $\n\in\MS_\rho$ denote by $\n^-$ what remains from $\n$ after the first step of MW algorithm is performed.
More precisely (cf.~\cite{MR863522}*{\S II.2}), writing $\n=\Delta'_1+\dots+\Delta'_{N'}$ with $\Delta'_1\ge_e\dots\ge_e\Delta'_{N'}$ let
$k\ge0$ be the maximal integer for which there exist indices $j_0,\dots,j_k$ such that $e(\Delta'_{j_0})=e(\n)$ and $\Delta'_{j_l}\prec\Delta'_{j_{l-1}}$ and
$e(\Delta'_{j_l})=e(\lshft{\Delta'}_{j_{l-1}})$, $l=1,\dots,k$.
Define
\begin{subequations}
\begin{gather} \label{eq: j_0}
j_0=\min\{j:e(\Delta'_j)=e(\n)\},\\
\label{eq: j_l}
j_l=\min\{j:\Delta'_j\prec\Delta'_{j_{l-1}}\text{ and }e(\Delta'_j)=e(\lshft{\Delta'}_{j_{l-1}})\},\ l=1,\dots,k.
\end{gather}
\end{subequations}
Then
\[
\n^-=\n+\sum_{l=0}^k({\Delta'_{j_l}}^--\Delta'_{j_l}).
\]
We have $e(\maxdel(\n^t))=e(\n)=e(\Delta'_{j_0})$, $b(\maxdel(\n^t))=e(\Delta'_{j_k})$ and $\n^t=(\n^-)^t+\maxdel(\n^t)$.

\subsection{Speh representations} \label{sub:speh}
A particularly important class of irreducible representations is that of \emph{Speh} representations.
They arise from multisegments $\sum_{i=1}^N\Delta_i$ such that $\Delta_{i+1}=\lshft{\Delta_i}$ for all $i=1,\dots,N-1$.
These representations comprise the building blocks for the unitary dual of $\G_n$.
We will recall this at length in \S\ref{sec: U0}.

\subsection{Symmetry} \label{sec: starsym}
In the case $D=F$ it is known that given $\pi_1,\pi_2,\sigma\in\Irr$,
\begin{equation} \label{eq: onlyGLn}
\sigma\hookrightarrow\pi_1\times\pi_2\iff\pi_2\times\pi_1\twoheadrightarrow\sigma.
\end{equation}
Let us recall the argument (\cite{MR863522}*{p. 173}).
Denote by $\tilde\pi$ the composition of $\pi\in\Rep(\G_m)$ with the isomorphism $g\mapsto \,^tg^{-1}$ of $\G_m$.
Then we have $\widetilde{\pi_1\times\pi_2}=\tilde\pi_2\times\tilde\pi_1$ while on the other hand by a classical result of
Gelfand--Kazhdan \cite{MR0404534}, $\tilde\sigma\simeq\sigma^\vee$ for any $\sigma\in\Irr$.
Thus, if $\sigma\hookrightarrow\pi_1\times\pi_2$ then
\[
\sigma^\vee\simeq\tilde\sigma\hookrightarrow\tilde\pi_2\times\tilde\pi_1\simeq\pi_2^\vee\times\pi_1^\vee=(\pi_2\times\pi_1)^\vee
\]
so that $\sigma$ is a quotient of $\pi_2\times\pi_1$.

It is not clear how to extend this argument for $D\ne F$.
Nevertheless we can use a similar argument to obtain a weaker result in the general case. (See Lemma \ref{lem: sym} below.)

\begin{defi}
An \emph{\ulb\ segment} is an ordered pair of integers $a, b$ with $a\le b$. We denote it $[a,b]$.
An \emph{\ulb\ multisegment} is a multiset of \ulb\ segments.
To any \ulb\ multisegment $\m=[a_1,b_1]+\dots +[a_N,b_N]$ and $\rho \in \Cusp$, we associate the rigid multisegment $\m_\rho$ (or $\m_{\rho,D}$ if
$D$ is not clear from the context) defined by
$$\m_\rho=[a_1,b_1]_\rho+\dots +[a_N,b_N]_\rho.$$
\end{defi}

Let $n \geq 0$. Denote $\opp{G_n}=\GL_n(\opp{D})$ where $\opp{D}$ is the central division $F$-algebra opposite to $D$.
The map $g\mapsto \,^tg^{-1}$ defines an isomorphism of groups $i:\opp{G_n} \longrightarrow G_n$ which
gives rise to an equivalence of categories $\pi\mapsto\pi\circ i$ between $\Rep(G_n)$ and $\Rep(\opp{G_n})$.
Writing $\pi^\star$ for the contragredient of $\pi\circ i$ we get
a (contravariant) equivalence of categories between $\Rep(G_n)$ and $\Rep(\opp{G_n})$.
(Recall that in the case $D=\opp{D}=F$ we have $\pi^\star\simeq \pi$ for any $\pi \in \Irr$.)

Note that if $\pi_i\in\Rep(\G_{m_i})$, $i=1,\dots,r$ then
\begin{equation} \label{prope:indu}
(\pi_1\times\dots\times\pi_r)^\star \simeq \pi_r^\star\times\dots\times \pi_1^\star.
\end{equation}
In particular, $\star$ preserves supercuspidality.
We also have $(\chi\circ\Nrd_D)^\star= \chi\circ\Nrd_{\opp{D}}$ for any character $\chi$ of $F^\times$.
Therefore
\begin{equation}\label{eq:pichi}
(\pi\chi)^\star=\pi^\star\chi
\end{equation}
for any $\pi\in\Rep(\G_m)$.
It easily follows that $s_\rho=s_{\rho^\star}$ for any $\rho\in\Cusp$ and that for any segment $\Delta=[a,b]_\rho$ we have
$\zele{\Delta}^\star=\zele{\Delta^\star}$ and $\lang{\Delta}^\star=\lang{\Delta^\star}$
where $\Delta^\star=[a,b]_{\rho^\star}$ and $\zele{\Delta^\star}$ and $\lang{\Delta^\star}$ are defined with respect to $\opp{G_n}$.
Defining $\m^\star$ for $\m \in \MS$ by linearity we get
\begin{equation}\label{eq:star}
\zele{\m}^\star=\zele{\m^\star}, \qquad \lang{\m}^\star=\lang{\m^\star}.
\end{equation}

We will use $\star$ as follows.

\begin{lem} \label{lem: sym}
Let $\m,\m', \m^{\prime \prime}$ be three \ulb\ multisegments.
Suppose that for any central division $F$-algebra $D$ and any supercuspidal irreducible representation $\rho$ of $\GL_m(D)$ we have
$$\zele{\m^{\prime\prime}_{\rho,D}}=\soc(\zele{\m_{\rho,D}} \times \zele{\m'_{\rho,D}}).$$
Then
$$\zele{\m^{\prime\prime}_{\rho,D}}=\cos(\zele{\m'_{\rho,D}}\times \zele{\m_{\rho,D}}).$$
Suppose moreover that $\zele{\m_{\rho,D}} \times \zele{\m'_{\rho,D}}$ is \EI.
Then $\zele{\m_{\rho,D}} \times \zele{\m'_{\rho,D}}$ irreducible if and only if
$\zele{\m_{\rho,D}} \times \zele{\m'_{\rho,D}}\simeq \zele{\m'_{\rho,D}} \times \zele{\m_{\rho,D}}$.
\end{lem}

\begin{proof}
Indeed, using the assumption for $(\opp{D},\rho^\star)$ instead of $(D,\rho)$,
\begin{multline*}
\cos((\zele{\m_{\rho^\star,\opp{D}}} \times \zele{\m'_{\rho^\star,\opp{D}}})^\star)=
(\soc(\zele{\m_{\rho^\star,\opp{D}}} \times \zele{\m'_{\rho^\star,\opp{D}}}))^\star\\=
(\zele{\m^{\prime\prime}_{\rho^\star,\opp{D}}})^\star=\zele{\m^{\prime\prime}_{\rho,D}}
\end{multline*}
(by \eqref{eq:star}). However, by \eqref{prope:indu} and \eqref{eq:star}
\begin{align*}
(\zele{\m_{\rho^\star,\opp{D}}} \times \zele{\m'_{\rho^\star,\opp{D}}})^\star \simeq
\zele{\m'_{\rho^\star,\opp{D}}}^\star \times \zele{\m_{\rho^\star,\opp{D}}}^\star \simeq \\
\zele{\m^{\prime\star}_{\rho^\star,\opp{D}}} \times \zele{\m^{\star}_{\rho^\star,\opp{D}}}=
\zele{\m^{\prime}_{\rho,D}} \times \zele{\m_{\rho,D}}.
\end{align*}
The result follows.

For the last part, remark that if $\zele{\m_{\rho,D}} \times \zele{\m'_{\rho,D}}\simeq \zele{\m'_{\rho,D}} \times \zele{\m_{\rho,D}}$,
then by the first part
\[
\zele{\m^{\prime\prime}_{\rho,D}}=\soc(\zele{\m_{\rho,D}} \times \zele{\m'_{\rho,D}})=
\cos(\zele{\m_{\rho,D}} \times \zele{\m'_{\rho,D}}).
\]
Since $\zele{\m_{\rho,D}} \times \zele{\m'_{\rho,D}}$ is assumed to be \EI, this can only happen if it is irreducible.
\end{proof}

\section{Left multiplier representations} \label{sec: leftmultiplier}

\subsection{Definition and basic properties}

\begin{defi}
Given $\pi=\zele{\m}$ and $\sigma=\zele{\n}$ we write $\lcond(\pi,\sigma)$ (resp., $\rcond(\pi,\sigma)$) for the condition
\[
\zele{\m+\n}=\soc(\pi\times\sigma)
\]
(resp.,
\[
\zele{\m+ \n}=\socc(\pi  \times \sigma).)
\]
\end{defi}

We summarize the basic properties of this definition in the following Lemma.
We are grateful to the referee for suggesting to us part \ref{part: referee} and its proof.

\begin{lem} \label{lem: lcondprop}
\begin{enumerate}
\item $\lcond(\pi,\sigma)$ implies that $\pi\times\sigma$ is \EI.
\item The conditions $\lcond(\pi,\sigma)$ and $\rcond(\pi^\vee,\sigma^\vee)$ are equivalent.
\item \label{part: irred} $\pi\times\sigma$ is irreducible (in which case $\pi\times\sigma\simeq\zele{\m+\n}$)
if and only if $\lcond(\pi,\sigma)$ and $\rcond(\pi,\sigma)$.
\item \label{part: referee} $\lcond(\pi,\sigma)$ is equivalent to the condition
\[
\pi\times\sigma\hookrightarrow\std(\m+\n).
\]
\item Similarly, $\rcond(\pi,\sigma)$ is equivalent to the condition $\tilde\std(\m+\n)\twoheadrightarrow\zele{\m}\times\zele{\n}$.
\item \label{part: trivlcond}
If $\m\not\prec\n$ then $\lcond(\pi,\sigma)$ and $\rcond(\sigma,\pi)$ hold.
\end{enumerate}
\end{lem}

\begin{proof}
The first part follows from Proposition \ref{useful_properties} part \ref{part: lmm'}
(and would also follow from part \ref{part: referee} of the current lemma).
The second part is also clear.
The third part follows from Proposition \ref{useful_properties} part \ref{part: lmm'}.

Clearly $\lcond(\pi,\sigma)$ holds if $\pi\times\sigma\hookrightarrow\std(\m+\n)$ (by Remark \ref{rem:main}).
Conversely, suppose that $\soc(\pi\times\sigma)=\zele{\m+\n}$.
Embed $\pi$ in $\std(\m)$ and $\sigma$ in $\std(\n)$.
Then $\pi\times\sigma$ embeds in $\std(\m)\times\std(\n)$ which by Proposition \ref{useful_properties} part \ref{part: anyorder},
admits a non-trivial map $f$ to $\std(\m+\n)$ and $[\zele{\m+\n}]\not\le[\Ker f]$.
Thus, $f$ is injective on $\pi\times\sigma$, for otherwise $\Ker f\rest_{\pi\times\sigma}$ would contain $\soc(\pi\times\sigma)$
which is $\zele{\m+\n}$ by assumption.

The next part follows by passing to the contragredient.
Finally, the last part follows from Proposition \ref{useful_properties} part \ref{part: pwnp}.
\end{proof}


\begin{defi}
\begin{enumerate}
\item We say that $\pi\in\Irr$ is a \emph{left multiplier} (\RI) if for any $\sigma\in\Irr$, $\pi\times\sigma$ is \EI.
\item If $\pi=\zele{\m}$ is \RI\ and $\n\in\MS$ then we denote by $\combsoczele(\m,\n)$ the multisegment such that
$\zele{\combsoczele(\m,\n)}=\soc(\pi\times\zele{\n})$.
Similarly, if $\pi=\lang{\m}$ is \RI\ and $\n\in\MS$ then we denote by $\combsoclang(\m,\n)$ the multisegment such that
$\lang{\combsoclang(\m,\n)}=\soc(\pi\times\lang{\n})$.
\end{enumerate}
\end{defi}

Thus, if $\pi=\zele{\m}$ is $\RI$ and $\sigma=\zele{\n}$, then the condition $\lcond(\pi,\sigma)$ is equivalent to
$\zele{\m+\n}\hookrightarrow\pi\times\sigma$, or to $\combsoczele(\m,\n)=\m+\n$.

\begin{rem} \label{rem: LMnew2}
If $D=F$ and $\pi$ is \RI\ then $\pi\times\sigma$ is irreducible if and only if $\pi\times\sigma\simeq\sigma\times\pi$ (by \eqref{eq: onlyGLn}).
In general, we do not know this for $D\ne F$. However, see Lemma \ref{lem: sym}.
\end{rem}

\begin{rem}
It was proved in \cite{MR2527415} that a representation of the form $\zele{\m}$ or $\lang{\m}$ is \RI\ if
$\m=\Delta_1+ \dots +\Delta_N$ is such that for all $1 \leq i \leq j \leq N$ either $\Delta_i=\Delta_j$ or $\Delta_i$ and $\Delta_j$ are disjoint and unlinked.
Moreover, if $\rho\in\Cusp$ and $\sigma\in\Irr$, a description of $\soc(\rho\times\sigma)$ was given (see Theorem \ref{thm: cusptimes} below).

We will extend this result in the sequel.
\end{rem}

Although, we will not use it in this paper, we mention the following recent beautiful result due to Kang, Kashiwara, Kim and Oh.\footnote{It is
certainly possible that the methods of these authors are instrumental for the determination of $\soc(\pi\times\sigma)$ as well.}
\begin{thm}\cite{MR3314831, 1412.8106}
Let $\pi\in\Irr$. Assume that $\pi \times \pi$ is irreducible. Then $\pi$ and $\pi^\vee$ are \RI.
Moreover, if $\soc(\pi\times \sigma)\simeq\socc(\pi\times \sigma)$ then $\pi\times \sigma$ is irreducible.
\end{thm}
This answers in the affirmative a conjecture by B.~Leclerc \cite{MR1959765}*{Conjecture 3}.
Actually, the results in \cite{MR3314831, 1412.8106} (as well as Leclerc's conjecture itself) are stated in a different (related) context,
but it is easy to translate the argument to our language (the role of the $R$-matrix is played by the standard intertwining operator).
Details will be given elsewhere.

Leclerc also gave an example of $\pi\in\Irr$ such that $\pi\times\pi$ is semisimple of length $2$ (hence $\pi$ is not \RI).
Namely, $\pi=\zele{[3,4]_\rho,[1,3]_\rho,[2,2]_\rho,[0,1]_\rho}$. This disproved a conjecture made in \cite{MR1237826}.
We do not know how to combinatorially characterize the property that $\pi\times\pi$ is irreducible.



We start with a simple observation.

\begin{lem} \label{lem: suffuniquesub}
Let $\pi\in\Irr$.
Suppose that for any $\sigma\in\Irr$ such that $\supp(\sigma)\subset\supp(\pi)$,
$\pi\times\sigma$ is \EI. Then $\pi$ is \RI.
\end{lem}

\begin{proof}
Embed $\sigma\hookrightarrow\sigma_1\times\sigma_2$ where $\supp(\sigma_1)\subset\supp(\pi)$ and $\sigma_2$ is left $\supp(\pi)$-reduced (see Lemma \ref{lem: decomI}.)
Then $\pi\times\sigma\hookrightarrow\pi\times\sigma_1\times\sigma_2$.
By Lemma \ref{lem:main} (applied with $\pi_1=\pi\times\sigma_1$ and $\pi_2=\sigma_2$), $\pi\times\sigma_1\times\sigma_2$
(and a fortiori $\pi\times\sigma$) is \EI.
The lemma follows.
\end{proof}

\subsection{First examples}
A special case of \RI\ representations is the class of \emph{saturated} representations:
\begin{defi}\label{def:saturated}
We say that $\pi \in\Irr$ is saturated if $\pi\times\sigma$ is irreducible
for any $\sigma \in\Irr$ such that $\supp(\sigma)\subset\supp(\pi)$.
\end{defi}

For instance, it follows from Proposition \ref{useful_properties} part \ref{part: pwnp} that $\zele{\Delta}$ and $\lang{\Delta}$ are saturated for any $\Delta\in\seg$.
We will eventually classify the saturated representation in Proposition \ref{prop:saturated} below.

From Lemma \ref{lem: suffuniquesub} we infer:

\begin{coro}\label{cor:saturated}
If $\pi$ is saturated then $\pi$ is \RI.
\end{coro}

The next corollary provides us with another family of examples of \RI\ representations.

\begin{coro}\label{coro:generic is RI}
Suppose that $\m$ is \TU. (See \S\ref{sec: TU}.)
Let $\pi$ be either $\zele{\m}$ or $\lang{\m}$.
Then $\pi$ is \RI. In particular, every unramified, tempered or (in the case $D=F$) generic irreducible representation is \RI.
\end{coro}

\begin{proof}
We first note that the last part of the corollary follows from the first part, together with \cite{MR584084}*{Theorem 9.7}, Remark \ref{rem: tempreps} and
the fact that a representation of $\G_n$ is unramified if and only if it is of the form $\zele{\m}$ where $\m$ is \TU\ and
$\supp\m\subset\{\mathbf{1}\abs{\cdot}^s:s\in\C\}$ where $\mathbf{1}$ is the trivial character of $G_1=D^*$. See \cites{MR850742, MR3194013}.

We prove the first part. Consider the case $\pi=\zele{\m}$. (The other case is similar.)
Write $\m=\Delta_1+\dots+\Delta_N$.
We argue by induction on $N$. We already mentioned the case $N=1$.
For the induction step, assume without loss of generality that $\Delta_N$ is a segment of maximal length
among the $\Delta_i$'s.
Thus, $\Delta_N$ is a connected component of $\supp\pi$.
Let $\sigma\in\Irr$ be such that $\supp(\sigma)\subset\supp(\pi)$. Then $\zele{\Delta_N}\times\sigma$
is irreducible since if $\sigma=\zele{\n}$ then any segment in $\n$ is unlinked with $\Delta_N$.
Therefore, by induction hypothesis
\[
\pi\times\sigma=\zele{\Delta_1}\times\dots\times\zele{\Delta_{N-1}}\times(\zele{\Delta_N}\times\sigma)
\]
is \EI. The corollary now follows from Lemma \ref{lem: suffuniquesub}.
\end{proof}

\subsection{Further analysis of the condition $\lcond(\pi,\sigma)$}
\begin{lem} \label{lem: indEI}
Fix $\rho\in\Cusp$ and let $\pi\in\Rep(\G_m)$, $\sigma\in\Rep(\G_{m'})$ and $\m,\n\in\MS_\rho$.
Assume that
\begin{enumerate}
\item $[\pi]\le[\std(\m)]$.
\item $[\sigma]\le[\std(\n)]$.
\item $\pi$ and $\sigma$ are \EI.
\item \label{part: pairwise<b} For any $\Delta\in\m$ and $\Delta'\in\n$ we have $\Delta'<_b\Delta$.
\end{enumerate}
Then $\soc(\pi)\otimes\soc(\sigma)$ occurs with multiplicity one in $[\rp(\pi\times\sigma)]$
and hence (by Lemma \ref{lem:main}) $\pi\times\sigma$ is \EI.
Moreover, $\lcond(\soc(\pi),\soc(\sigma))$ holds.

A similar statement holds if $<_b$ is replaced by $<_e$ in assumption \ref{part: pairwise<b}.
\end{lem}

\begin{proof}
Let us analyze $[\rp(\pi\times\sigma)]$ using the geometric lemma (see \S\ref{lemmegeo}).
One summand is $[\pi\otimes\sigma]$ in which $\soc(\pi)\otimes\soc(\sigma)$ occurs once by assumption.
The other summands are of the form $[\tau_1\times\tau_1'\otimes\tau_2\times\tau_2']$ where
$\tau_1\otimes\tau_2$ is a composition factor of $\rp_{(d,m-d)}(\pi)$ (and in particular of
$\rp_{(d,m-d)}(\std(\m))$) and $\tau'_1\otimes\tau'_2$ is a composition
factor of $\rp_{(d',m'-d')}(\sigma)$ (and in particular of $\rp_{(d',m'-d')}(\std(\n))$) with $0\le d<m$, $0<d'\le m'$ and $d+d'=m$.
We claim that $\soc(\sigma)\not\le[\tau_2\times\tau_2']$.
Indeed, otherwise $\supp\tau_2=\supp\tau'_1$ and in particular, $\supp\tau'_1\subset\Z_{\ge b(\m)}$.
However, it is easy to see that this implies that $b(\m)\in\supp\tau'_1$, $b(\Delta')=b(\m)$ and $\supp\tau_1'\subset\Delta'$
where $\Delta'$ is the largest segment of $\n$ with respect to $\ge_b$.
On the other hand it is also clear that if $b(\m)\in\supp\tau_2$ then $\Delta\subset\supp\tau_2$
where $\Delta$ is the smallest segment of $\m$ with respect to $\le_b$.
We get a contradiction since $\Delta'\subsetneq\Delta$.

The condition $\lcond(\soc(\pi),\soc(\sigma))$ follows from Lemma \ref{lem: lcondprop} part \ref{part: trivlcond},
Proposition \ref{useful_properties} part \ref{part: vdash}
and the fact that for any $\m'\vdash\m$, $\n'\vdash\n$ we have $\m'\not\prec\n'$.

The argument with $<_e$ instead of $<_b$ in assumption \ref{part: pairwise<b} is similar.
\end{proof}

The following is an immediate consequence.

\begin{lem} \label{lem: redgem}
Suppose that $\pi=\zele{\m}\in\Irr_\rho$ and let $\sigma=\zele{\n}\in\Irr_\rho$.
Let $\sigma_1=\zele{\n_{\not<_b\m}}$ and $\sigma_2=\zele{\n_{<_b\m}}$.
Assume that $\pi\times\sigma_1$ is $\EI$.
Then $\pi\times\sigma_1\times\std(\n_{<_b\m})$ is \EI\ and therefore $\pi\times\sigma$ is \EI\ and
$\soc(\pi\times\sigma)=\soc(\pi'\times\sigma_2)$ where $\pi':=\soc(\pi\times\sigma_1)$.
Moreover, we have $\lcond(\pi',\sigma_2)$.
Thus, $\lcond(\pi,\sigma)$ if and only if $\lcond(\pi,\sigma_1)$.
\end{lem}

Throughout the rest of this section assume that $\pi\in\Irr$ is rigid and \RI\ and let $\rho=e(\pi)\in\Cusp$.
Let also $\sigma=\lang{\n}\in\Irr_\rho$ be such that $\supp\sigma\not\subset\Z_{\le\rho}$.
In order to determine $\soc(\pi \times \sigma)$, we can make the argument of Lemma \ref{lem: suffuniquesub} more effective as follows.

\begin{lem} \label{lem: eff stages}
The representation $\pi\times\sigma_{\le\rho}\times\tilde{\lstd}(\n_{>_e\rho})$ is \EI\ and
\[
\soc(\pi\times\sigma)=\soc(\pi\times\sigma_{\le\rho}\times\sigma_{\not\le\rho})=
\soc(\pi\times\sigma_{\le\rho}\times\tilde{\lstd}(\n_{>_e\rho})).
\]
Thus, if $\soc(\pi\times\sigma_{\le\rho})=\lang{\m'}$ then
$\soc(\pi\times\sigma)=\lang{\m'+\n_{>_e\rho}}$.

In other words, if $\lang{\m}$ is \RI\ then $\combsoclang(\m,\n)=\combsoclang(\m,\n_{\le_e e(\m)})+\n_{>_e e(\m)}$.
\end{lem}


\begin{proof}
Recall that $\sigma_{\not\le\rho}=\lang{\n_{>_e\rho}}$. Thus,
\[
\pi\times\sigma\hookrightarrow\pi\times\sigma_{\le\rho}\times\sigma_{\not\le\rho}\hookrightarrow
\pi\times\sigma_{\le\rho}\times\tilde{\lstd}(\n_{>_e\rho}).
\]
By Lemma \ref{lem:main} (with $\pi_1=\pi\times\sigma_{\le\rho}$ and $\pi_2=\tilde{\lstd}(\n_{>_e\rho})$)
and Remark \ref{rem:foreff} we infer that $\pi\times\sigma_{\le\rho}\times\tilde{\lstd}(\n_{>_e\rho})$
is \EI.
The last assertion of the lemma follows from (the analogue for $\lang{\cdot}$ of) Proposition \ref{useful_properties} part \ref{part: pwnp} and the fact that $\supp\m'\subset\Z_{\le\rho}$.
\end{proof}

From now on assume that $\pi=\zele{\m}$.
\begin{coro} \label{cor: k=l}
We have
\[
\Delta'\notin\combsoczele(\m,\n)^t
\]
for any $\Delta'\supsetneq\maxdel(\n^t)$.
\end{coro}

Indeed, by the previous lemma, every $\Delta\in\combsoczele(\m,\n)^t$ with $e(\Delta)=e(\n)$ belongs to $\n^t$.

\begin{lem} \label{lem: reptheoryanalog}
Let $\Delta'$ be any segment in $\n$ such that $e(\Delta')=e(\sigma)$.
Set $\n'=\n-\Delta'$ and $\sigma'=\lang{\n'}$. Then any two of the following conditions imply the third.
\begin{enumerate}
\item $\lcond(\pi,\sigma')$.
\item $(\m+\n^t)^t=(\m+\n'^t)^t+\Delta'$.
\item $\lcond(\pi,\sigma)$.
\end{enumerate}
\end{lem}

\begin{proof}
Write $\n=\Delta_1+\dots+\Delta_t$ where $\Delta_1\ge_e\dots\ge_e\Delta_t$ so that $e(\Delta_1)=e(\sigma)$.
Let $k$ be the largest index such that $e(\Delta_k)=e(\Delta_1)$, so that $\Delta'=\Delta_i$ for some $1\le i\le k$.
In particular $\lang{\Delta_1+\dots+\Delta_k}=\lang{\Delta_1}\times\dots\times\lang{\Delta_k}$.
We have
\[
\sigma\hookrightarrow\sigma''\times\lang{\Delta_1+\dots+\Delta_k}
\]
where $\sigma''=\lang{\Delta_{k+1}+\dots+\Delta_t}$. Thus,
\[
\pi\times\sigma\hookrightarrow\pi\times\sigma'\times\lang{\Delta'}\hookrightarrow
\pi\times\sigma''\times \lang{\Delta_1+\dots+\Delta_k}.
\]
As before, the right-hand side is \EI. Therefore,
\[
\soc(\pi\times\sigma)=\soc(\soc(\pi\times\sigma')\times\lang{\Delta'}).
\]
Suppose that $\pi':=\soc(\pi\times\sigma')=\zele{\m'}=\lang{\m'^t}$. Then
\[
\soc(\pi'\times\lang{\Delta'})=\lang{\m'^t+\Delta'}
\]
(since $\Delta'$ does not precede any segment in $\m'^t$).
The Lemma follows.
\end{proof}

Under certain additional assumptions we can sharpen the previous lemma as follows.
\begin{lem} \label{lem: equcond}
Suppose that for any $\Delta\in\seg_\rho$, $\m_{\le_e\Delta}$ and $\m_{>_e\Delta}$ are \RI.
Then the condition $\lcond(\m,\n)$ is equivalent to the conjunction of $\lcond(\m,\n^-)$ and $(\m+\n)^-=\m+\n^-$.
\end{lem}

\begin{proof}
By the previous lemma it suffices to show that $\lcond(\m,\n)$ implies $(\m+\n)^-=\m+\n^-$.
We may assume that $\n=\n_{\ge_e\mindel(\m)}$ since neither the condition $\lcond(\m,\n)$
nor the condition $(\m+\n)^-=\m+\n^-$ is affected by replacing $\n$ by $\n_{\ge_e\mindel(\m)}$.
Let $\m=\Delta_1+\dots+\Delta_k$ and $\n=\Delta'_1+\dots+\Delta'_{k'}$ with $\Delta_1\ge_e\dots\ge_e\Delta_k$
and $\Delta'_1\ge_e\dots\ge_e\Delta'_{k'}$.
Let $i_1<\dots<i_m$ be the indices given by
\begin{enumerate}
\item $i_1=\min\{i:e(\Delta'_i)=e(\n)\}$,
\item $i_{j+1}=\min\{i:\Delta'_i\prec\Delta'_{i_j}\text{ and }e(\Delta'_i)=e(\lshft{\Delta'}_{i_j})\}$, $j=1,\dots,m-1$,
\item There is no $\Delta'\in\n$ such that $\Delta'\prec\Delta'_{i_m}$ and $e(\Delta')=e(\lshft{\Delta'}_{i_m})$.
\end{enumerate}
Assume by contradiction that $(\m+\n)^-\ne\m+\n^-$.
Then there exist indices $l'\le m$ and $l\le k$ such that $\Delta_l\prec\Delta'_{i_{l'}}$, $e(\Delta_l)=e(\lshft{\Delta'}_{i_{l'}})$
and either $l'=m$ or $\Delta_l\subsetneq\Delta'_{i_{l'+1}}$.

Suppose first that $k=l$. Then necessarily $l'=m$ (since $\n=\n_{\ge_e\Delta_k}$).
We get a contradiction to Corollary \ref{cor: k=l} since $^+\maxdel(\n)\in(\m+\n)^t$.

If $k>l$ we write $\m=\m_1+\m_2$ and $\n=\n_1+\n_2$ where $\m_1=\m_{\ge_e\Delta_l}$,
$\m_2=\m_{<_e\Delta_l}$, $\n_1=\n_{\ge_e\Delta_l}$ and $\n_2=\n_{<_e\Delta_l}$.

Then
\[
\zele{\m}\times\zele{\n}\hookrightarrow\zele{\m_1}\times\zele{\m_2}\times\zele{\n_1}\times\zele{\n_2}
\]
and therefore by assumption,
\[
\zele{\m+\n}\hookrightarrow\zele{\m_1}\times\zele{\m_2}\times\zele{\n_1}\times\zele{\n_2}.
\]
We claim that this implies that
\begin{equation} \label{eq: inclu2}
\zele{\m+\n}\hookrightarrow\zele{\m_1}\times\zele{\n_1+\m_2}\times\zele{\n_2}\hookrightarrow
\zele{\m_1}\times\zele{\n_1}\times\zele{\m_2}\times\zele{\n_2}.
\end{equation}
Indeed, we have
\[
\cos(\zele{\m_2}\times\zele{\n_1})=\zele{\n_1+\m_2}=\soc(\zele{\n_1}\times\zele{\m_2})
\]
since $\n_1\not\prec\m_2$.
Therefore,
\[
\zele{\n_1+\m_2}\not\le[\Ker A]
\]
where $A$ is the projection
\[
\zele{\m_2}\times\zele{\n_1}\xrightarrow{A}\cos(\zele{\m_2}\times\zele{\n_1})\simeq\zele{\n_1+\m_2}.
\]
It follows that
\[
\zele{\m+\n}\not\le[\Ker (\id\times A\times\id)]
\]
where $\id\times A\times\id$ is the induced intertwining operator
\[
\zele{\m_1}\times\zele{\m_2}\times\zele{\n_1}\times\zele{\n_2}\rightarrow
\zele{\m_1}\times\zele{\n_1+\m_2}\times\zele{\n_2}
\]
since any irreducible $\zele{\m'}\le[\Ker(\id\times A\times\id)]$ satisfies $\m'\underset{\neq}{\vdash}\m_1+\n_1+\m_2+\n_2=\m+\n$.
This gives \eqref{eq: inclu2}.
In turn, by Lemma \ref{lem: indEI}
\[
\zele{\m_1}\times\zele{\n_1}\times\zele{\m_2}\times\zele{\n_2}
\]
is \EI\ since both $\zele{\m_1}\times\zele{\n_1}$ and $\zele{\m_2}\times\zele{\n_2}$ are \EI\
by assumption and any segment of $\m_1+\n_1$ is $>_e$ any segment of $\m_2+\n_2$.
Thus, we get
\[
\zele{\m+\n}=\soc(\zele{\m_1}\times\zele{\n_1}\times\zele{\m_2}\times\zele{\n_2})=
\soc(\soc(\zele{\m_1}\times\zele{\n_1})\times\soc(\zele{\m_2}\times\zele{\n_2})).
\]
Therefore, necessarily
\[
\zele{\m_1+\n_1}=\soc(\zele{\m_1}\times\zele{\n_1})
\]
which yields a contradiction to the case $k=l$ considered above.
\end{proof}



\section{Irreducibility of $\zele{\Delta} \times \sigma$} \label{sec: square-integrable}

\subsection{First description of $\soc(\zele{\Delta}\times\sigma)$}
Fix $\rho\in\Cusp$ and let $\Delta\in\seg_\rho$ and $\sigma\in\Irr_\rho$.
Note that by Lemma \ref{lem: lcondprop} part \ref{part: trivlcond}, the condition $\lcond(\zele{\Delta},\sigma)$ holds if $\supp\sigma\subset\Z_{\le e(\Delta)}$.
From this and Lemma \ref{lem: eff stages}, together with the definition of $\m^t$, we infer:

\begin{prop} \label{prop: delsig}
Let $\Delta\in\seg_\rho$ and $\sigma=\lang{\m}\in\Irr_\rho$.
Write $\m_1=\m_{\le_e e(\Delta)}$ and $\m_2=\m_{\ge_b b(\Delta)}$. (See \eqref{eq: m_P}.)
Then
\begin{enumerate}
\item $\soc(\zele{\Delta}\times \sigma)=\lang{(\m_1^t+\Delta)^t+\m_{>_e e(\Delta)}}$.
 \item $\soc(\zele{\Delta}\times\sigma)=\lang{\m+\Delta^t}$ if and only if $(\m_1^t+\Delta)^t=\m_1+\Delta^t$.
 \item $\cos(\zele{\Delta}\times\sigma)=\lang{\m+\Delta^t}$ if and only if $(\m_2^t+\Delta)^t=\m_2+\Delta^t$.
 \item $\zele{\Delta}\times\sigma$ is irreducible if and only if $(\m_1^t+\Delta)^t=\m_1+\Delta^t$ and
 $(\m_2^t+\Delta)^t=\m_2+\Delta^t$.
\end{enumerate}
\end{prop}

Indeed, the first part follows from Lemma \ref{lem: eff stages} and the remark above.
The second part is an immediate consequence.
The third part follows by passing to the contragredient.
Finally the last part follows from the previous parts and the fact that $\lang{\m+\Delta^t}$ occurs with multiplicity one
in $[\zele{\Delta}\times\sigma]$.

Proposition \ref{prop: delsig} gives a satisfactory necessary and sufficient combinatorial condition for the irreducibility of $\zele{\Delta}\times\sigma$
which is algorithmically very practical.
We can interpret the first part of the proposition as follows:
in order to obtain $\soc(\zele{\Delta}\times\zele{\m})^t$ one performs the M\oe glin--Waldspurger algorithm
on $\m$ until the remaining (truncated) segments are supported in $\Z_{\le e(\Delta)}$. Then one inserts $\Delta$ as an additional segment and continues with the algorithm.
In other words, one implements a modified algorithm on $\m+\Delta$ in which $\Delta$ is allowed to be used in the construction of
new segments only after the support of the remainder of $\m$ (after the previous segments were constructed) is contained in $\Z_{\le e(\Delta)}$.

Using Lemma \ref{lem: sym} we infer
\begin{coro}\label{cor:irred1}
Let $\Delta\in\seg$ and $\m\in\MS$.
Then $\zele{\Delta}\times\zele{\m}$ is irreducible (in which case $\zele{\Delta}\times\zele{\m} \simeq \zele{\Delta+\m} $)
if and only if $\zele{\Delta}\times\zele{\m} \simeq \zele{\m} \times \zele{\Delta}$.
\end{coro}

We can also use Proposition \ref{prop: delsig} to characterize when $\pi\hookrightarrow\zele{\Delta}\times\sigma$
for some $\sigma$.

\begin{coro} \label{cor: inxpi'}
Let $\pi=\lang{\n}\in\Irr_\rho$ and $\Delta\in\seg_\rho$.
Then there exists $\sigma\in\Irr$ such that $\pi=\soc(\zele{\Delta}\times\sigma)$ (i.e., $\pi\hookrightarrow\zele{\Delta}\times\sigma$)
if and only if $(\n_{\le_e e(\Delta)})^t$ contains $\Delta$. In this case $\sigma=\lang{\m}$ where
$\m_{>_e e(\Delta)}=\n_{>_e e(\Delta)}$ and $(\m_{\le_e e(\Delta)})^t=(\n_{\le_e e(\Delta)})^t-\Delta$.
In particular, $\sigma$ is uniquely determined by $\pi$ and $\Delta$.
\end{coro}

A special case of Proposition \ref{prop: delsig} is the following.
Recall that a pair of segments are \emph{juxtaposed} if they are linked and disjoint.

\begin{coro}\cite{MR3237446}*{Th\'eor\`eme 3.4} \label{cor: BLMoneseg}
Suppose that no $\Delta'\in\m$ is juxtaposed with $\Delta$. Then $\zele{\Delta}\times\lang{\m}$ is irreducible.
\end{coro}

\begin{proof}
We can assume that $\Delta\in\seg_\rho$ and $\m\in\MS_\rho$.
Set $\m_1=\m_{\leq_e e(\Delta)}$. We show that $(\m_1 + \Delta^t)^t= \m_1^t+ \Delta$.
Indeed, $\Delta^t=\sum_{\rho\in\Delta}\rho$.
When constructing $(\m_1 + \Delta^t)^t$ using the MW-algorithm, the first segment is formed using the end points of $\Delta^t$
(hence all of $\Delta^t$) and nothing else since no $\Delta'\in\m_1$ precedes $b(\Delta)$ by our assumption.
Then one remains with $\m_1$, so that the MW-algorithm gives indeed $\Delta+\m_1^t$.

The condition $((\m_{\ge_b b(\Delta)})^t+\Delta)^t=\m_{\ge_b b(\Delta)}+\Delta^t$ follows by passing to the contragredient and using
the fact that $^t$ commutes with the contragredient.
Hence, by Proposition \ref{prop: delsig}, $\zele{\Delta}\times\lang{\m}$ is irreducible.
\end{proof}

Note that the combinatorial condition in Corollary \ref{cor: BLMoneseg} is by no means necessary for the irreducibility of $\zele{\Delta}\times\lang{\m}$.

\subsection{Second description of $\soc(\zele{\Delta}\times\sigma)$}
In practice, it will be also useful to have another description of $\soc(\zele{\Delta}\times\sigma)$,
Recall that $\zele{\Delta}$ is \RI\ for any segment $\Delta\in\seg$. We start with the following result, which is a particular case of Lemma \ref{lem:main}.

\begin{lem} \label{lem: subrep1}
Let $\Delta\in\seg$, $\rho\in\Cusp$ and $\sigma\in\Irr$.
Suppose that $b(\Delta)\ne\rho$ and $\sigma$ is left $\rho$-reduced.
Then for any $n\ge0$ the representation $\rho^{\times n}\times\zele{\Delta}\times\sigma$ is \EI.
\end{lem}

\begin{proof}
By hypothesis, $\zele{\Delta}\times\sigma$ is left $\rho$-reduced and by Corollary \ref{cor:saturated} (or \cite{MR2527415}), it is \EI.
The result now follows from Lemma \ref{lem:main}.
\end{proof}

We deduce the following proposition.
\begin{prop}\label{prop:segnon}
Let $\Delta$ be a segment and $\sigma \in \Irr$. Let $\rho=b(\Delta)$.
Write $\sigma\hookrightarrow\rho^{\times a}\times\sigma_{\lnot\{\rho\}}$.
Then
\begin{equation}\label{eq:expl}
\soc(\zele{\Delta}\times\sigma)=\soc(\rho^{\times(a+1)}\times\soc(\zele{\,^-\Delta}\times\sigma_{\lnot\{\rho\}})).
\end{equation}
\end{prop}

\begin{proof}
We have
\[
\zele{\Delta}\times\sigma\hookrightarrow\zele{\Delta}\times\rho^{\times a}\times\sigma_{\lnot\{\rho\}}
\simeq\rho^{\times a}\times\zele{\Delta}\times\sigma_{\lnot\{\rho\}}
\hookrightarrow\rho^{\times(a+1)}\times\zele{\,^-\Delta}\times\sigma_{\lnot\{\rho\}}.
\]
By Lemma \ref{lem: subrep1}, $\soc(\rho^{\times(a+1)}\times\zele{\,^-\Delta}\times\sigma_{\lnot\{\rho\}})$ is irreducible,
hence equals the right hand side of \eqref{eq:expl}.
The proposition follows.
\end{proof}

To describe $\soc(\zele{\Delta}\times\sigma)$ using \eqref{eq:expl} we thus need to understand:
\begin{enumerate}
\item How to get $\sigma_{\lnot\{\rho\}}$ from $\sigma$.
\item How to treat the case where $\Delta=\{\rho\}$ with $\rho\in\Cusp$.
\end{enumerate}
This was done in \cite{MR2527415}*{\S6} using \cite{MR863522}*{Lemme II.9}.
We will describe it in a slightly different combinatorial language which will be useful in the sequel.

\subsection{Matching functions\footnote{The contents of this section are no doubt very familiar to combinatorists.
For convenience we give a self-contained treatment.}}
Let $X$ and $Y$ be finite sets and $\rltn$ a relation between $Y$ and $X$.
We are interested in injective functions
$f:X\rightarrow Y$ satisfying $f(x)\rltn x$ for all $x\in X$.
Such a function will be called a $\rltn$-matching function (or simply a matching function if $\rltn$
is clear from the context).
By Hall's criterion such an $f$ exists if and only if for all subsets $A\subset X$ we have
\begin{equation} \label{eq: HallCriterion}
\#\{y\in Y:y\rltn x\text{ for some }x\in A\}\ge \# A.
\end{equation}

In some cases it is possible to construct $f$ explicitly.
Suppose that $X$ and $Y$ are totally ordered with respect to $\le_X$ and $\le_Y$ respectively.
A simple-minded recipe would be to define $f$ recursively (from the largest element of $X$ to the smallest) by
\begin{equation} \label{eq: def f}
f(x)=\min\{y \in Y\setminus f(X_{>x}):y\rltn x\}
\end{equation}
where $X_{>x}=\{x'\in X:x'>x\}$.
Of course for this to be well defined, we need to know that for all $x\in X$
there exist some $y\notin f(X_{>x})$ such that $y\rltn x$.
Clearly, we need some extra conditions on the relation $\rltn$ for this definition to work.
We will say that the relation $\rltn$ is \emph{\trns}\ if the following condition is satisfied:
\begin{multline} \label{eq: specrelation}
\forall\ x_1\ge_X x_2\in X\ \ \forall\ y_1\ge_Y y_2\in Y\ \\
y_1\rltn x_1\ \&\ y_2\rltn x_1\ \&\  y_2\rltn x_2\ \implies y_1\rltn x_2.
\end{multline}
Schematically,
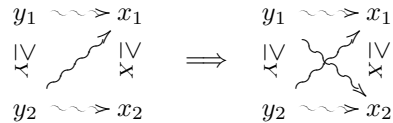
\[
\xymatrix{
y_1 \ar@{~>}[r] \ar@{}[d]|{\begin{sideways}\begin{sideways}\begin{sideways}$\ge_Y$\end{sideways}\end{sideways}\end{sideways}} &
x_1 \ar@{}[d]|{\begin{sideways}\begin{sideways}\begin{sideways}$\ge_X$\end{sideways}\end{sideways}\end{sideways}}\\
y_2\ar@{~>}[r]\ar@{~>}[ur] &  x_2
}
\begin{array}{c}
\\ \\ \\
\ \ \ \Longrightarrow\ \ \ \end{array}
\xymatrix{
y_1 \ar@{~>}[r] \ar@{~>}[dr] \ar@{}[d]|{\begin{sideways}\begin{sideways}\begin{sideways}$\ge_Y$\end{sideways}\end{sideways}\end{sideways}} &
x_1 \ar@{}[d]|{\begin{sideways}\begin{sideways}\begin{sideways}$\ge_X$\end{sideways}\end{sideways}\end{sideways}}\\
y_2\ar@{~>}[r]\ar@{~>}[ur] &  x_2
}
\]

More generally, even if Hall's criterion is not satisfied we can still speak about $\rltn$-\emph{matchings} (or simply matchings,
if $\rltn$ is clear from the context) between $X$ and $Y$ by which
we mean injective functions $f$ from a subset of $X$ to $Y$ satisfying $f(x)\rltn x$ for all $x$ in the domain of $f$.
We view such a function as a relation between $X$ and $Y$.

Imitating the above, we define a `greedy' $\rltn$-matching between $X$ and $Y$.
Namely, we define $f$ and its domain $I$ recursively by
\begin{multline*}
x\in I\iff \exists y\in Y\setminus f(I\cap X_{>x})\text{ such that }y\rltn x\\\text{  in which case }
f(x)=\min\{y\in Y\setminus f(I\cap X_{>x}):y\rltn x\}.
\end{multline*}
We call $f$ the \emph{best $\rltn$-matching} between $X$ and $Y$.
The name is justified by the following lemma:

\begin{lem} \label{lem: optf}
Assume that $\rltn$ is \trns.
Let $g$ be a $\rltn$-matching between $X$ and $Y$ and let $f$ be the best $\rltn$-matching between $X$ and $Y$.
Then the size of the domain of $f$ is not smaller than the size of the domain of $g$.

In particular, the following conditions are equivalent:
\begin{enumerate}
\item There exists a $\rltn$-matching function from $X$ to $Y$ (with domain $X$).
\item Hall's criterion \eqref{eq: HallCriterion} is satisfied.
\item The best $\rltn$-matching between $X$ and $Y$ is a function from $X$ to $Y$.
\end{enumerate}
\end{lem}

\begin{proof}
Let $g$ be a matching between $X$ and $Y$.
Suppose that $f$ does not coincide with $g$ (otherwise there is nothing to prove).
Let $x_1=x_1(g)$ be the smallest element of $X$ such that the restrictions of $f$ and $g$ to $X_{>x_1}$ coincide.
We argue by induction on $x_1$.

Suppose first that $x_1$ is in the domain of $f$ but not in the domain of $g$.
If $f(x_1)$ is not in the image of $g$, we extend $g$ by setting $g'(x_1)=f(x_1)$.
Otherwise $f(x_1)=g(x_2)$, necessarily with $x_2<x_1$ and we define $g'$ by $g'(x)=g(x)$ for all $x\ne x_2$
in the domain of $g$ and $g'(x_1)=f(x_1)$.

Suppose now that $x_1$ is in the domain of $g$.
Then from the definition of $f$ it is clear that $x_1$ is in the domain of $f$ and $f(x_1)<g(x_1)$.
Let $y_1=g(x_1)$ and $y_2=f(x_1)$.
If $y_2$ is not in the image of $g$ we modify $g$ by setting $g'(x_1)=y_2$ and $g'(x)=g(x)$ for all $x\ne x_1$ in the domain of $g$.
Otherwise $y_2=g(x_2)$, necessarily with $x_2<x_1$ and we define $g'$ in the domain of $g$ by $g'=g\circ\sigma$ where $\sigma$ is the transposition
switching $x_1$ and $x_2$.
We claim that $g'$ is a matching, i.e. $g'(x)\rltn x$ for all $x$ in the domain of $g'$. This is clear if $x\ne x_2$ while
for $x=x_2$ this follows from the condition \eqref{eq: specrelation}.

Thus, in all cases we can find a matching $g'$ whose domain is at least as large as the domain of $g$ such that either $x_1(g')<x_1(g)$
or $g'$ coincides with $f$ (if $x_1(g)$ is the minimal element of $X$).
This proves the induction step as well as the base of the induction.
\end{proof}

By a definition, a \emph{ray} in $X$ is either $X$ itself or a subset of $X$ of the form $X_{>x_0}$ for some $x_0\in X$.
The following lemma is clear from the construction:

\begin{lem} \label{lem: matchinstages}
Suppose that $\rltn$ is \trns\ and let $f$ be the best matching between $X$ and $Y$. Let $A$ be a ray in $X$ and let $B$ be its complement.
Let $g$ be the best matching between $A$ and $Y$ and let $C$ be the complement of the range of $g$.
Let $h$ be best matching between $B$ and $C$.
Then $g$ (resp., $h$) is the restriction of $f$ to $A$ (resp., $B$).
In particular, $f$ is a function on $X$ if and only if $g$ and $h$ are functions on $A$ and $B$ respectively.
\end{lem}

Another trivial remark is the following:
\begin{rem} \label{rem: trivmatch}
Suppose that $\rltn$ is \trns\ and let $f$ be the best matching between $X$ and $Y$. Suppose that $A$ is a subset of $X$ and $B$ is a subset of $Y$ such that
for any $x\in X$ and $y\in Y$, $y\rltn x$ implies $x\in A$ and $y\in B$. Then the domain of $f$ is contained in $A$, the range of $f$ is contained in $B$
and $f$ is the best matching between $A$ and $B$.
\end{rem}

Finally, we will use the following:
\begin{lem} \label{lem: addonematch}
Let $\rltn$ be a \trns\ relation between $X$ and $Y'$ (suitably ordered) and let $f'$ be the best matching between $X$ and $Y'$.
Let $y_0$ be the maximal element of $Y'$, let $Y=Y'\setminus\{y_0\}$ and let $f$ be the best matching between $X$ and $Y$.
Assume that for any $x\in X$ and $y\in Y$, if $y\rltn x$ then $y_0\rltn x$.
Let $a$ be the size of the complement $C$ of the range of $f$ and let $a'$ be the size of the complement $C'$ of the range of $f'$.
Then $a\le a'\le a+1$.
Moreover, $a'=a+1$ if and only if the domain of $f$ is $\{x\in X:y_0\rltn x\}$. In this case
$C'=C\cup\{y_0\}$ and $f'$ coincides with $f$.
\end{lem}

\begin{proof}
Since $f$ is a matching between $X$ and $Y'$ and the intersection of $f'$ with $X\times Y$ is a matching between $X$ and $Y$, it follows from
Lemma \ref{lem: optf} that $a\le a'\le a+1$.
By our condition, the domain of any matching between $X$ and $Y$ is contained in $\{x\in X:y_0\rltn x\}$.
It is also clear from the definition of the best matching that if the domain of $f$ is $\{x\in X:y_0\rltn x\}$ then $f'$ coincides with $f$
and $C'=C\cup\{y_0\}$.
Conversely, if $y_0\rltn x'$ and $x'$ is not in the domain of $f$ then we can extend $f$ by $f(x')=y_0$
to get a strictly bigger matching between $X$ and $Y'$. Hence we have $a'\le a$.
\end{proof}

\subsection{The case when $\Delta$ is supercuspidal}
Let $\rho\in\Cusp$ and $\m=\Delta_1+\dots+\Delta_N\in\MS$.
Write $Y_{\rho;\m}=\{i:b(\Delta_i)=\rho\}$ and $X_{\rho;\m}=\{i:b(\Delta_i)=\rshft{\rho}\}$.
We define $\rltn$ (or $\rltn_\m$ if $\m$ is not clear from the context) to be the relation between $Y_{\rho;\m}$ and $X_{\rho;\m}$ given by
\[
j\rltn i\text{ if and only if }\Delta_j\prec\Delta_i.
\]
We assume that the $\Delta_i$'s are ordered such that $\Delta_{i_1}\supset\Delta_{i_2}$ whenever $i_1\le i_2$ and either
$i_1,i_2\in X_{\rho;\m}$ or $i_1,i_2\in Y_{\rho;\m}$ (\textit{e.g.}, we could take either the left or right {\ordered} forms of $\m$).
Then $\rltn$ is \trns\ with respect to the standard ordering on $X_{\rho;\m}$ and $Y_{\rho;\m}$.
In fact for any $i_1\in X_{\rho;\m}$ and $j\in Y_{\rho;\m}$ such that
$j\rltn i_1$ and for any $i_2\in X_{\rho;\m}$ with $i_2\le i_1$ we have $j\rltn i_2$.

The following is a restatement of results obtained independently by M\'inguez and Jantzen.

\begin{thm}{\cite{MR2527415}*{Th\'eor\`eme 7.5}, \cite{MR2306606}*{Theorem 2.2.1}}\label{thm: cusptimes}
Let $\rho \in \Cusp$ and $\m \in \MS$.
Let $f$ be the best matching between $X_{\rho;\m}$ and $Y_{\rho;\m}$. Then
\begin{enumerate}
\item $\soc(\rho\times\zele{\m})=\zele{\m+\{\rho\}}$ if and only if $f$ is a function from $X_{\rho;\m}$ to $Y_{\rho;\m}$.
\item If $f$ is not a function from $X_{\rho;\m}$ to $Y_{\rho;\m}$ and $i\in X_{\rho;\m}$ is the minimal index which does not belong to the domain of $f$
then $\soc(\rho\times\zele{\m})=\zele{\m-\Delta_i+\,^+\Delta_i}$.
\item More generally, let $b$ be the size of the complement $B$ of the domain of $f$ in $X_{\rho;\m}$ and
write $B=\{l_1,\dots,l_b\}$ with $l_1<\dots<l_b$. Then for any $a\ge0$,
$\soc(\rho^{\times a}\times\zele{\m})=\zele{\m'}$
where
\[
\m'=\begin{cases}\m-\sum_{i=1}^a\Delta_{l_i}+\sum_{i=1}^a\,^+\Delta_{l_i}&\text{if }a\le b,\\
\m-\sum_{i=1}^b\Delta_{l_i}+\sum_{i=1}^b\,^+\Delta_{l_i}+\overbrace{\rho+\dots+\rho}^{b-a\text{ times}}&\text{  otherwise.}
\end{cases}
\]
\item Let $a$ be the size of the complement $C$ of the range of $f$ in $Y_{\rho;\m}$.
Write $C=\{i_1,\dots,i_a\}$ with $i_1<\dots<i_a$. Then
\label{eq: extrctrho} the integer $a$ is the maximal $b \geq 0$ such that
there exists $\sigma \in \Irr$ such that $\zele{\m}$ is a subrepresentation of $\rho^{\times b}\times\sigma$.
Moreover, this $\sigma$ is unique and is given by $\zele{\m-\sum_{j=1}^b\Delta_{i_j}+\sum_{j=1}^b\,^-\Delta_{i_j}}$.
Thus, $\zele{\m}_{\{\rho\}}=\rho^{\times a}$ and
\[
\zele{\m}_{\lnot\{\rho\}}=\zele{\m-\sum_{i\in C}\Delta_i+\sum_{i\in C}\,^-\Delta_i}.
\]
In particular $\zele{\m}$ is left $\rho$-reduced if and only if $a=0$, i.e., if and only if the inverse of $f$ is a function
from $Y_{\rho;\m}$ to $X_{\rho;\m}$.
\end{enumerate}
\end{thm}

The theorem is proved using only Jacquet modules in \cite{MR2527415} and \cite{MR2306606}.
In fact, it is easy to see that Theorem \ref{thm: cusptimes} follows from Proposition \ref{prop:  delsig} and Corollary \ref{cor:  inxpi'}
together with the M\oe glin--Waldspurger algorithm.

\subsection{The combinatorial setup for general $\Delta$}

Now let $\Delta \in \seg$ and $\m=\Delta_1+\dots+\Delta_N \in \MS$.
(For convenience we will also allow $\Delta_i=\emptyset$. This is of course immaterial.)
Let $X_{\Delta;\m}=\{i:\Delta\prec\Delta_i\}$ and $Y_{\Delta;\m}=\{i:\lshft\Delta\prec\Delta_i\}$.
(This is consistent with the notation of the previous subsection.)
As before let $\rltn$ (or $\rltn_\m$) be the relation
\[
j\rltn i\iff \Delta_j\prec\Delta_i
\]
between $Y_{\Delta;\m}$ and $X_{\Delta;\m}$.
We write $\llcond(\Delta,\m)$ for the condition that there exists a matching function from $X_{\Delta; \m}$ to $Y_{\Delta; \m}$,
that is an injective function $f:X_{\Delta;\m}\rightarrow Y_{\Delta;\n}$ such that $f(x)\rltn x$ for all $x$.

Endow $X_{\Delta;\m}$ with a total ordering $\le_X$ such that
\begin{equation} \label{eq: orderX}
i_1\le_X i_2\ \implies\ \Delta_{i_1}\ge_b\Delta_{i_2}.
\end{equation}
Endow $Y_{\Delta;\m}$ with a total ordering $\le_Y$ such that
\begin{equation} \label{eq: orderY}
j_1\le_Y j_2\ \implies e(\Delta_{j_1})\ge e(\Delta_{j_2}).
\end{equation}
(Note that unlike in the case $\Delta=\{\rho\}$ we cannot guarantee that $\le_X$ and $\le_Y$ come
from the \emph{same} ordering of the segments of $\m$.)
It is easy to check that $\rltn$ is \trns\ with respect to such orderings.
In fact,
\begin{multline*}
\forall i_1,i_2\in X_{\Delta;\m}, \forall j_1,j_2\in Y_{\Delta;\m}\
i_1\ge_X i_2\ \&\ j_1\ge_Y j_2\ \&\ j_1\rltn i_1\ \&\  j_2\rltn i_2\\\ \implies j_1\rltn i_2.
\end{multline*}
Thus by Lemma \ref{lem: optf}, $\llcond(\Delta,\m)$ is satisfied if and only if the best matching between $X_{\Delta;\m}$ and $Y_{\Delta;\m}$
(with respect to $\le_X$ and $\le_Y$) is a function from $X_{\Delta;\m}$ to $Y_{\Delta;\m}$.

\subsection{Criterion for the irreducibility of $\zele{\Delta} \times \sigma$}

\begin{prop} \label{prop: onesegment}
For any $\Delta\in\seg$ and $\m\in\MS$
the conditions $\lcond(\zele{\Delta},\zele{\m})$ and $\llcond(\Delta,\m)$ are equivalent.
\end{prop}

\begin{proof}
First, we may assume that $\Delta\in\seg_\rho$ and $\m\in\MS_\rho$ for some $\rho\in\Cusp$.
We prove the proposition by induction on the length of $\Delta$.
The case where $\Delta$ is supercuspidal is covered by Theorem \ref{thm: cusptimes}.
For the induction step we first reduce to the case where $\Delta_i\ge_b\Delta$ for all $i$.
Let $\sigma=\zele{\m}$ and let $\sigma_i$, $i=1,2$ be as in Lemma \ref{lem: redgem} with $\m=\Delta$.
Clearly, $X_{\Delta;\m_{\ge_b\Delta}}=X_{\Delta;\m}$ and $Y_{\Delta;\m_{\ge_b\Delta}}=Y_{\Delta;\m}$.
Therefore the conditions $\llcond(\Delta,\m)$ and $\llcond(\Delta,\m_{\ge_b\Delta})$ are equivalent.
(Note also that $X_{\Delta;\m_{<_b\Delta}}=Y_{\Delta;\m_{<_b\Delta}}=\emptyset$.)
On the other hand, it follows from Lemma \ref{lem: redgem} that the conditions $\lcond(\zele{\Delta},\sigma)$ and $\lcond(\zele{\Delta},\sigma_1)$ are also equivalent.

Thus, we may assume that $\Delta_i\ge_b\Delta$ for all $i$.
For convenience assume that $\Delta_1\ge_b\dots\ge_b\Delta_N$.
Let $\rho=b(\Delta)$.
By assumption, $Y_{\rho;\m}\subset Y_{\Delta;\m}$.
Fix a total order $\le_Y$ on $Y_{\Delta;\m}$ satisfying \eqref{eq: orderY}.
Let $g$ be the best $\rltn_\m$-matching between $X_{\rho;\m}$ and $Y_{\rho;\m}$
(where we take the standard order on $X_{\rho;\m}$ and $\le_Y$ on $Y_{\rho;\m}$).
Since $Y_{\rho;\m}\subset Y_{\Delta;\m}$, the domain of $g$ is contained in $X_{\rho;\m}\cap X_{\Delta;\m}$.
Thus, $g$ is the best matching between $X_{\rho;\m}\cap X_{\Delta;\m}$ and $Y_{\rho;\m}$, which is also the best matching between
$X_{\rho;\m}\cap X_{\Delta;\m}$ and $Y_{\Delta;\m}$ (Remark \ref{rem: trivmatch}).
Let $C$ be the complement of the range of $g$ in $Y_{\rho;\m}$ and let $a$ be the size of $C$.
Recall that $\sigma_{\{\rho\}}=\rho^{\times a}$ and $\sigma_{\lnot\{\rho\}}=\zele{\m'}$
where $\m'=\Delta'_1+\dots+\Delta'_N$ is given by
\[
\Delta'_i=\begin{cases}^-\Delta_i&\text{if }i\in C,\\\Delta_i&\text{otherwise.}\end{cases}
\]
Note that $X_{\,^-\Delta;\m'}=X_{\Delta;\m}\setminus X_{\rho;\m}$ and $Y_{\,^-\Delta;\m'}=
(Y_{\Delta;\m}\setminus Y_{\rho;\m})\cup C$.
Moreover, the restriction of $\le_Y$ to $Y_{\,^-\Delta;\m'}$ still satisfies the analogue of \eqref{eq: orderY} for the $\Delta'_i$'s
and the intersection of $\rltn_\m$ with $Y_{\,^-\Delta;\m'}\times X_{\,^-\Delta;\m'}$ is $\rltn_{\m'}$.
Let $h$ be best $\rltn_{\m'}$-matching between $X_{\,^-\Delta;\m'}$ and $Y_{\,^-\Delta;\m'}$.

Note that $X_{\rho;\m}\cap X_{\Delta;\m}$ is a ray of $X_{\Delta;\m}$.
Thus by Lemma \ref{lem: matchinstages}, $\llcond(\Delta,\m)$ is equivalent to the statement that
the domain of $g$ is $X_{\rho;\m}\cap X_{\Delta;\m}$ and the domain of $h$ is $X_{\,^-\Delta;\m'}$.
By induction hypothesis, the domain of $h$ is $X_{\,^-\Delta;\m'}$ if and only if
\[
\zele{\m''}=\soc(\zele{\,^-\Delta}\times\zele{\m'})
\]
where $\m''=\m'+\,^-\Delta=\Delta'_1+\dots+\Delta'_N+\,^-\Delta=\Delta''_1+\dots+\Delta''_{N+1}$.
Recall that by Proposition \ref{prop:segnon}
\[
\soc(\zele{\Delta}\times\zele{\m})=\soc(\rho^{\times(a+1)}\times\soc(\zele{\,^-\Delta}\times\zele{\m'})).
\]
Thus, to prove the proposition it will suffice to show that $\zele{\m+\Delta}\hookrightarrow\rho^{\times(a+1)}\times\sigma'$
for some $\sigma'\in\Irr$ if and only if the domain of $g$ is $X_{\rho;\m}\cap X_{\Delta;\m}$ and in this case
$\sigma'=\zele{\m''}$.

Clearly $X_{\rho;\m+\Delta}=X_{\rho;\m}$ and $Y_{\rho;\m+\Delta}=Y_{\rho;\m}\cup\{N+1\}$.
Let $g'$ be the best $\rltn_{\m+\Delta}$-matching between $X_{\rho;\m+\Delta}$ and $Y_{\rho;\m+\Delta}$
(where we extend $\le_Y$ by letting $N+1$ be the largest element).
By Lemma \ref{lem: addonematch} the complement $C'$ of the range of $g'$ in $Y_{\rho;\m+\Delta}$ is of size $a+1$
if and only if the domain of $g$ is $X_{\rho;\m}\cap X_{\Delta;\m}$ and in this case $C'=C\cup\{N+1\}$ and $g'$ coincides with $g$.
Our claim therefore follows from Theorem \ref{thm: cusptimes} part \ref{eq: extrctrho}.
\end{proof}



Symmetrically, define
\begin{align*}
\tilde{X}_{\Delta; \m}&=\{i:\Delta_i\prec\Delta\},\\
\tilde{Y}_{\Delta; \m}&=\{i:\lshft\Delta_i\prec\Delta\}.
\end{align*}
We write $\rrcond(\Delta,\m)$ for the condition that there exists a $\rltn$-matching from $\tilde{X}_{\Delta;\m}$
to $\tilde{Y}_{\Delta; \m}$ where $\rltn$ is the relation
\[
j\rltn i\iff \Delta_i\prec\Delta_j
\]
between $\tilde{Y}_{\Delta; \m}$ and $\tilde{X}_{\Delta; \m}$.

By passing to the contragredient we get:
\begin{prop} \label{prop: onesegment2}
For $\Delta\in\seg$ and $\m\in\MS$ the conditions $\rcond(\zele{\Delta},\zele{\m})$ and $\rrcond(\Delta,\m)$
are equivalent.
\end{prop}

From Propositions \ref{prop: onesegment} and \ref{prop: onesegment2} and Lemma \ref{lem: lcondprop}  part \ref{part: irred} we infer
\begin{coro}\label{cor:irred2}
Let $\Delta\in\seg$ and $\m\in\MS$.
Then $\zele{\Delta}\times\zele{\m}$ is irreducible (in which case
$\zele{\Delta}\times\zele{\m} \simeq \zele{\Delta+\m}$)
if and only if $\llcond(\Delta,\m)$ and $\rrcond(\Delta,\m)$.
\end{coro}

The proofs in this section never use the condition $s_\rho>0$.
Invoking the principle of ``mirror symmetry'' (cf.~\S\ref{sec: mirror symmetry}) we get:

\begin{prop}
For $\Delta\in\seg$ and $\m\in\MS$. Then
\begin{enumerate}
 \item $\lang{\Delta+\m}\hookrightarrow\lang{\Delta}\times\lang{\m}$
 (i.e., $\lang{\Delta+\m}=\soc(\lang{\Delta}\times\lang{\m})$) if and only if $\rrcond(\Delta,\m)$.
\item $\lang{\Delta}\times\lang{\m}\twoheadrightarrow\lang{\Delta+\m}$ (i.e.,
$\lang{\Delta+\m}=\cos(\lang{\Delta}\times\lang{\m})$) if and only if $\llcond(\Delta,\m)$.
\item The following conditions are equivalent:
\begin{enumerate}
\item $\lang{\Delta}\times\lang{\m}$ is irreducible.
\item $\llcond(\Delta,\m)$ and $\rrcond(\Delta,\m)$.
\item $\lang{\Delta}\times\lang{\m}\simeq\lang{\m}\times\lang{\Delta}$.
\item $\lang{\Delta}\times\lang{\m}\simeq\lang{\m+\Delta}$.
\end{enumerate}
\end{enumerate}
\end{prop}

Of course, the equivalence of the irreducibilities of $\lang{\Delta}\times\lang{\m}$ and $\zele{\Delta}\times\zele{\m}$
can also be deduced directly from the fact that the Zelevinsky involution is a ring homomorphism.

\section{Applications} \label{sec: applications}

\label{sec: nece}

\subsection{A sufficient condition for the irreducibility of $\pi \times \sigma$}

As an application of Proposition \ref{prop:  onesegment}, we will give a sufficient condition for irreducibility
of $\pi\times\sigma$, improving the main result of \cite{MR3237446} (see Corollary \ref{cor:mainBLM} below).

\begin{prop}\label{prop:sum}
Let  $\m, \n \in \MS$. Suppose that either
\begin{subequations}
\begin{equation} \label{eq: alldeltam'l}
\text{for every segment $\Delta \in \m$  we have $\llcond(\Delta,\n)$}
\end{equation}
or
\begin{equation} \label{eq: alldelta'mr}
\text{for every segment $\Delta' \in \n$ we have  $\rrcond(\Delta',\m)$.}
\end{equation}
\end{subequations}
Then $\lcond(\zele{\m},\zele{\n})$ holds.

Similarly, suppose that either
\begin{subequations}
\begin{equation} \label{eq: alldeltam'r}
\text{for every segment $\Delta \in \m$ we have  $\rrcond(\Delta,\n)$}
\end{equation}
or
\begin{equation} \label{eq: alldelta'ml}
\text{for every segment $\Delta' \in \n$ we have  $\llcond(\Delta',\m)$.}
\end{equation}
\end{subequations}
Then $\rcond(\zele{\m},\zele{\n})$ holds.

Thus, if either \eqref{eq: alldeltam'l} or \eqref{eq: alldelta'mr} and if either \eqref{eq: alldeltam'r} or \eqref{eq: alldelta'ml}
then $\zele{\m}\times\zele{\n}$ and $ \lang{\m} \times \lang{\n}$ are irreducible.
\end{prop}

\begin{proof}
Suppose that for every segment $\Delta$ in $\m$ we have $\llcond(\Delta,\n)$.
We will prove the statement by induction on the size of $\m$.
The case where $\m$ is a single segment follows from Propositions \ref{prop: onesegment}.
For the induction step we can assume as usual that $\m,\n\in\MS_\rho$ for some $\rho\in\Cusp$.
Let $\Delta$ (resp., $\Delta'$) be the smallest segment in $\m$ (resp., $\n$) with respect to $\le_b$.
Note that for any $\Delta''$ in $\m$ the condition $\llcond(\Delta'',\n_{\ge_b\Delta})$ is equivalent to (hence follows from) $\llcond(\Delta'',\n)$.
Thus, by Lemma \ref{lem: redgem} we reduce to the case where $\Delta'\ge_b\Delta$.

Assume therefore that $\Delta'\ge_b\Delta$.
Note that $\zele{\m} \times \zele{\n} $ is a submodule of $  \zele{\m-\Delta}   \times\zele{\Delta}   \times \zele{\n} $.
As $\Delta \le_b \Delta'$ we have $\rrcond(\Delta,\n)$ by Lemma \ref{lem: lcondprop} part \ref{part: trivlcond}.
By hypothesis we also have $\llcond(\Delta,\n)$.
We deduce that $\zele{\Delta}\times \zele{\n} $ is irreducible, hence isomorphic to $\zele{\n}\times\zele{\Delta}$. Thus,
\[
\zele{\m} \times \zele{\n} \hookrightarrow \zele{\m-\Delta}   \times \zele{\n}   \times\zele{\Delta}.
\]
By induction hypothesis and Lemma \ref{lem: lcondprop} part \ref{part: referee}, the latter is a submodule of
\[
\std(\m-\Delta+ \n)\times\zele{\Delta}
\]
which is isomorphic to $\std(\m +\n)$ since $\Delta \le_b \Delta'$.

The statement for the condition \eqref{eq: alldeltam'r} (resp., \eqref{eq: alldelta'ml}) follows the previous case upon passing to the contragredient
(resp., by applying Lemma \ref{lem: sym} and interchanging $\m$ and $\n$).
Similarly, the statements for \eqref{eq: alldelta'mr} and \eqref{eq: alldelta'ml} are equivalent by passing to the contragredient.

For $\lang{\m}\times\lang{\n}$ we can argue using \S\ref{sec: mirror symmetry}
(or the Zelevinsky involution).
\end{proof}

As a special case we get:

\begin{coro}\label{coro:sum}
Let $\m, \n \in \MS$.
Assume that either $\zele{\Delta'}\times\zele{\m}$ is irreducible for every segment $\Delta' \in \n$
or $\zele{\Delta}\times\zele{\n}$ is irreducible for every segment $\Delta \in \m$.
Then $\zele{\m}\times\zele{\n}$ and $ \lang{\m} \times \lang{\n} $ are irreducible.
\end{coro}

\begin{coro} \label{coro: gentimesany}
Let $\m,\n\in\MS$. Assume that $\n$ is \TU.
Then $ \zele{\m} \times \zele{\n} $ (or $\lang{\m}\times\lang{\n}$) is irreducible if and only if for every segment $\Delta'$ in $\n$ we have
$\llcond(\Delta',\m)$ and $\rrcond(\Delta',\m)$ (i.e., $\zele{\m}\times\zele{\Delta'}$ is irreducible).
\end{coro}

In particular, we get a simple necessary and sufficient criterion for the irreducibility of $\pi \times \sigma$ when
$\pi\in\Irr$ is either unramified, tempered or (in the case $D=F$) generic, and $\sigma\in\Irr$ is arbitrary.

\begin{proof}
The sufficiency follows from Proposition \ref{prop:sum}.
The necessity follows from Proposition \ref{useful_properties} part \ref{part: pairwise unlinked}.
\end{proof}

Using Corollary \ref{cor: BLMoneseg} we can also recover the main result of \cite{MR3237446} as a special case.
\begin{coro}{\cite{MR3237446}*{Th\'eor\`eme 3.9}}\label{cor:mainBLM}
Let $\m,\n \in \MS$. Suppose that for every $\Delta \in \m$ and every $\Delta' \in \n$, $\Delta$ and $\Delta'$ are not
juxtaposed. Then
\[
\zele{\m} \times \lang{\n}
\]
is irreducible.
\end{coro}

\subsection{Applications to Speh and saturated representations} \label{sec: Speh}

We go back to saturated representations (see Definition \ref{def:saturated}).

\begin{lem} \label{lem: onesegspeh}
Let $\pi=\zele{\m}$, $\m=\Delta_1+\dots+\Delta_N$ be a Speh representation with $\Delta_{i+1}=\lshft{\Delta_i}$, $i=1,\dots,N-1$
and let $\Delta\in\seg$. Then $\llcond(\Delta,\m)$ is equivalent to $\Delta\not\prec\Delta_N$ and
$\rrcond(\Delta,\m)$ is equivalent to $\Delta_1\not\prec\Delta$.
Thus, $\zele{\Delta}\times\pi$ is irreducible if and only if $\Delta\not\prec\Delta_N$ and $\Delta_1\not\prec\Delta$.
In particular, this happens if $\Delta$ is not linked to $\supp\m$.
\end{lem}

\begin{proof}
It is clear that $\llcond(\Delta,\m)$ cannot hold if $\Delta\prec\Delta_N$.
On the other hand, if $\Delta\not\prec\Delta_N$ then we can take the matching function $f(i)=i+1$.
Similarly for $\rrcond(\Delta,\m)$ (with $f(i)=i-1$).
The last part follows from Corollary \ref{cor:irred2}.
\end{proof}

From Proposition \ref{prop:sum} we get:
\begin{prop}\label{prop:Speh is saturated}
Let $\pi=\zele{\Delta_1+\dots+\Delta_N}$ be a Speh representation with $\Delta_{i+1}=\lshft{\Delta_i}$, $i=1,\dots,N-1$
and let $\sigma=\zele{\n}\in\Irr$.
Assume that $\Delta_1\not\prec\Delta$ for any segment $\Delta$ in $\n$.
Then $\lcond(\pi,\sigma)$ holds. In particular, this is satisfied if $\supp\sigma\subset\Z_{\le e(\pi)}$
Similarly, if $\Delta\not\prec\Delta_N$ for any $\Delta$ in $\n$ (e.g., if $\supp\sigma\subset\Z_{\ge b(\pi)}$)
then $\rcond(\pi,\sigma)$ holds.
In particular, $\pi$ is saturated and hence \RI.
\end{prop}

We can immediately generalize Proposition \ref{prop: delsig} to Speh representations.

\begin{coro}\label{coro:subspeh}
Let $\pi=\zele{\m}\in\Irr_\rho$ be a Speh representation and $\sigma=\lang{\n}\in\Irr_\rho$ for some $\rho\in\Cusp$.
Write $\n_1=\n_{\le_e e(\pi)}$ and $\n_2=\n_{\ge_b b(\pi)}$.
 Then
\begin{enumerate}
\item $\soc(\pi\times \sigma)=\lang{(\n_1^t+\m)^t+\n_{>_ee(\pi)}}$
 \item $\soc(\pi\times\sigma)=\lang{\n+\m^t}$ if and only if $(\n_1^t+\m)^t=\n_1+\m^t$.
 \item $\cos(\pi\times\sigma)=\lang{\n+\m^t}$ if and only if $(\n_2^t+\m)^t=\n_2+\m^t$.
 \item $\pi\times\sigma$ is irreducible if and only if $(\n_1^t+\m)^t=\n_1+\m^t$ and
 $(\n_2^t+\m)^t=\n_2+\m^t$.
 \item $\pi\times\sigma$ is irreducible if and only if $\pi\times\sigma=\sigma\times\pi$.
\end{enumerate}
\end{coro}

Indeed, the first part follows from Lemma \ref{lem: eff stages} and Proposition \ref{prop:Speh is saturated}.
The second part follows immediately. The third part follows by passing to the contragredient.
The fourth part is an immediate consequence, while the last part now follows from Lemma \ref{lem: sym}.

Similarly, we can generalize Corollary \ref{cor: inxpi'}.
\begin{coro}
Let $\pi=\zele{\m}\in\Irr_\rho$ be a Speh representation and let $\pi'=\lang{\n}\in\Irr_\rho$.
Then there exists $\sigma\in\Irr$ such that $\pi'=\soc(\pi\times\sigma)$ (i.e., $\pi'\hookrightarrow\pi\times\sigma$)
if and only if there exists $\m'\in\MS$ such that $(\n_{\le_e e(\pi)})^t=\m+\m'$. In this case $\sigma=\lang{(\m')^t+\n_{>_e e(\pi)}}$.
In particular, $\sigma$ is uniquely determined by $\pi$ and $\pi'$.
\end{coro}

Let us now analyze irreducibility of parabolic induction of Speh representations.
In the case $D=F$, a sufficient condition for the irreducibility of $\pi_1\times\dots\times\pi_k$ where $\pi_i$
are Speh representations was proved in \cite{MR1026752}*{\S I.9} and it is one of the main local results of [ibid.].
This was extended in \cite{MR3224483} to any $D$, where this condition is also shown to be necessary.
We can easily recover these irreducibility results.
First we reduce to the case $k=2$:

\begin{coro} \label{coro: irredspeh}
Let $\pi_1,\dots,\pi_k\in\Irr$. Suppose that all but at most one of the $\pi_i$'s are Speh representations.
Then $\pi_1\times\dots\times\pi_k$ is irreducible if and only if $\pi_i\times\pi_j$ is irreducible for all $i<j$.
\end{coro}

Note that in the case where all but at most one of the $\pi_i$'s are segment representations this is precisely Corollary \ref{coro: gentimesany}.

\begin{proof}
The only non-trivial direction is the `if' part. We will prove the assertion by induction on $k$.
The case $k=1$ is trivial. For the induction step we may assume without loss of generality that $\pi_k$ is a Speh representation.
By induction hypothesis $\Pi:=\pi_1\times\dots\times\pi_{k-1}$ is irreducible.
Also, since $\pi_k\times\pi_i$ is irreducible for all $i<k$, we have $\pi_k\times\pi_i\simeq\pi_i\times\pi_k$
and therefore $\pi_k\times\Pi\simeq\Pi\times\pi_k$.
By the last part of Corollary \ref{coro:subspeh} we conclude that $\Pi\times\pi_k$ is irreducible as required.
\end{proof}

Finally, it remains to consider the irreducibility of the product of two Speh representations.
More precisely, we explicate the conditions $\lcond(\pi,\sigma)$ and $\rcond(\pi,\sigma)$.
\begin{coro} \label{coro: irredtwospeh}
Let $\pi=\zele{\Delta_1+\dots+\Delta_N}$ and $\sigma=\zele{\Delta'_1+\dots+\Delta'_{N'}}$ be two Speh representations
supported on the same cuspidal line, with $\Delta_{i+1}=\lshft{\Delta_i}$, $i=1,\dots,N-1$ and
$\Delta'_{i+1}=\lshft{\Delta'_i}$, $i=1,\dots,N'-1$.
Then $\lnot\lcond(\pi,\sigma)$ (the negation of $\lcond(\pi,\sigma)$) is equivalent to the condition
\begin{multline} \label{eq: TadicSpeh}
\supp\pi\prec\supp\sigma\ \&\ b(\Delta_1)<b(\Delta_1')\ \&\ e(\Delta_1)<e(\Delta_1')\\
\ \&\ b(\Delta_N)<b(\Delta_{N'})\ \&\ e(\Delta_N)<e(\Delta_{N'})
\end{multline}
while $\lnot\rcond(\pi,\sigma)$ is equivalent to the symmetric condition (exchanging $\pi$ and $\sigma$).
Thus, $\pi\times\sigma$ is reducible if and only if either \eqref{eq: TadicSpeh} or the symmetric condition is satisfied.
In particular, $\pi\times\sigma$ is irreducible if $\supp\pi$ and $\supp\sigma$ are not linked.
\end{coro}

\begin{proof}
In view of Lemma \ref{lem: onesegspeh}, the condition \eqref{eq: alldeltam'l} is not satisfied if and only if $\Delta_i\prec\Delta'_{N'}$
for some $i$, or equivalently if $b(\Delta_N)<b(\Delta'_{N'})$ and $e(\Delta_N)<e(\Delta'_{N'})$ and $b(\Delta'_{N'})\le e(\rshft{\Delta_1})$.
Similarly, \eqref{eq: alldelta'mr} is not satisfied if and only if $b(\Delta_1)<b(\Delta'_1)$
and $e(\Delta_1)<e(\Delta'_1)$ and $b(\Delta'_{N'})\le e(\rshft{\Delta_1})$.
From Proposition \ref{prop:sum} it follows that if \eqref{eq: TadicSpeh} is not satisfied then $\lcond(\pi,\sigma)$ holds.
Similarly, if the symmetric condition to \eqref{eq: TadicSpeh} fails then $\rcond(\pi,\sigma)$ holds.

Suppose now that \eqref{eq: TadicSpeh} is satisfied. Then obviously the symmetric condition is not satisfied and therefore
$\rcond(\pi,\sigma)$ holds. Thus, if $\lcond(\pi,\sigma)$ also held then $\pi\times\sigma$ would be irreducible
and by Corollary \ref{coro:subspeh} we would get (writing $\pi=\zele{\m}$, $\sigma=\zele{\n}$) $\m^t+\n^t=(\m+\n)^t$.
However, it is easy to check that this combinatorial condition is not fulfilled in this case -- see \cite{MR3224483}*{\S5}.
(Alternatively we can use Propositions \ref{prop: laddercond} and Lemma \ref{lem: twoladdercomb} below.)
Similarly, if the symmetric condition to \eqref{eq: TadicSpeh} is satisfied then $\rcond(\pi,\sigma)$ fails.
\end{proof}

In a special case we can go further:
\begin{coro}\label{coro:neededU0}
Suppose that $\pi=\pi_1\times\dots\times\pi_k$ where $\pi_i$ are Speh representations and $\supp\pi_i$ are mutually unlinked.
Then $\pi$ is \RI.
\end{coro}

\begin{proof}
The proof is analogous to that of Corollary \ref{coro:generic is RI}.
Let $\pi=\pi_1\times\dots\times\pi_k$ where $\pi_i$ are Speh representations
where $\supp\pi_i$ are mutually unlinked.
First note that $\pi$ is irreducible by Corollaries \ref{coro: irredspeh} and \ref{coro: irredtwospeh}.
To show that $\pi$ is \RI\ we argue by induction on $k$. We already proved the case $k=1$.
For the induction step assume without loss of generality that $\supp\pi_k$ is the longest segment
among $\supp\pi_i$, $i=1,\dots,k$.
Thus, $\supp\pi_k$ is a connected component of $\supp\pi$ (see \S\ref{sec: segment}).
Let $\sigma\in\Irr$ be such that $\supp(\sigma)\subset\supp(\pi)$. If $\sigma=\zele{\m}$
then any segment $\Delta$ in $\m$ is unlinked with $\supp(\pi_k)$. Therefore
$\pi_k\times\zele{\Delta}$ is irreducible by Lemma \ref{lem: onesegspeh}.
Thus, $\pi_k\times\sigma$ is irreducible by Corollary \ref{coro:sum}.
Therefore, by induction hypothesis $\pi\times\sigma=\pi_1\times\dots\times\pi_{k-1}\times(\pi_k\times\sigma)$ is \EI.
The corollary follows from Lemma \ref{lem: suffuniquesub}.
\end{proof}

Finally, we can classify saturated representations.

\begin{prop}\label{prop:saturated}
A representation $\pi\in \Irr$ is saturated if and only if we can write $\pi=\pi_1\times\dots\times\pi_k$ where for all $i$ $\pi_i$
is a Speh representation and $\supp(\pi_i)$ is a connected component of $\supp(\pi)$.
\end{prop}

\begin{proof}
For the `if' part let $\pi=\pi_1\times\dots\times\pi_k$ be as in the statement of the proposition.
We already saw that $\pi$ is irreducible.
Let $\Delta\in\seg$ be such that $\Delta\subset\supp(\pi)$.
Then by the condition on $\pi$, $\Delta$ and $\supp(\pi_i)$ are unlinked for all $i$.
Hence $\zele{\Delta}\times\pi_i$ is irreducible by Lemma \ref{lem: onesegspeh}.
It follows from Corollary \ref{coro: irredspeh} that $\zele{\Delta}\times\pi$ is irreducible.
It follows now from Corollary \ref{coro:sum} that $\pi$ is saturated.

Conversely, suppose that $\pi=\zele{\m}$ is saturated.
We denote by $n(\Delta)$ the multiplicity of a segment $\Delta$ in $\m$.
We show that for any $\Delta\subset\supp\pi$ such that $\lshft\Delta\subset\supp\pi$ we have $n(\lshft\Delta)=n(\Delta)$.
(This will clearly imply that $\pi$ has the required form.)
Indeed, if $\m=\Delta_1+\dots+\Delta_N$ and $f$ is a $\rltn$-matching from
$X_{\lshft\Delta;\m}$ to $Y_{\lshft\Delta;\m}$ then $f$ maps $\{i:\Delta_i=\Delta\}$
to $\{i:\Delta_i=\lshft\Delta\}$ and therefore $n(\lshft\Delta)\ge n(\Delta)$.
Similarly, if $f$ is a matching from $\tilde{X}_{\Delta; \m}$ to $\tilde{Y}_{\Delta; \m}$ then it maps $\{i:\Delta_i=\lshft\Delta\}$
to $\{i:\Delta_i=\Delta\}$ and therefore $n(\lshft\Delta)\le n(\Delta)$.
\end{proof}

\subsection{Ladder representations}
Finally, we consider a class of representations which generalize the Speh representations.
This class was studied in \cites{MR3163355, MR2996769, 1409.1367}.
It is important not only in its own right but also since it occurs naturally in the study of discrete series representations of classical groups
\cites{MR1896238, MR3020114}.

\begin{defi}
A rigid representation $\pi =\zele{\Delta_1 + \dots +\Delta_N} \in \Irr$ is called a \emph{ladder representation} if
$\Delta_i\not\subset\Delta_j$ for any $i\ne j$.
Thus, if $\Delta_i=[c_i,d_i]_\rho$ with $\Delta_1\ge_b\dots\ge_b\Delta_N$ then $c_1>\dots>c_N$ and $d_1>\dots>d_N$.
We will also say that $\Delta_1 + \dots +\Delta_N$ itself is a ladder in this case.
\end{defi}

For a ladder representation $\pi=\zele{\m}$ the left and right {\ordered} forms of $\m$ coincide -- we will simply
refer to it as the {\ordered} form of $\m$.

Many properties which hold for Speh representations also hold more generally for ladder representations.
(One notable exception is that ladder representations are generally not saturated  by Proposition \ref{prop:saturated}.)
As an example, we will show in this section that every ladder representation is \RI.

\emph{For the rest of the section let $\pi=\zele{\m}$ be a ladder representation} and let $\Delta$
be the largest segment in $\m$ with respect to $\ge_b$ (or equivalently, $\ge_e$).

The key property that we use is the following.
\begin{lem} \label{lem: basicladder}
Let $\n\in\MS$ be such that $e(\Delta')=e(\Delta)$ and $\Delta'\subset\Delta$ for all $\Delta'$ in $\n$.
Then $\pi\times\zele{\n}$ is irreducible.
\end{lem}

\begin{proof}
This follows from Corollary \ref{coro:sum}. Indeed for any $\Delta'$ in $\n$ the condition $\llcond(\Delta',\m)$ is trivial
while for $\rrcond(\Delta',\m)$ we can take the matching function $f(j)=j-1$ in the {\ordered} form of $\m$.
\end{proof}

\begin{prop} \label{prop: ladder is good}
Every ladder representation is \RI.

Moreover, suppose that $\sigma=\zele{\n}\in\Irr$ is such that $\supp(\sigma)\subset\Z_{\le e(\Delta)}$.
Let $\pi'=\zele{\m-\Delta}$, $\sigma'=\zele{\n_{<_e\Delta}}$ and $\tau=\zele{\Delta+\n_{\ge_e\Delta}}$.
Then
\[
\soc(\pi\times\sigma)=\soc(\tau\times\Sigma)
\]
where $\Sigma=\soc(\pi'\times\sigma')$ and $\lcond(\tau,\Sigma)$ holds.
In particular, $\lcond(\pi,\sigma)$ if and only if $\lcond(\pi',\sigma')$.
\end{prop}

\begin{proof}
We will prove the proposition by induction on the size of $\m$.
The base of the induction is the trivial case $\m=0$.
For the induction step let $\sigma=\zele{\n}$, $\pi'$, $\sigma'$ and $\tau$ be as in the statement of the proposition.
It follows from Lemma \ref{lem: basicladder} that $\pi\times\zele{\n_{\ge_e\Delta}}$ is irreducible. Now
\begin{multline*}
\pi\times\sigma\hookrightarrow\pi\times\zele{\n_{\ge_e\Delta}}\times\sigma'\simeq
\zele{\n_{\ge_e\Delta}}\times\pi\times\sigma'\\\hookrightarrow
\zele{\n_{\ge_e\Delta}}\times\zele{\Delta}\times\pi'\times\sigma'=\tau\times\pi'\times\sigma'.
\end{multline*}
By induction hypothesis, $\pi'\times\sigma'$ is \EI. Let $\Sigma=\soc(\pi'\times\sigma')$.
By Lemma \ref{lem: indEI} (with $\pi=\tau$ and $\sigma=\pi'\times\sigma'$)
we conclude that $\tau\times\pi'\times\sigma'$, and a fortiori $\pi\times\sigma$ is \EI\
and the condition $\lcond(\tau,\Sigma)$ is satisfied.
By Lemma \ref{lem: suffuniquesub} it follows that $\pi$ is \RI.
Moreover, $\soc(\pi\times\sigma)=\soc(\tau\times\Sigma)$.
The last part follows.
\end{proof}

Combining Proposition \ref{prop: ladder is good} with Lemma \ref{lem: eff stages} we can
explicate $\soc(\pi\times\sigma)$ combinatorially by the following recursive algorithm:

\begin{coro} \label{cor: laddercombsoc}
Let $\n\in\MS$ and define $\n'=(\n_{\le_e e(\Delta)})^t$. Then
\[
\combsoczele(\m,\n^t)^t=
\n_{>_e e(\Delta)}+(\Delta+\n'_{\ge_e\Delta}+\combsoczele(\m-\Delta,\n'_{<_e\Delta}))^t.
\]
In particular, (by Lemma \ref{lem: sym}) $\pi\times\sigma$ is irreducible if and only if $\pi\times\sigma=\sigma\times\pi$.
\end{coro}

Corollary \ref{cor: laddercombsoc}, together with the M\oe glin--Waldspurger algorithm give a perfectly feasible algorithm for deciding the irreducibility of
$\pi\times\sigma$ (or even, $\lcond(\pi,\sigma)$) for any $\sigma\in\Irr$ in terms of the Zelevinsky classification.

Moreover, as in the case of Speh representations we have:

\begin{lem} \label{coro: irredladder}
Let $\pi_1,\dots,\pi_k\in\Irr$. Suppose that all but at most one of the $\pi_i$'s are ladder representations.
Then $\pi_1\times\dots\times\pi_k$ is irreducible if and only if $\pi_i\times\pi_j$ is irreducible for all $i<j$.
\end{lem}

\begin{proof}
The proof is exactly as that of Corollary \ref{coro: irredspeh} using Corollary \ref{cor: laddercombsoc}.
\end{proof}

Thus, in principle we get a combinatorial condition for the irreducibility of $\pi\times\sigma$ when $\pi$ is an irreducible
product of ladder representations and $\sigma\in\Irr$. (We do not know how to easily characterize combinatorially the property
of being an irreducible products of ladder representations.)

Going back to the case where $\pi$ is a ladder representation, we may ask whether there is a simpler (or at least more concise)
combinatorial criterion for $\llcond(\pi,\sigma)$ as in the case where $\pi=\zele{\Delta}$ for $\Delta\in\seg$.

Suppose that $\pi=\zele{\m}$ is a ladder and $\sigma=\zele{\n}\in\Irr$.
Let $(\Delta_1,\dots,\Delta_N)$ be the {\ordered} form of $\m$.
Write $\n=\Delta'_1+\dots+\Delta'_{N'}$.
Let
\[
X_{\m;\n}=\{(i,j)\in\{1,\dots,N\}\times\{1,\dots,N'\}:\Delta_i\prec\Delta'_j\},
\]
and
\[
Y_{\m;\n}=\{(i,j)\in\{1,\dots,N\}\times\{1,\dots,N'\}:\lshft\Delta_i\prec\Delta'_j\}.
\]
We define relations $\rltn=\rltn_{\m;\n}$ (resp., $\rltn'$) between $Y_{\m;\n}$ and $X_{\m;\n}$ by
\[
(i_2,j_2)\rltn (i_1,j_1)\iff
\text{either ($i_1=i_2$ and $\Delta'_{j_2}\prec\Delta'_{j_1}$) or ($j_1=j_2$ and $\Delta_{i_1}\prec\Delta_{i_2}$)}.
\]
(resp.,
\[
(i_2,j_2)\rltn' (i_1,j_1)\iff (i_2,j_2)\rltn (i_1,j_1)\text{ and }i_1\le i_2+1.)
\]
Let $\llcond(\m,\n)$ (resp., $\llcond'(\m,\n)$) be the condition that there exists a $\rltn$-matching (resp., $\rltn'$-matching) function from $X_{\m;\n}$ to $Y_{\m;\n}$.
Recall that this means that there exists an injective function $f:X_{\m;\n}\rightarrow Y_{\m;\n}$ such that
$f(x)\rltn x$ for all $x\in X_{\m;\n}$.

The condition $\llcond(\m,\n)$ makes sense for any $\m$ (not necessarily a ladder).

\begin{lem} \label{lem: llcond'}
Suppose that $\m$ is a ladder and $\n$ is any.
Then the conditions $\llcond(\m,\n)$ and $\llcond'(\m,\n)$ are equivalent.
\end{lem}

\begin{proof}
Let $X=X_{\m;\n}$ and $Y=Y_{\m;\n}$.
Suppose that $f$ and $\tilde f$ are $\rltn$-matchings from $X$ to $Y$.
We will say that $\tilde f$ is an \emph{improvement} of $f$ if there exists $x=(i,j)\in X$ such that
\begin{enumerate}
\item $\tilde f(x)\rltn'x$,
\item $f(x)\not\rltn'x$,
\item $\tilde f(x')=f(x')\rltn'x'$ for all $x'=(i',j')$ such that either $j'<j$ or $j'=j$ and $i'>i$.
\end{enumerate}
To prove the lemma we will show that if $f$ is a $\rltn$-matching which is not a $\rltn'$-matching
then there exists an improvement of $f$. Let $x_0=(i_0,j_0)\in X$ be such that $f(x_0)\not\rltn'x_0$
with $j_0$ minimal and $i_0$ maximal (with respect to this $j_0$).
By assumption we have $f(x_0)=(i_0-k,j_0)$ where $k>1$.

Let $y_0=(i_0-1,j_0)\in Y$.
We will construct an integer $0\le m\le k$ and indices $j_1,\dots,j_m$ with the following properties.
\begin{enumerate}
\item $x_l=(i_0-l,j_l)\in X$ for $l=1,\dots,m$.
\item $y_l=(i_0-l-1,j_l)\in Y$ for $l=1,\dots,m'$ where $m'=\min(m,k-1)$.
\item $\Delta'_{j_r}\prec\Delta'_{j_s}$ for all $0\le r<s\le m$.
\item $f(x_l)=y_{l-1}$, $l=1,\dots,m$.
\item If $m<k$ then $y_m$ is not in the image of $f$.
\end{enumerate}
Suppose that the elements $x_0,\dots,x_{l-1}$ and $y_0,\dots,y_{l-1}$ have been constructed for some $l=1,\dots,k$.
If $y_{l-1}$ is not in the image of $f$ then we take $m=l-1$ and stop.
Otherwise we take $x_l\in X$ to be the element such that $y_{l-1}=f(x_l)$.
By the minimality property of $x_0$ we necessarily have $x_l=(i_0-l,j_l)$ for some $j_l$ where $\Delta'_{j_{l-1}}\prec\Delta'_{j_l}$.
(For otherwise, $x_l=(i_0-l+s,j_{l-1})$ for some $s>0$. Since $j_{l-1}\le j_0$, we must have $s=1$ by the minimality of $x_0$,
hence $x_l=x_{l-1}$, in contradiction with the injectivity of $f$.)
Note that $\Delta'_{j_r}\prec\Delta'_{j_l}$ for all $r<l$ since
\[
b(\lshft{\Delta'_{j_l}})\le e(\Delta_{i_0-l})\le e(\Delta_{i_0-k})\le e(\Delta'_{j_0})\le e(\Delta'_{j_r}).
\]
If $l<k$ we take $y_l=(i_0-l-1,j_l)$ and continue. (Note that $y_l\in Y$ since $b(\Delta'_{j_l})\le e(\rshft{\Delta}_{i_0-l})\le
e(\Delta_{i_0-l-1})$.)
Otherwise we take $m=k$ and stop.
Note that in the case $m=k$ we have $f(x_0)\rltn x_k$.

We define $\tilde f$ by
\[
\tilde f(x)=\begin{cases}y_l&\text{if }x=x_l,\ l=0,\dots,m',\\f(x_0)&\text{if $x=x_k$ (and $m=k$),}\\
f(x)&\text{otherwise.}\end{cases}
\]
It is easy to see that $\tilde f$ is an improvement of $f$.
\end{proof}


The following is an analogue of Lemma \ref{lem: equcond} in the case of ladder representations.

\begin{prop} \label{prop: llcond by stages}
Suppose that $\m\in\MS$ is a ladder, $\rho=e(\m)$ and $\n\in\MS_\rho$ is any multisegment such that $\supp\n\not\subset\Z_{\le\rho}$.
Then the following conditions are equivalent:
\begin{enumerate}
\item $\llcond'(\m,\n)$.
\item $\llcond'(\m,\n^-)$ and $(\m+\n)^-=\m+\n^-$.
\end{enumerate}
\end{prop}

\begin{proof}
Let $J=\{j_0,\dots,j_k\}$ be as in \eqref{eq: j_0} and \eqref{eq: j_l} so that
\[
\n^-=\n+\sum_{j\in J}({\Delta'_j}^--\Delta'_j).
\]
Let $X=X_{\m;\n}$, $Y=Y_{\m;\n}$, $\tilde X=X_{\m;\n^-}$, $\tilde Y=Y_{\m;\n^-}$.
Note that $\tilde X\subset X$, $\tilde Y\subset Y$ and the complements $X_0=X\setminus\tilde X$,
$Y_0=Y\setminus\tilde Y$ are given by
\[
X_0=\{(i,j)\in X:j\in J\text{ and }e(\Delta_i)=e(\lshft{\Delta_j'})\}
\]
and
\[
Y_0=\{(i,j)\in Y:j\in J\text{ and }e(\Delta_i)=e(\Delta_j')\}.
\]
Note that $\abs{Y_0}\le\abs{X_0}$ since the map $g:Y_0\rightarrow X_0$ given by $(i,j_l)\mapsto (i,j_{l-1})$ is one-to-one.
It is easy to see that the condition $(\m+\n)^-=\m+\n^-$ is equivalent to the surjectivity of $g$.
We will show by successive approximation that
\begin{multline} \label{eq: exmatch}
\text{if there is a $\rltn'$-matching from $X$ to $Y$ then there is such a matching $f$}\\
\text{satisfying $f(X_0)\subset Y_0$.}
\end{multline}
Note that if $f$ satisfies the latter condition then $g$ is onto, $f(X_0)=Y_0$ and $f(\tilde X)=\tilde Y$,
so that $f$ defines a $\rltn'$-matching from $\tilde X$ to $\tilde Y$.
Conversely, if $g$ is onto and $\llcond'(\m,\n^-)$ then we can define a matching
from $X$ to $Y$ by combining $g^{-1}$ with a matching from $\tilde X$ to $\tilde Y$. Thus, it remains to prove \eqref{eq: exmatch}.

Given two $\rltn'$-matchings $f$ and $\tilde f$ from $X$ to $Y$ we say that $\tilde f$ is an improvement of
$f$ if there exists $x_0=(i_0,j_0)\in X_0$ such that
\begin{enumerate}
\item $f(x_0)\notin Y_0$,
\item $\tilde f(x_0)\in Y_0$,
\item For all $x=(i,j)\in X_0$ with $i<i_0$ we have $f(x)=\tilde f(x)\in Y_0$.
\end{enumerate}
In order to show \eqref{eq: exmatch} it suffices to show that if $f$ is a $\rltn'$-matching such that $f(X_0)\not\subset Y_0$ then there exists
an improvement of $f$.

We will construct such an improvement as follows.
Let $x_0=(i_0,j_0)\in X_0$ be such that $y_0=f(x_0)\notin Y_0$ and $i_0$ is minimal with this property.
Then necessarily $y_0=(i_0,j_1)$ for some $j_1$
where $e(\Delta'_{j_1})=e(\Delta_{i_0})$ and $\Delta'_{j_1}\prec\Delta'_{j_0}$.
Therefore, there exists a unique $j_2\in J$ such that $e(\Delta'_{j_2})=e(\Delta'_{j_1})$.
We have $\Delta'_{j_2}\prec\Delta'_{j_0}$ and $\Delta'_{j_2}\subset\Delta'_{j_1}$.
Thus, $y_1=(i_0,j_2)\in Y_0$ and $y_1\rltn'x_0$. If $y_1$ is not in the image of $f$ we define
\[
\tilde f(x)=\begin{cases}y_1&\text{if }x=x_0,\\f(x)&\text{otherwise.}\end{cases}
\]
Otherwise, $y_1=f(x_1)$ for some $x_0\ne x_1\in X$. There are two possibilities: either
$x_1=(i_0,j_3)$ where $\Delta'_{j_2}\prec\Delta'_{j_3}$ or $x_1=(i_0+1,j_2)$.
In the first case, $y_0\rltn' x_1$ and we define $\tilde f$ by
\[
\tilde f(x)=\begin{cases}y_1&\text{if }x=x_0,\\y_0&\text{if }x=x_1,\\f(x)&\text{otherwise.}\end{cases}
\]
In the second case, let $x_2=(i_0+1,j_1)$ and observe that $x_2\in X$, $y_0\rltn' x_2$ and $x_2\ne x_0,x_1$. Let $y_2=f(x_2)\ne y_0,y_1$.
Then necessarily $y_2=(i_0+1,j_4)$ (since $y_2\ne y_0$) for some $j_4$ where $\Delta'_{j_4}\prec\Delta'_{j_1}$. Note that $\Delta'_{j_4}\prec\Delta'_{j_2}$
since $b(\Delta'_{j_2})\le e(\rshft{\Delta}_{i_0+1})\le e(\rshft{\Delta'}_{j_4})$. Hence, $y_2\rltn' x_1$, and we can define
\[
\tilde f(x)=\begin{cases}y_1&\text{if }x=x_0,\\y_2&\text{if }x=x_1,\\y_0&\text{if }x=x_2,\\
f(x)&\text{otherwise.}\end{cases}
\]
In all cases $\tilde f$ improves $f$.
\end{proof}

Combining Corollary \ref{cor: laddercombsoc} with Proposition \ref{prop: llcond by stages} and Lemmas \ref{lem: equcond}
and \ref{lem: llcond'} we conclude:

\begin{prop} \label{prop: laddercond}
For any ladder $\m\in\MS$ and any $\n\in\MS$ the conditions $\llcond(\m,\n)$ and $\lcond(\zele{\m},\zele{\n})$ are equivalent.
Similarly $\llcond(\n,\m)$ and $\rcond(\zele{\m},\zele{\n})$ are equivalent.
Thus, $\zele{\m}\times\zele{\n}$ is irreducible if and only if $\llcond(\m,\n)$ and $\llcond(\n,\m)$.
\end{prop}

\begin{proof}
We use induction on the size of $\m$.
For the induction step, suppose first that $\supp\n\subset\Z_{\le e(\m)}$.
We observe that in the notation of Corollary \ref{cor: laddercombsoc}, $\llcond(\m,\n)$ is equivalent to $\llcond(\m-\Delta,\n_{<_e\Delta})$.
Indeed, let $X=X_{\m;\n}$, $X'=X_{\m-\Delta;\n_{<_e\Delta}}$, viewed as a subset of $X$, and $X_0=X\setminus X'$.
Similarly, let $Y=Y_{\m;\n}$, $Y'=Y_{\m-\Delta;\n_{<_e\Delta}}$ and $Y_0=Y\setminus Y'$.
It is easy to see that if $y\in Y_0$ and $y\rltn x\in X$ then necessarily $x\in X_0$.
Thus, any matching from $X$ to $Y$ restricts to a matching from $X'$ to $Y'$.
Conversely, if $x=(i,j)\in X_0$ then $(i-1,j)\in Y_0$ and $(i-1,j)\rltn x$.
Thus, we can extend any matching from $X'$ to $Y'$ to a matching from $X$ to $Y$
by taking $(i,j)\mapsto (i-1,j)$ on $X_0$.

Thus, in this case the induction step follows from Corollary \ref{cor: laddercombsoc}.
In the case $\supp\n\not\subset\Z_{\le e(\m)}$ the induction step follows from Proposition \ref{prop: llcond by stages} and
Lemma \ref{lem: equcond} (taking into account Lemma \ref{lem: llcond'}).

The dual statement follows by passing to contragredient.
\end{proof}


In the case where both $\m$ and $\n$ are ladders it is possible to simplify the condition $\llcond(\m,\n)$.
Suppose that $(\Delta_1,\dots,\Delta_N)$ and $(\Delta'_1,\dots,\Delta'_{N'})$ are the {\ordered} forms of $\m$ and $\n$ respectively,
and let $X=X_{\m;\n}$ and $Y=Y_{\m;\n}$.
We denote by $\nlcond(\m,\n)$ (resp., $\nlcond'(\m,\n)$) the condition that there exist indices $1\le i\le N$, $1\le j\le N'$ and $k\ge0$ such that
\begin{enumerate}
\item $(i+l,j+l)\in X$ for all $l=0,\dots,k$,
\item $(i-1,j),(i+k,j+k+1)\notin Y$,
\item (resp., and in addition, $(i+l,j+l+1)\in Y\setminus X$ for all $l=0,\dots,k-1$).
\end{enumerate}

\begin{lem} \label{lem: twoladdercomb}
In the notation above, the following conditions are equivalent for two ladders $\m$, $\n$.
\begin{enumerate}
\item \label{part: neglcond} The negation of $\llcond(\m,\n)$.
\item \label{part: neglcond''} There does not exists a $\rltn''$-matching function from $X$ to $Y$ where
\[
(i_2,j_2)\rltn'' (i_1,j_1)\iff (i_2,j_2)\rltn' (i_1,j_1)\text{ and }j_2\le j_1+1.
\]
\item \label{part: nlcond} $\nlcond(\m,\n)$.
\item \label{part: nlcond'} $\nlcond'(\m,\n)$.
\end{enumerate}
\end{lem}

\begin{proof}

We first note that if $(i,j)\in X$ and $(i',j)\in Y$ then $(i',j)\rltn(i,j)$ if and only if $i'<i$.
Similarly if $(i,j')\in Y$ then $(i,j')\rltn(i,j)$ if and only if $j'>j$.

We also note that the extra condition $(i+l,j+l+1)\in Y\setminus X$ in $\nlcond'(\m,\n)$ implies that
either $b(\Delta_{i+l})=b(\Delta'_{j+l+1})$ or $e(\Delta_{i+l})=e(\Delta'_{j+l+1})$.

Obviously \ref{part: neglcond}$\implies$\ref{part: neglcond''}.

\ref{part: neglcond''}$\implies$\ref{part: nlcond}:
Suppose that $\nlcond(\m,\n)$ is not satisfied. Then it is easy to see that
\[
f(i,j)=\begin{cases}(i-1,j)&\begin{aligned}\text{if there exists $k\ge0$ such that $(i+l,j+l)\in X$}\\
\text{for all $1\le l\le k$ and }(i+k,j+k+1)\notin Y,\end{aligned}\\
(i,j+1)&\text{otherwise,}\end{cases}
\]
defines a $\rltn''$-matching from $X$ to $Y$.

\ref{part: nlcond}$\implies$\ref{part: nlcond'}:
We take a triplet $(i,j,k)$ as in the condition $\nlcond(\m,\n)$, with $k$ minimal.
We claim that $(i,j,k)$ satisfies the condition pertaining to $\nlcond'(\m,\n)$.
First, it clear from the minimality of $k$ that $(i+l,j+l+1)\in Y$ for all $l=0,\dots,k-1$.
We also have $(i-1,j+1),(i+k-1,j+k+1)\notin Y$ and if
$(i+l-1,j+l+1)\in Y$ for some $1\le l<k$ then both $(i+l,j+l+1)$ and $(i+l-1,j+l)$ are in $X$.
Thus, if $(i+l,j+l+1)\in X$ for some $l=0,\dots,k-1$ and $l$ is minimal with this property,
then $(i+l,j+l+1,k')$ would also satisfy the defining property of $\nlcond(\m,\n)$ for some $k'<k-l$,
in contradiction with the minimality of $k$.

\ref{part: nlcond'}$\implies$\ref{part: neglcond}:
Suppose that $\nlcond'(\m,\n)$ is satisfied with indices $i,j,k$ as above. Then
\begin{multline*}
\{(i',j')\in Y:(i',j')\rltn(i+l,j+l)\text{ for some }0\le l\le k\}\subset\\\{(i+l-1,j+l):l=1,\dots,k\},
\end{multline*}
so that a $\rltn$-matching is impossible.
\end{proof}

We remark that in the special case where both $\zele{\m}$ and $\zele{\n}$ are ladders satisfying $b(\Delta_{i+1})=b(\lshft{\Delta_i})$ for all $i=1,\dots,N-1$
and $b(\Delta'_{j+1})=b(\lshft{\Delta'_j})$ for all $j=1,\dots,N'-1$, the irreducibility criterion of Proposition \ref{prop: laddercond}
can be simplified even further in accordance with the results of \cite{MR1985725}. We omit the details.



\begin{rem}
We do not expect the property $\llcond(\m,\n)$ to have any representation-theoretic significance if neither $\m$ nor $\n$ is a ladder.
As a simple example, if $\rho\in\Cusp$, the multisegments $\m=\{\rho\}+\{\rho\}+\{\rshft{\rho}\}$ and
$\n=\{\rho\}+\{\rshft{\rho}\}+\{\rshft{\rho}\}$ satisfy $\llcond(\m,\n)$ although $\lcond(\zele{\m},\zele{\n})$
does not hold: we have
\begin{multline*}
\soc(\zele{\m}\times\zele{\n})=
\zele{\{\rho\}+\{\rho\}+\{\rshft{\rho}\}+\{\rshft{\rho}\}+(\rho,\rshft{\rho})}\simeq\\
\zele{\{\rho\}+\{\rshft{\rho}\}}\times\zele{\{\rho\}+\{\rshft{\rho}\}}\times\zele{(\rho,\rshft{\rho})}.
\end{multline*}
\end{rem}

\section{A simple proof of Tadi\'c hypothesis (U0)} \label{sec: U0}

The goal of this section is to explain how the results of \S\ref{sec: Speh} simplify the proof of
the classification theorem of the unitary dual of $G_m$.
The classification of the unitary dual of $\GL_n(\F)$ was established by Tadi\'c \cite{MR870688} in the 1980's.
(Significant progress on this problem was obtained earlier by Bernstein in \cite{MR748505}. Ideas of Rodier were also influential.)
Later on, under three hypotheses, named (U0), (U1) and (U2), Tadi\'c gave the classification of the unitary dual of $\GL_m(D)$, see \cite{MR1040995}.
(Actually, Tadi\'c confined himself to characteristic zero since at the time the Jacquet--Langlands correspondence was only available
in this case. However, once this was established in positive characteristic by Badulescu, there is no change in the strategy --
see \cite{MR2732351} and the discussion below.)
In \S\ref{par:3hyp}, we explain Tadi\'c's theorem and the hypotheses (U0), (U1), (U2).
Then, in \S\ref{par:previous} we recall how these hypotheses were proved. Finally, in \S\ref{par:new} we give a new proof of a weaker version of hypothesis (U0).
This leads to a purely combinatorial and uniform proof of the classification of the unitary dual of $\GL_m(D)$, and hence the hypothesis (U0) in general.

\subsection{Hypotheses (U0), (U1) and (U2)} \label{par:3hyp}
We denote by $\Irr^u$ the subset of $\Irr$ consisting of unitarizable representations and set $\Cusp^u=\Cusp\cap \Irr^u$.
(Recall that $\rho\in\Cusp$ is in $\Cusp^u$ if and only if its central character is unitary.)
For $\rho\in \Cusp$, $d\ge1$ and $n\ge0$, let $\m_{n,d}^{(\rho)}$ be the multisegment
\[
\m_{n,d}^{(\rho)}=\sum_{i=1}^n\left[\frac{n+1}2-i-\frac{d-1}{2},\frac{n+1}2-i+\frac{d-1}{2}\right]_\rho
\]
and set $\speh^{(\rho)}_{n,d}=\lang{\m_{n,d}^{(\rho)}}\in\Irr$.
We also define
\begin{align*}
B_{\rigid}&=\{\speh^{(\rho)}_{n,d}; n,d \in \Z_{\geq 1}, \rho \in\Cusp^u\},
\\
B&=B_{\rigid} \sqcup \{ \sigma\nu_\rho^\alpha  \times \sigma\nu_\rho^{- \alpha}; \sigma =\speh^{(\rho)}_{n,d} \in B_{\rigid}, 0 < \alpha < 1/2\}.
\end{align*}
Thus, $B_{\rigid}$ consists of the Speh representations with unitary central character (see \S\ref{sub:speh}).
Note that $B\subset\Irr$ --- this follows immediately from Proposition \ref{useful_properties} part \ref{part: pwnp}.

In the case $D=F$ Tadi\'c gave the following classification of the unitary dual of $\GL(m,D)$.
\begin{thm}  \cite{MR870688} \label{thm: TadicGLn}
Suppose that $D=F$. Then the map $\NN(B)\longrightarrow\Irr^u$ given by
\begin{equation}\label{Tadicmap}
\tau_1+ \dots + \tau_s \mapsto \tau_1\times \dots \times \tau_s
\end{equation}
is a bijection (and in particular well defined).
\end{thm}

Tadi\'c further conjectured that Theorem \ref{thm: TadicGLn} holds in fact for any $D$ \cite{MR1040995}.
To that end he postulated the following hypotheses:
\begin{enumerate}
\item[(U0) ] If $\pi, \pi' \in \Irr^u$, then $\pi \times \pi' $ is irreducible.
\item[(U1) ] $B_{\rigid}\subset \Irr^u$.
\item[(U2) ] $B\setminus B_{\rigid}\subset \Irr^u$.
\end{enumerate}

More precisely, he showed:
\begin{thm}\label{Tadic1}(\cite{MR1040995}*{\S6})
Suppose that the hypotheses (U0), (U1), (U2) are true. Then the map \eqref{Tadicmap} is bijective.
\end{thm}
Once again, the assumption of characteristic zero in \cite{MR1040995} can be lifted as explained in \cite{MR2732351}.

\subsection{Previous proofs of (U0), (U1) and (U2)}\label{par:previous}
Let us make a few remarks about the previous proofs of hypotheses (U0), (U1), (U2).

In the case $F=D$ the assertion (U0) is a consequence of well-known theorem of Bernstein \cite{MR748505} as was noticed by Jacquet
(cf. \cite{MR0499005}).
However, Bernstein's method is not applicable to the case $F\ne D$ since it relies crucially on the properties of the mirabolic subgroup.
In the case $F \neq D$, the assertion (U0) was proved by S\'echerre in \cite{MR2492994}:\footnote{As explained in \cite{MR2732351}
the proof works in any characteristic.}
using the theory of types of Bushnell--Kutzko \cite{MR1204652} for $\GL_n(\F)$ and their generalization \cites{MR2081220, MR2188448, MR2216835,MR2427423} to $\GL_m(\D)$,
together with some results of Barbasch and Moy \cites{MR1186959, MR1010153} on unitarity, he was able to transfer Bernstein's result to $\GL_m(\D)$.

Tadi\'c remarked in \cite{MR1181278} that (U1) and (U2) are a consequence of (U0).
The argument was written in detail in \cite{MR2732351}, thus completing the proof of the classification of the unitary dual.
We will recall the argument, replacing (U0) with the results of \S\ref{sec: nece}. (See also \cite{MR3269346}.)
We also remark that in characteristic $0$, Badulescu and Renard proved (U1) without assuming (U0) using global methods \cite{MR2055385}.

First, assuming (U1), the property (U2) is an immediate consequence of the irreducibility of $\pi\times\pi$ for any $\pi\in B_{\rigid}$
(Proposition \ref{prop:Speh is saturated} or Corollary \ref{coro: irredtwospeh}) and the elementary fact
that complementary series extend until the first point of reducibility.
More precisely, let $\pi \in \Irr^u$ and let $\alpha$ be the smallest real number $\geq 0$ such that $\pi\abs{\cdot}^\alpha \times \pi\abs{\cdot}^{-\alpha}$ is reducible.
Then $\pi\abs{\cdot}^\beta  \times\pi \abs{\cdot}^{-\beta}$ is unitarizable for all $0\le\beta<\alpha$.

Moreover, $\socc(\pi\abs{\cdot}^\alpha \times \pi\abs{\cdot}^{-\alpha})$ is also unitarizable\footnote{In fact, by a result of Milici\'c all composition factors of
$\pi\abs{\cdot}^\alpha \times \pi\abs{\cdot}^{-\alpha}$ are unitarizable -- see \cite{MR1181278}*{\S 3c}, but we will not need to use this.}
and this can be used to give a nice inductive proof of (U1) as follows. (See \cite{MR2732351}*{Lemme 4.2}.)

One proves that $\speh^{(\rho)}_{k,d}$ is unitarizable for any $\rho\in\Cusp^u$ by induction on $k$.
For $k=1$ the representation $\speh^{(\rho)}_{1,d}=\lang{[(1-d)/2,(d-1)/2]_\rho}$ is square-integrable and hence unitarizable.
For the induction step, assume that $\pi=\speh^{(\rho)}_{k,d}$ is unitarizable.
We recall that $\pi\nu_\rho^\alpha\times\pi\nu_\rho^{-\alpha}$ is irreducible for $0\le\alpha<\frac12$.
On the other hand by (the mirror symmetry analogue of) Proposition \ref{prop:Speh is saturated} or Corollary \ref{coro: irredtwospeh}
(applied twice) we have
\[
\cos(\pi\nu_\rho^{1/2}\times\pi\nu_\rho^{-1/2})\simeq
\lang{\m_{k-1,d}^{(\rho)}+\m_{k+1,d}^{(\rho)}}\simeq
\speh^{(\rho)}_{k-1,d}\times\speh^{(\rho)}_{k+1,d}.
\]
By the above, we deduce that $\speh^{(\rho)}_{k-1,d}\times\speh^{(\rho)}_{k+1,d}$ is unitarizable.
On the other hand, it is well known (and easy to see) that if $\pi_1\times\pi_2$ is irreducible and unitarizable with $\pi_1$, $\pi_2$ hermitian
then $\pi_1$ and $\pi_2$ are unitarizable (and irreducible) -- see \cite{MR1181278}*{\S 3d}).
We deduce that $\speh^{(\rho)}_{k+1,d}$ is unitarizable, concluding the induction step of the proof of (U1).

\subsection{New proof of (U0)}\label{par:new}
In this subsection we prove that \eqref{Tadicmap} is bijective using (U1), (U2) (shown above)
and the following lemma which proves a weaker form of (U0).
Then, (U0) will be a consequence of the classification, rather than an input towards it.
We emphasize that assuming the classification of the admissible dual (\S\ref{sec: multi and classification}),
the ensuing proof of (U0) is purely combinatorial and in particular does not use types or reduction to the case $D=F$.
\begin{lem}\label{lemU'0}
If $\pi\in B$ and $\pi'\in \Irr^u$, then $\pi \times \pi'$ is irreducible.
\end{lem}

\begin{proof}
Let $\pi\in B$. By hypotheses (U1) and (U2) proved above we have $\pi\in\Irr^u$.
On the other hand, it follows from Corollary \ref{coro:neededU0} that $\pi$ is \RI.
Hence $\soc(\pi \times \pi')$ is irreducible for any $\pi'\in\Irr$.
If in addition $\pi'\in\Irr^u$ then $\pi\times\pi'$ is irreducible since it is unitarizable and in particular semisimple.
\end{proof}

Let us prove now that \eqref{Tadicmap} is bijective.
First, observe that \eqref{Tadicmap} is well-defined in view of (U1), (U2) and Lemma \ref{lemU'0}.
(The irreducibility also follows from Corollaries \ref{coro: irredspeh} and \ref{coro: irredtwospeh}.)
To prove injectivity one shows that every $\pi\in B_{\rigid}$ is a prime element in the factorial ring $\Rr$.
(This property is called (U3) and its proof does not use (U0), see \cite{MR1040995}*{\S6}.)\footnote{Using
Tadi\'c determinantal formula \cites{MR1359141, MR3163355} the argument becomes even more transparent.}
Finally, to prove surjectivity of \eqref{Tadicmap}, Tadi\'c shows, using Proposition \ref{useful_properties} part \ref{part: lmm'}
(that he calls property (U4))
that for any hermitian $\pi \in \Irr$ there exist $\sigma_1, \dots, \sigma_n,$ $\tau_1, \dots, \tau_m\in B$ such that
$\pi \times\sigma_1 \times \dots \times \sigma_n $ and $ \tau_1 \times \dots \times \tau_m$ have a composition factor in common \cite{MR1181278}*{Lemma 8.2}.

Now, by (U1) and (U2) the representations $\sigma_1 \times \dots \times \sigma_n $ and $ \tau_1 \times \dots \times \tau_m$  are unitarizable.
If moreover $\pi \in \Irr^u$ then, by Lemma \ref{lemU'0}, $\pi \times\sigma_1 \times \dots \times \sigma_n $ and $ \tau_1 \times \dots \times \tau_m$
are irreducible and hence equivalent.
By (U3), we deduce that $\pi$ is a product of some $\speh_{k,d}^{(\rho)}$'s. Since $\pi$ is hermitian, it is necessarily a product of elements in $B$.
We deduce that the map \eqref{Tadicmap} is surjective.

From Lemma \ref{lemU'0} and the bijectivity of \eqref{Tadicmap} we immediately conclude:
\begin{coro}\label{U0new}
(U0) holds.
\end{coro}



\appendix

\section{The admissible dual of $\GL_m(\D)$\\
by Erez Lapid, Alberto M\'inguez and Marko Tadi\'c} \label{sec: appendAD}
As in the body of the paper, let $D$ be a local non-archimedean division algebra with center $F$.
In this appendix we review the classification of the admissible dual of $\GL_m(D)$ in terms of multisegments.

There are two classification schemes, \textit{\`a la Zelevinsky} and \textit{\`a la Langlands}.
The Langlands classification is of course valid for any reductive group (see e.g. \cite{MR2567785}), but in the case at hand it can be
refined and made more explicit.
In the case where $D=F$ these classifications are due to Zelevinsky \cite{MR584084} (see also \cite{MR689531})
following up on his work with Bernstein \cites{MR0425030, MR0425031, MR0579172}.
There are two main tools in the work of Bernstein--Zelevinsky: restriction to the mirabolic subgroup (and more generally, the notion of derivatives)
and the geometric lemma (which they introduce). The first one is special to the case $D=F$ while the second one is general.
In the case where $D\neq F$, the explication of the Langlands classification is due to Tadi\'c \cite{MR1040995} in the case where the characteristic of $F$ is zero.
One of the ingredients in this classification is the local Jacquet--Langlands correspondence which had been proved by Deligne--Kazhdan--Vign\'eras
using the simple trace formula \cite{MR771672}.
The local Jacquet--Langlands correspondence was extended to the positive characteristic case by Badulescu \cite{MR1951441} using the method of close fields due to
Kazhdan and Deligne \cites{MR771673, MR874049}.
As explained in \cite{MR2732351}, using an additional irreducibility result \cite{MR2097163}, this gives the refined Langlands classification for any $D$
in positive characteristic as well.
The Zelevinsky classification in the case $D\ne F$ is due to M\'inguez--S\'echerre \cite{MR3049700} who consider representations
over fields of characteristic different from the residual characteristic of $F$.
They use type theory which was developed by Bushnell--Kutzko in the case $D=F$ \cite{MR1204652} and subsequently by
S\'echerre and his co-authors in the general case in a series of papers (\cites{MR2081220, MR2188448, MR2216835, MR2427423, MR2889743, MR2946230}).
All these methods ultimately reduce the classification to the case $D=F$.

The two classifications are related by the Zelevinsky involution (see \S\ref{ap:involution}).

In the following we will explain how the approach of \cite{MR3049700} can be pushed further (in the complex case) using some additional
input to give a uniform proof of the classifications in all cases.
We will avoid using the Jacquet--Langlands correspondence (and hence the simple trace formula),
the theory of close fields, the Bernstein--Zelevinsky derivatives and type theory.
In fact, all the necessary ingredients for the proof are contained in the standard text \cite{CassNotes}.
Another advantage is that we get the two classifications in one fell swoop (without need to worry about the Zelevinsky involution).

The approach is mostly combinatorial relying on the geometric lemma but at certain points we use unitarity.
The ``non-combinatorial'' input is given in \S\ref{sec: noncombi}.

We remark that it is also possible to simplify the prerequisites of the proof by Tadi\'c for the refined Langlands classification in this context \cite{MR1040995}
-- see \cite{TadRemark}, \cite{MR1758232}*{\S2}.
In particular, one can eliminate the use of the Jacquet--Langlands correspondence and make the proof uniform in $D$ and valid in any characteristic.
We prefer to use an approach which is more combinatorial and works equally well for the Zelevinsky and Langlands classifications.

In the following we will freely use the notation and terminology of \S\ref{sec: notation}.
In particular, $G_n=\GL_n(D)$.
\subsection{Fundamental irreducibility results} \label{sec: noncombi}
The point of departure is the following well-known result.

\begin{thm}\label{thm:Olshanski}
Let $\rho\in\Cusp(G_m)$. Then there exists a unique unramified character $\nu_\rho$ of $F^*$
of the form $\nu_\rho=\abs{\cdot}^{s_\rho}$ with $s_\rho\in\R_{>0}$ such that
for any $\rho'\in\Cusp$ the induced representation $\rho \times \rho'$ is
reducible if and only if $\deg(\rho)=\deg(\rho')$ and either $\rho'\simeq\rho\nu_\rho$ or $\rho\simeq\rho'\nu_\rho$.
In other words,
\begin{enumerate}
\item $\rho\times\rho'$ is irreducible if $\rho'\in\Cusp\setminus\{\rho\abs{\cdot}^s:s\in\R\}$.
\item For $s\in\R$, $\rho\times\rho\abs{\cdot}^s$ is reducible if and only if $s=\pm s_\rho$.
\end{enumerate}
\end{thm}

The first part of the Theorem follows from a general argument of Casselman (see e.g., \cite{MR2567785}*{Theorem VII.1.3}) which is valid for any reductive group
and which uses unitarity in an essential way.

In the case $D=F$, the second part of Theorem \ref{thm:Olshanski} is a classical result of Bernstein--Zelevinsky and $s_{\rho}=1$ for all $\rho\in\Cusp$
(\cite{MR0579172}*{Theorem 4.2} and \cite{MR584084}*{\S 1.11}).

In the general case, at least in characteristic $0$, the second part of Theorem \ref{thm:Olshanski} follows from general results about reductive groups due to
Harish-Chandra and Silberger, together with a specific result for the general linear group due to Olshanski.
More precisely, fix $\rho\in\Cusp(G_m)$ and consider the family $\pi_s:=\rho\abs{\cdot}^s\times\rho\abs{\cdot}^{-s}\in\Rep(G_{2m})$, $s\in\C$
with the standard intertwining operators $M(s):\pi_s\rightarrow\pi_{-s}$.
Once again by general theory, $M(s)$ is a rational function in $q^s$ and $(q^{ms}-q^{-ms})M(s)$ is holomorphic
\cites{MR610479, MR1989693}.\footnote{Cf.~\cite{MR3001800}*{\S6} for a simpler proof in the case at hand.
Note that [ibid., Lemma 17] carries over to the case $D\ne F$ without difficulty.}
In the case where the characteristic of $F$ is $0$, the second part of Theorem \ref{thm:Olshanski} follows from \cite{MR577138} and the following key result.
\begin{thm}\cite{MR0499010} \label{thm:olshnski}
$M(s)$ has a simple pole at $s=0$ and the residue is a positive scalar multiple of the identity operator.
\end{thm}
A more streamlined proof of Theorem \ref{thm:olshnski} appears in \cite{MR1748913}*{Proposition 5.1}.
Note that in [loc.~cit.] this is stated and proved only when $D=F$ is of characteristic $0$.
However, the proof applies in the general case without change --
it only relies on the orbit structure of the action of $G_m\times G_m$ on the space of $m\times m$ matrices over $D$
(and more specifically on the fact that the invertible matrices form a unique open orbit).

One can use the Jacquet--Langlands correspondence \cites{MR771672, MR1951441} to determine $s_\rho$.
In particular, $s_\rho$ is an integer dividing the degree of $D$ and coprime to $m$
(\cite{MR771672}*{Theorem B.2.b} and \cite{MR2329758}*{Lemma 2.4}).
Alternatively, Theorem \ref{thm:Olshanski} (and the information about $s_\rho$) is also proved in \cite{MR2492994}*{\S4}
(which works in any characteristic as was pointed out in \cite{MR2732351}) without the Jacquet--Langlands correspondence
but using instead the theory of types developed in \cites{MR1204652, MR2081220, MR2188448, MR2216835, MR2427423, MR2889743, MR2946230}
to reduce the problem to a question about the Hecke algebra of type $A_n$.

We will now give an alternative, elementary argument for the second part of Theorem \ref{thm:Olshanski} (using Theorem \ref{thm:olshnski}
but avoiding \cite{MR577138}, the Jacquet--Langlands correspondence, and type theory). The argument is special to the general
linear group and is valid in any characteristic. It does not give any information about $s_\rho$.


First, observe that there is \emph{some} $s>0$ for which $\pi_s$ is reducible.
This follows from Theorem \ref{thm:olshnski} by a standard argument using unitarity (cf.~\cite{MR584084}*{\S 1.11}).
Indeed, the operator $sM(s)$ defines an invariant hermitian form on $\pi_s$ for $s\in\R$ which is positive-definite at $s=0$.
If $\pi_s$ were irreducible for all $s>0$ then $\pi_s$ would be unitarizable.
However, the matrix coefficients of $\pi_s$ are unbounded on $G_{2m}$ for $s\gg1$.

Let $\nu_\rho$ be a character of the form $\nu_\rho=\abs{\cdot}^{s_\rho}$ with $s_\rho>0$
such that $\rho\nu_\rho\times\rho$ is reducible.
We will momentarily show that $\nu_\rho$ is unique.
At any rate, at this stage we can still define segments and the representation $\zele{\Delta}$, $\lang{\Delta}$ for any segment $\Delta$
(with respect to our choice of $\nu_\rho$). An easy induction computes the minimal Jacquet modules of $\zele{\Delta}$, $\lang{\Delta}$
(formulas \eqref{eq: Jaczele} and \eqref{eq: JacLang} respectively) -- see for example \cite{MR1040995}*{Proposition 2.7}.
In particular, $\lang{\Delta}$ is essentially square-integrable by Casselman's criterion.

\begin{lem}
The character $\nu_\rho$ is unique (i.e., $s_\rho$ is unique).
\end{lem}

\begin{proof}
We could have used \cite{MR577138} but we will give a more elementary and self-contained argument which is special for $G_m$.
Suppose that $\rho\nu_\rho^a\times\rho$ is reducible for some $0<a\ne1$. Without loss of generality assume that $a>1$.
(Otherwise we change $\nu_\rho$ to $\nu_\rho^a$.)
Let $\delta=\soc(\rho\nu_\rho^a\times\rho)$ so that $\rp(\delta)=\rho\nu_\rho^a\otimes\rho$
and $\delta$ is an essentially square-integrable subrepresentation.
For any $n\ge1$ let $\delta_n=\lang{[0,n]_\rho}=\soc(\pi_n)$ where $\pi_n:=\rho\nu_\rho^n\times\dots\times\rho\nu_\rho^i\times\dots\times\rho$.
We take $n=\lceil a\rceil-1\ge1$, i.e, the integer such that $n<a\le n+1$.
The representation $\rho\nu_\rho^a\times\pi_n$ is a product of distinct supercuspidal representations, and in particular it is \EI.
We can assume without loss of generality that it has a unitary central character. Consider $\sigma=\soc(\rho\nu_\rho^a\times\pi_n)\in\Irr$. Then
$\sigma=\soc(\rho\nu_\rho^a\times\delta_n)$. On the other hand, by Frobenius reciprocity there is a non-trivial map
$\delta_{n-1}\nu_\rho\times\delta\rightarrow\rho\nu_\rho^a\times\pi_n$
and hence
\begin{equation} \label{eq: inen+1}
[\sigma]\le[\delta_{n-1}\nu_\rho\times\delta].
\end{equation}

Assume first that $a=n+1$, i.e., $a$ is an integer. Then $\sigma=\delta_{n+1}$. On the other hand,
$\delta_{n-1}\nu_\rho$ and $\delta$ are square-integrable and therefore $\delta_{n-1}\nu_\rho\times\delta$
is unitary. Since it is also \EI, it is irreducible. Hence $\delta_{n-1}\nu_\rho\times\delta=\sigma=\delta_{n+1}$
which is clearly impossible.

Now assume that $a<n+1$. Let $\rp=\rp_{(m,\dots,m)}$. Then $[\rp(\sigma)]\le\sum_{i=0}^n\varrho_i$ where
\[
\varrho_i=\rho\nu_\rho^n\otimes\dots\otimes\rho\nu_\rho^{i+1}\otimes\rho\nu_\rho^a\otimes\rho\nu_\rho^i\otimes\dots\otimes\rho.
\]
(The summand $\rho\nu_\rho^n\otimes\rho\nu_\rho^{n-1}\dots\otimes\rho\otimes\rho\nu_\rho^a$ is excluded by \eqref{eq: inen+1}.)
It easily follows that $\sigma$ is square-integrable since the exponents of all the $\varrho_i$'s satisfy Casselman's criterion.
Indeed, it is enough to check this for $\varrho_0$ since all other exponents differ from it by a non-negative sum of roots.
Let $e\in\R$ be the exponent such that $\rho\nu_\rho^e$ is unitary. By assumption $(n+2)e=n(n+1)/2+a$. In particular, $e>0$.
We need to show the inequalities $n(n+1)/2-i(i+1)/2>(n-i)e$, $i=0,1,\dots,n-1$.
It is easy to see that each such inequality is equivalent to $a<\frac12ni+n+i+1$.

In particular, $\sigma$ is unitarizable and hence $\sigma^\vee\simeq\sigma^c$ where $^c$ denotes complex conjugation.
Denoting by $\bar\rp$ the minimal Jacquet module with respect to the opposite parabolic we get
\[
\bar\rp(\sigma)^c\simeq\bar\rp(\sigma^c)\simeq\bar\rp(\sigma^\vee)\simeq\rp(\sigma)^\vee
\]
where the last equivalence is Casselman's pairing. Note that $\bar\rp(\sigma)=\rp(\sigma)\circ w_0$
where $w_0$ is the permutation (on blocks) $w_0(i)=n+3-i$ of $\{1,\dots,n+2\}$. Therefore we would get
\[
w_0(\varrho_i^c)=\varrho_j^\vee
\]
for some $i,j\in\{0,\dots,n\}$. This would imply that either $\rho^\vee=\rho^c\nu_\rho^n$ or
$\rho^\vee=\rho^c\nu_\rho^a$, both inconsistent with $n<a<n+1$.
\end{proof}

In order to complete the proof of Theorem \ref{thm:Olshanski} we still have
to prove the irreducibility of $\rho\times\rho$. In fact, this is equivalent
to the existence of a pole at $s=0$ for $M(s)$ but we will give a different argument which proves a more
general result that will be needed later. We first need an auxiliary statement.

\begin{lem} \label{part: distinctsqr}
Let $\Delta_i=[-a_i,a_i]_\rho$, $i=1,\dots,k$ with $a_i\in\Z$ distinct.
Then $\lang{\Delta_1}\times\dots\times\lang{\Delta_k}$ is irreducible.
\end{lem}

\begin{proof}
Indeed, we may assume without loss of generality that $\rho$ is unitary and enumerate the $\Delta_i$'s so that
$a_1>\dots>a_k$. Thus, $b(\Delta_i),e(\Delta_i)\notin\Delta_j$ for all $i<j$.
It is easy to see that $\lang{\Delta_1}\otimes\dots\otimes\lang{\Delta_k}$ occurs with multiplicity one in
$[\rp(\lang{\Delta_1}\times\dots\times\lang{\Delta_k})]$. Therefore $\lang{\Delta_1}\times\dots\times\lang{\Delta_k}$
is \EI\ (cf.~Lemma \ref{lem:main}). Since it is unitary, it is irreducible.
\end{proof}

Finally, we show
\begin{lem} \label{part: rhotimesn}
For any $n\ge1$, $\rho^{\times n}$ is irreducible.
\end{lem}

\begin{proof}
We may assume without loss of generality that $\rho$ is unitary.
Once again we could have appealed to Harish-Chandra's commuting algebra theorem (\cite{MR544991}*{Theorem 5.5.3.2}).
Instead we will give a more elementary argument which is special to the general linear group.
For any $i=1,\dots,n$ let $\Delta_i$, $\Delta_i^{>0}$ and $\Delta_i^{<0}$  be the segments $[-i+1,i-1]_\rho$, $[1,i-1]_\rho$ and $[1-i,-1]_\rho$ respectively.
By Lemma \ref{part: distinctsqr}, $\pi:=\lang{\Delta_1}\times\dots\times\lang{\Delta_n}$ is irreducible.
On the other hand, $\lang{\Delta_i}\hookrightarrow\lang{\Delta_n^{>0}}\times\rho\times\lang{\Delta_n^{<0}}$ and therefore
\[
[\pi]\le[\lang{\Delta_1^{>0}}\times\dots\times\lang{\Delta_n^{>0}}\times\rho^{\times n}\times\lang{\Delta_1^{<0}}\times\dots\times\lang{\Delta_n^{<0}}].
\]
Since $\pi$ is irreducible, there exists an irreducible constituent $\tau$ of $\rho^{\times n}$ such that
\[
[\pi]\le[\lang{\Delta_1^{>0}}\times\dots\times\lang{\Delta_n^{>0}}\times\tau\times\lang{\Delta_1^{<0}}\times\dots\times\lang{\Delta_n^{<0}}].
\]
Let $\lambda$ be the composition $(mn(n-1)/2,mn,mn(n-1)/2)$ of $mn^2$.
If $[\rp_\lambda(\sigma)]=\sum_i\alpha_i\otimes\beta_i\otimes\gamma_i$ then we write
\[
[\rp_\lambda^{(+,0,-)}(\sigma)]=\sum_{i:\supp(\alpha_i)\subset\{\rho\nu_\rho^s:s>0\},\supp(\beta_i)=\{\rho\},\supp(\gamma_i)\subset\{\rho\nu_\rho^s:s<0\}}
\alpha_i\otimes\beta_i\otimes\gamma_i.
\]
Then,
\begin{multline*}
[\rp_\lambda^{(+,0,-)}(\lang{\Delta_1^{>0}}\times\dots\times\lang{\Delta_n^{>0}}\times\tau\times\lang{\Delta_1^{<0}}\times\dots\times\lang{\Delta_n^{<0}})\\=
[\lang{\Delta_1^{>0}}\times\dots\times\lang{\Delta_n^{>0}}\otimes\tau\otimes\lang{\Delta_1^{<0}}\times\dots\times\lang{\Delta_n^{<0}}].
\end{multline*}
On the other hand,
\[
[\rp_\lambda^{(+,0,-)}(\pi)]=[\lang{\Delta_1^{>0}}\times\dots\times\lang{\Delta_n^{>0}}
\otimes\rho^{\times n}\otimes\lang{\Delta_1^{<0}}\times\dots\times\lang{\Delta_n^{<0}}].
\]
We infer that $\tau=\rho^{\times n}$ proving our assertion.
\end{proof}

This completes the alternative proof of Theorem \ref{thm:Olshanski}.

Note that clearly, $\nu_{\rho\chi}=\nu_\rho$ for any $\rho\in\Cusp$ and character $\chi$ of $F^*$.

We write $\rshft{\rho}=\rho\nu_\rho$ and $\lshft{\rho}=\rho\nu_\rho^{-1}$.

Finally, we need another irreducibility result:
\begin{lem} \label{lem: ind21}
For any $\rho\in\Cusp$ the induced representations $\zele{[0,1]_\rho} \times \rho$, $\zele{[0,1]_\rho} \times\rshft{\rho}$,
$\lang{[0,1]_\rho}\times\rho$ and $\lang{[0,1]_\rho}\times\rshft{\rho}$ are irreducible.
\end{lem}

\begin{proof}
By passing to the contragredient it is enough to prove the irreducibility of $\pi_1:=\rho\times\lang{[0,1]_\rho}$ and
$\pi_2:=\rho\times\zele{[0,1]_\rho}$.
Suppose first that $\pi_1$ is irreducible and assume on the contrary that $\pi_2$ is reducible.
If $m$ denotes the degree of $\rho$, then
\[
[\rp_{(2m,m)}(\pi_2)]=\zele{[0,1]_\rho}\otimes\rho+\rho\times\rho\otimes\rshft{\rho}.
\]
Thus, $\rp_{(2m,m)}(\pi_2)$ is of length two and hence $\pi_2$ is also of length two
and one of its constituents $\sigma_2$ satisfies $[\rp_{(2m,m)}(\sigma_2)]=\zele{[0,1]_\rho}\otimes\rho$
and hence $[\rp_{(m,m,m)}(\sigma_2)]=[\rho\otimes\rshft{\rho}\otimes\rho]$.
On the other hand
\[
[\rp_{(m,2m)}(\pi_2)]=\rho\otimes (\rho\times\rshft{\rho}+\zele{[0,1]_\rho}).
\]
and therefore
\[
\rp_{(m,2m)}(\sigma_2)\simeq\rho\otimes\lang{[0,1]_\rho}.
\]
We deduce that $\sigma_2$ is a subrepresentation of $\pi_1$ which contradicts its irreducibility.

Finally, we will show that $\pi_1$ is irreducible. We argue as before.
We have
\[
\rho\times\lang{[-1,1]_\rho}\hookrightarrow\pi_1\times\lshft{\rho}
\]
and the left-hand side is irreducible by Lemma \ref{part: distinctsqr}.
Thus, there exists an irreducible subquotient $\sigma$ of $\pi_1$ such that
\[
[\rho\times\lang{[-1,1]_\rho}]\le[\sigma\times\lshft{\rho}].
\]
Taking the Jacquet module $\rp_{(3m,m)}$ on both sides and comparing the part of the form $*\otimes\rshft{\rho}$ we get
\[
[\rho\times\lang{[0,1]_\rho}]\le[\sigma]
\]
so that $\pi_1$ is irreducible as required.
\end{proof}

\subsection{The classification} \label{sec: classification}
In the following we will explain the classification approach along the lines of \cite{MR3049700}, but avoiding the use of types.

A key step is the following.

\begin{lem}(\cite{MR3178433}*{Th\'eor\`eme 7.38}) \label{lem: pairwise unlinked}
Let $\m=\Delta_1+\dots+\Delta_N$ be a \TU\ multisegment.
Then the induced representation $\std(\m)= \zele{\Delta_1} \times \dots \times  \zele{\Delta_N} $ is irreducible.
\end{lem}

For the proof, the following simple result (based on Casselman's pairing) will be needed.

\begin{lem}(\cite{MR3178433}*{Lemme 2.5}) \label{lem:subquo}
Let $\pi \in \Rep(G_n)$. Suppose there exist $\rho_1, \dots,\rho_r \in \Irr$ such that:
\begin{enumerate}
\item $\pi$ is a subrepresentation of $\rho_1 \times \dots \times \rho_r $ and a quotient of $\rho_r \times \dots \times \rho_1 $.
\item $\rho_1 \otimes \dots \otimes \rho_r $ appears with multiplicity one in $[\rp(\rho_1 \times \dots \times \rho_r )]$.
\end{enumerate}
Then $\pi$ is irreducible.
\end{lem}

\begin{proof}[Proof of Lemma \ref{lem: pairwise unlinked}]
We will follow the proof of \cite{MR3178433}*{Th\'eor\`eme 7.38}, but avoiding type theory.

The reduction to the case $N=2$ is proved exactly as in \cite{MR3178433}*{Corollaire 7.32} (using Lemma \ref{part: rhotimesn} for the minimal case).
We assume henceforth $N=2$. We prove the lemma by induction on $l(\Delta_1) +l(\Delta_2)$.
The cases $l(\Delta_1) +l(\Delta_2)\le 3$ are covered by Theorem \ref{thm:Olshanski} and Lemma \ref{lem: ind21}.
For the induction step we may therefore assume that $l(\Delta_1)+l(\Delta_2)>3$.

Suppose first that $\Delta_1=\Delta_2$. Then we have (by induction hypothesis)
\begin{multline*}
\zele{\Delta_1} \times   \zele{\Delta_2}\hookrightarrow
\zele{\Delta_1}\times b(\Delta_2)\times \zele{^-\Delta_2}\simeq
 b(\Delta_2)\times \zele{\Delta_1}\times\zele{^-\Delta_2})\\\hookrightarrow
(b(\Delta_2)\times b(\Delta_1)) \times (\zele{^-\Delta_1} \times \zele{^-\Delta_2})
\simeq (b(\Delta_1)\times b(\Delta_2)) \times (\zele{^-\Delta_1} \times \zele{^-\Delta_2}).
\end{multline*}
Similarly, $\zele{\Delta_1} \times   \zele{\Delta_2}$ is a quotient of
$ (\zele{^-\Delta_1} \times \zele{^-\Delta_2})  \times (b(\Delta_1)\times b(\Delta_2))$.
The result follows now from Lemma \ref{lem:subquo} and our induction hypothesis.

Now suppose that $\Delta_1=[-1,1]_\rho$ and $\Delta_2=[0,0]_\rho$. Then $\zele{\Delta_1} \times \rho$ is a subrepresentation of
$\lshft{\rho}\times (\zele{[0,1]_\rho} \times   \rho)$.
By Lemma \ref{lem: ind21} and Frobenius reciprocity, we deduce that $[\rp(\soc(\zele{\Delta_1} \times \rho))]$ contains
\begin{equation}\label{eq:tau}
\lshft{\rho}\otimes \rho \otimes \rho \otimes\rshft{\rho}
\end{equation}
with multiplicity two. The same is true for $\cos(\zele{\Delta_1} \times \rho)$.
As \eqref{eq:tau} appears with multiplicity two in $[\rp(\zele{\Delta_1} \times \rho)]$
we deduce that $\zele{\Delta_1}\times\rho$ is irreducible.

In all other cases, upon passing to the contragredient and exchanging $\Delta_1$ with $\Delta_2$ if necessary, we can assume that
$l(\Delta_1)>1$, $^-\Delta_1$ and $\Delta_2$ are unlinked and $b(\Delta_1) \notin\{ b(\Delta_2), \lshft{b(\Delta_2)}\}$.
Then $\zele{\Delta_1} \times   \zele{\Delta_2}$ is a subrepresentation of $b(\Delta_1) \times (\zele{^-\Delta_1} \times \zele{\Delta_2}) $
and (by induction hypothesis) a quotient of $ (\zele{^-\Delta_1} \times \zele{\Delta_2})  \times b(\Delta_1)$.
The result now follows from Lemma \ref{lem:subquo} and our induction hypothesis.
\end{proof}

We now turn to the proof of Theorem \ref{thm: classification} which is \cite{MR3049700}*{Th\'eor\`emes 1.3 and 1.4}
together with \cite{MR3049700}*{Remarque 5.5 and Remarque 5.9}.

The fact that the map $\m\mapsto\zele{\m}$ is well-defined and injective is proved
in \cite{MR3049700}*{\S5} using only Lemma \ref{lem: pairwise unlinked},  Jacquet functors and the Geometric Lemma.

The proof that any $\pi\in\Irr$ is of the form $\zele{\m}$ for some $\m$ relies in addition on the following
analog of \cite{MR584084}*{Proposition 4.6}:

\begin{lem} \cite{MR3049700}*{Lemma 5.12} \label{lem:2seg}
Let $\Delta$ and $\Delta'$ be two linked segments. Then
\[
[\zele{\Delta} \times \zele{\Delta'}]=[\zele{\Delta,\Delta'}]+[\zele{\Delta \cap\Delta'} \times \zele{\Delta \cup\Delta'}].
\]
\end{lem}

The proof in [ibid.] uses type theory to translate the result of \cite{MR584084}*{Proposition 4.6} to the case where $D \neq F$.
(This is the only usage of type theory in the proof of \cite{MR3049700}*{Th\'eor\`emes 1.3 and 1.4}.)
We give an alternative proof as follows.

\begin{proof}
We may assume without loss of generality that $\Delta'\prec\Delta$.
We prove the lemma by induction on $l= l(\Delta)+ l(\Delta')$.
In the minimal case $l=2$ the lemma is straightforward.

The first two parts of the proof of \cite{MR584084}*{Proposition 4.6} depend only on Jacquet functors and the Geometric Lemma and are therefore
valid in our setting. The first part deals with the case where $\Delta$ and $\Delta'$ are juxtaposed.
Thus, we may assume that $\Delta$ and $\Delta'$ are not juxtaposed.
The second part shows that $\zele{\Delta \cap\Delta'} \times \zele{\Delta \cup\Delta'}$ is a quotient of $\zele{\Delta} \times \zele{\Delta'}$.
(The irreducibility of $\zele{\Delta \cap\Delta'} \times \zele{\Delta \cup\Delta'}$ follows from Lemma \ref{lem: pairwise unlinked}.)
(The rest of the argument of \cite{MR584084}*{Proposition 4.6} uses derivatives so we need an alternative argument.)
Since $\zele{\Delta,\Delta'}\hookrightarrow\zele{\Delta} \times \zele{\Delta'}$ and
$\zele{\Delta,\Delta'}$ appears with multiplicity one in $[\zele{\Delta} \times \zele{\Delta'}]$
we deduce that
\[ [\zele{\Delta} \times \zele{\Delta'}] \geq [\zele{\Delta,\Delta'}]+[\zele{\Delta \cap\Delta'} \times \zele{\Delta \cup\Delta'} ].
\]
Let $\rp$ denote the Jacquet functor associated to the parabolic subgroup of type $(m, (l-1)m)$ where $m=\deg(\rho)$.
Since any irreducible subquotient $\pi$ of $\zele{\Delta} \times \zele{\Delta'}$ satisfies $\rp(\pi)\neq 0$ it suffices to prove that
\begin{equation}\label{eq:rp} [\rp(\zele{\Delta,\Delta'}) ] \geq [\rp( \zele{\Delta} \times \zele{\Delta'})] -
[\rp( \zele{\Delta \cap\Delta'} \times \zele{\Delta \cup\Delta'} )].
\end{equation}
By the Geometric Lemma and the fact that $\Delta'\prec\Delta$ but $\Delta$ and $\Delta'$ are not juxtaposed, the right-hand side of \eqref{eq:rp} equals $A +B$ where
\begin{align*}
A&=[b(\Delta)\otimes(\zele{^-\Delta} \times\zele{\Delta'})]- [b(\Delta)\otimes( \zele{^-\Delta \cap\Delta'} \times \zele{^-\Delta \cup\Delta'} )] \\
B&=[b(\Delta')\otimes(\zele{\Delta} \times\zele{^-\Delta'})]-[b(\Delta')\otimes (\zele{\Delta \cap\,^-\Delta'} \times \zele{\Delta \cup\,^-\Delta'} )].
\end{align*}
Note that $B=0$ if $a'=a-1$.
On the other hand
\begin{equation}\label{eq:brho}
\zele{\Delta ,\Delta'} \hookrightarrow \zele{\Delta} \times \zele{\Delta'} \hookrightarrow b(\Delta) \times \zele{^-\Delta} \times \zele{\Delta'}.
\end{equation}
By Lemmas \ref{lem: mult1 jf} and \ref{lem:main},  $\zele{\Delta ,\Delta'}$ is the unique irreducible subrepresentation
of $b(\Delta) \times \zele{^-\Delta} \times \zele{\Delta'}$ and
$$\zele{\Delta ,\Delta'} \hookrightarrow  b(\Delta) \times \zele{^-\Delta,\Delta'}.$$
By Frobenius reciprocity and our induction hypothesis, we deduce that $A \leq [\rp(\zele{\Delta , \Delta'} )]$.

In the same way
\begin{equation}\label{eq:brho2}
\zele{\Delta ,\Delta'} \hookrightarrow  \zele{\Delta} \times b(\Delta') \times\zele{^-\Delta'}\end{equation}
and if $a' <a-1$, the RHS of \eqref{eq:brho2} is isomorphic to $b(\Delta') \times\zele{\Delta} \times \zele{^-\Delta'}$.
We deduce as before that $B \leq [\rp(\zele{\Delta ,\Delta'} )]$.

The Lemma follows.
\end{proof}

\subsection{Proof of Proposition \ref{useful_properties}}\label{par: preuve}

Part \ref{part: support is preserved} is trivial from the definition.

Part \ref{part: pairwise unlinked} is a consequence of Lemmas \ref{lem:  pairwise unlinked} and \ref{lem:2seg}.

For part \ref{part: anyorder} we write $\{e(\Delta_1),\dots,e(\Delta_N)\}=\{\rho_1,\dots,\rho_k\}$ (as a set) where
if $\rho_i=\rho_j\abs{\cdot}^s$, $i<j$ with $s\in\R$ then necessarily $s>0$.
Let $\pi_i$ be the product of the $\zele{\Delta_j}$'s such that $e(\Delta_j)=\rho_i$.
Then $\pi_1,\dots,\pi_k\in\Irr$ and $\std(\m)=\pi_1\times\dots\times\pi_k$.

Let $\Pi=\zele{\Delta_1}\times\dots\times\zele{\Delta_N}$.
It is easy to see that $\pi_1\otimes\dots\otimes\pi_k$ occurs in $[\rp(\Pi)]$ with multiplicity one
and moreover, there is no other composition factor of $\rp(\Pi)$ in the same Bernstein component
of that Levi subgroup.
Hence there is a non-trivial map $\rp(\Pi)\rightarrow\pi_1\otimes\dots\otimes\pi_k$, which by Frobenius reciprocity
yields a non-zero map $\Pi\xrightarrow{f}\std(\m)$.
Finally, the image of $f$ contains $\soc(\std(\m))=\zele{\m}$. Therefore, $\zele{\m}$ does not occur as a subquotient of $\Ker f$
since $\std(\m)$ is \EI\ and $[\Pi]=[\std(\m)]$.

Part \ref{part: vdash} is proved in \cite{MR1040995}*{Thoerem 5.3} for the Langlands classification but the proof is valid for both classifications
(see \cite{MR3049700}*{\S 6.2.1}).

Part \ref{part: lmm'} follows from part \ref{part: vdash} as in \cite{MR1040995}*{Proposition 2.3}.

Part \ref{part: pwnp} is an easy consequence of Theorem \ref{thm: classification}.
Indeed $\zele{\m} \times \zele{\n}$ is a subrepresentation of $\std(\m)\times \std(\n)\simeq \std(\m+\n)$
whose socle is $\zele{\m+\n}$. Similarly for the other statements. (See also \cite{MR3049700}*{Proposition 5.1} for the last part.)

Finally, part \ref{part: contragredient} follows from the fact that $\tilde{\std}(\m)\simeq\std(\m^\vee)^\vee$
(cf.~\cite{MR3049700}*{Proposition 5.10}).

\begin{rem}
As was pointed out by the referee, part \ref{part: anyorder} is a special case of a more general statement:
for any $\pi_1,\dots,\pi_k\in\Irr$ and a permutation $\sigma$ of $\{1,\dots,k\}$ there exists a non-zero intertwining map
\[
\pi_1\times\dots\times\pi_k\rightarrow\pi_{\sigma(1)}\times\dots\times\pi_{\sigma(k)}
\]
Indeed, fixing $r_1\gg\dots\gg r_k$, the standard intertwining operators
\[
M(s):\pi_1\abs{\cdot}^{r_1s}\times\dots\times\pi_k\abs{\cdot}^{r_ks}\rightarrow
\pi_{\sigma(1)}\abs{\cdot}^{r_{\sigma(1)}s}\times\dots\times\pi_{\sigma(k)}\abs{\cdot}^{r_{\sigma(k)}s}
\]
are well-defined for $\Re s\gg1$ and admit a meromorphic continuation. Thus, for an appropriate $r\in\Z$,
$s^rM(s)$ gives the required intertwining operator at $s=0$.
\end{rem}

\subsection{Principle of ``mirror symmetry'' \cite{MR3178433}*{\S 7.3.1}} \label{sec: mirror symmetry}
From a combinatorial point of view the connection between the Zelevinsky and Langlands classifications
is simply that $\nu_\rho$ is replaced by $\nu_\rho^{-1}$. (Thus, we allow $s_\rho$ to be negative.)
The set of segments is unchanged under this operation but $b(\Delta)$ and $e(\Delta)$ are interchanged.
Similarly, the relation $\prec$ becomes the inverse relation, $\le_b$ becomes $\ge_e$ and vice versa.
All the proofs given or referred to in \S\ref{sec: classification},\ref{par: preuve} only use Jacquet module techniques and do not depend
on the positivity of $s_\rho$. (This is not true for Theorem \ref{thm:Olshanski}, but there we are of course free to choose $s_\rho$.
Similarly, the statement of Lemma \ref{lem: ind21} is symmetric under mirror symmetry.)
Therefore, the same proof applies to Theorem \ref{thm: classificationL}.

\begin{rem} \label{rem: tempreps}
The fact that all essentially square-integrable irreducible representations are of the form $\lang{\Delta}$ for some $\Delta\in\seg$
is a theorem of Bernstein (at least for $D=F$) which is stated without proof in \cite{MR584084}*{Theorem 9.3}.
For a proof which works over any $D$ see \cite{AuelThesis} or \cite{MR1758232}*{\S2}.
It follows immediately that the tempered irreducible representations are of the form $\lang{\m}$
where $\m$ is \TU.
We will not use these facts in any serious way.
\end{rem}

\subsection{Involution}\label{ap:involution}
An immediate consequence of the classification is that the ring $\Rr$ is freely generated by
$\zele{\Delta}$, $\Delta\in\seg$ (or alternatively by $\lang{\Delta}$, $\Delta\in\seg$)
(\cites{MR584084, MR1040995}).
Following Zelevinsky (and as in \cite{MR1040995}) define the ring homomorphism $\pi \mapsto \pi^{\zi}$ of $\Rr$,
characterized by $\zele{\Delta}^{\zi}=\lang{\Delta}$ for any segment $\Delta \in \seg$.
Zelevinsky showed that $\pi\mapsto\pi^{\zi}$ is an involution of $\Rr$ (in the case $D=F$ but the proof works for any $D$)
and conjectured (in the case $D=F$) that
\begin{equation} \label{eq: prirr}
\pi^{\zi}\in\Irr\text{ if }\pi\in\Irr.
\end{equation}
This property was proved (in a more general context) independently by Bernstein (unpublished), Aubert \cites{MR1285969, MR1390967}
and Schneider--Stuhler \cite{MR1471867}.
It is easy to see that property \eqref{eq: prirr} implies that for any multisegment $\m \in \MS$, we have $\zele{\m}^{\zi}=\lang{\m}$ and $\lang{\m}^{\zi}=\zele{\m}$
(see \cite{MR689531}*{\S 5.3} when $D=F$ and \cite{MR3049700}*{Proposition A.6} in general).
Therefore, in view of Theorem \ref{thm: classification}, Zelevinsky's involution induces an involution of $\MS$ also denoted $\m \mapsto \m^{\zi}$.
In \cite{MR863522}, in the case where $D=F$, M\oe glin and Waldspurger defined combinatorially an involution $\m \mapsto \m^\sharp$ on $\MS$
with the property that $\m^{\zi}=\m^\sharp$ (assuming \eqref{eq: prirr} which was not known at the time).
This works for any $D$ (see \cite{MR2349436} or \cite{MR2527415}*{Th\'eor\`eme 7.7}).

As pointed out in \cite{MR1040995}*{Remark 7.1}
one can turn the tables and define a $\Z$-linear map $\pi \mapsto \pi^t$ on $\Rr$ by $\zele{\m}^t=\lang{\m}$ for every multisegment $\m \in \MS$.
Denote again by $\m \mapsto \m^t$ the induced permutation on $\MS$.
(Thus, \eqref{eq: prirr} is equivalent to the property that $\pi^{\zi}=\pi^t$ for all $\pi\in\Rr$, i.e. that $\pi \mapsto \pi^t$ is a ring homomorphism of $\Rr$.)
We claim that with this definition, it follows from the results of \cite{MR863522}, without using \eqref{eq: prirr}, that $\m^t=\m^\sharp$
for all $\m\in\MS$ (and in particular $\pi \mapsto \pi^t$ is an involution).
In fact, at this stage we do not need to use the result on Hecke algebras in \cite{MR863522}*{\S I}
since we now have at our disposal the two classifications with the good properties of Proposition \ref{useful_properties}.

First note that the proof of \cite{MR863522}*{Lemme II.9} holds for any $D$ since it only uses the properties of Proposition \ref{useful_properties}.
Similarly, \cite{MR863522}*{Lemme II.10.2} is deduced from \cite{MR863522}*{Lemme II.9} (applied to $\opp{D}$)
by utilizing the functor $^\star$ of \S\ref{sec: starsym}.
We will prove that $\lang{\m}=\zele{\m^\sharp}$, i.e., that $\m^\sharp=\m^t$, by induction on the (total) length of $\m$.
The base of the induction is the trivial case $\m=0$. For the induction step we may suppose that $\m$ is rigid.
Let $(\Delta_1,\dots,\Delta_N)$ be the left {\ordered} form of $\m$ and let $\rho=b(\Delta_1)$.
Set $^-\m=\Delta_1+\dots+\,^-\Delta_i+\dots +\Delta_N$ where $i$ is the largest index such that $b(\Delta_i)=\rho$. Then
exactly as in \cite{MR863522}*{Lemme II.11} one proves that $\lang{\m}$ is a subrepresentation of $ \lang{^-\m}\times \rho$. Alternatively,
we can argue as follows.
Let $\m'=\Delta_1+\dots+\Delta_i$ and $^-\m'=\Delta_1+\dots+\Delta_{i-1}+\,^-\Delta_i$. Then
\begin{multline*}
\lang{\m'}=\lang{\Delta_1}\times\dots\times\lang{\Delta_i}\hookrightarrow\\
 \lang{\Delta_1}\times\dots\times\lang{\Delta_{i-1}}\times\lang{^-\Delta_i}\times\rho=
 \lang{^-\m'}\times\rho.
\end{multline*}
Moreover, $\lang{\m'}=\soc(\lang{^-\m'}\times\rho)$ since $\lang{^-\m'}\otimes\rho$ occurs with multiplicity $i$
in both $[\rp(\lang{\m'})]$ and $[\rp(\lang{^-\m'}\times\rho)]$. Since $\lang{\m-\m'}$ is right $\Z_{\ge\rho}$-reduced, it follows from (the symmetric analogue of)
Lemma \ref{lem: mult1 jf} and Lemma \ref{lem:main} that
\[
\soc(\lang{\m-\m'}\times\lang{^-\m'}\times\rho)=\soc(\lang{\m-\m'}\times\lang{\m'})=\lang{\m}.
\]
Hence also $\soc(\lang{^-\m}\times\rho)=\lang{\m}$ as required.

By induction hypothesis we have $\lang{^-\m}=\zele{(^-\m)^\sharp}$. Thus,
\[
\zele{\m^t}\hookrightarrow\zele{(^-\m)^\sharp}\times\rho.
\]
As explained in the proof of \cite{MR863522}*{Th\'eor\`eme 13} it follows from this relation, together with \cite{MR863522}*{Lemme II.3} (the combinatorial
backbone of [ibid.]) and \cite{MR863522}*{Lemme II.10.2} that
\begin{equation} \label{eq:tgesharp}
 \m^t\ge\m^\sharp
\end{equation}
where $\ge$ is the total order on multisegments supported in $\Z_\rho$ introduced in \cite{MR863522}*{\S II.1}.
Since \eqref{eq:tgesharp} holds for all $\m$ of length $l$ and
both $\m\mapsto\m^t$ and $\m\mapsto\m^\sharp$ are permutations on the (finite) set of multisegments of length $l$ with a given support, we infer that $\m^t=\m^\sharp$.



\def\cprime{$'$}


\bigskip

  \footnotesize{

Erez Lapid, \textsc{Department of Mathematics, Weizmann Institute of Science,
Rehovot 7610001, Israel}\par\nopagebreak
  \textit{E-mail address}: \texttt{erez.m.lapid@gmail.com}

  \medskip

Alberto M\'inguez, \textsc{Institut de Math\'ematiques de Jussieu - Paris Rive Gauche, Universit\'e Paris 6,
        4 place Jussieu, 75005, Paris, France}\par\nopagebreak
  \textit{E-mail address}: \texttt{alberto.minguez@imj-prg.fr}

}

\end{document}